\documentclass[12pt]{article}

\usepackage{times}

\usepackage{enumitem}
\usepackage[utf8]{inputenc}
\usepackage{commath}
\usepackage{amsthm, amscd}
\usepackage{amsmath}
\usepackage{amssymb}
\usepackage{cite}
\usepackage{setspace}
\usepackage{ stmaryrd }
\usepackage{tikz-cd}

\usepackage{tocloft}
\usepackage{sectsty}
\usepackage{xparse}

\newcommand*\tasklabelformat[1]{#1)}

\usepackage{titlesec}

\usepackage{tasks}[newest]
\settasks{
  label = \alph* ,
  label-format = \tasklabelformat ,
  label-width  = 12pt
}

\usepackage{bigints}
\usepackage{comment}
\usepackage{mathtools}
\usepackage{mathrsfs}
\usepackage{fancyhdr}

\allowdisplaybreaks[4]

\numberwithin{equation}{section}
\usepackage{etoolbox}
\patchcmd{\thebibliography}
  {\settowidth}
  {\setlength{\itemsep}{0pt plus -10pt}\settowidth}
  {}{}
\apptocmd{\thebibliography}
  {
  }
  {}{}
\makeatletter
\newtheorem*{rep@theorem}{\rep@title}
\newcommand{\newreptheorem}[2]{%
\newenvironment{rep#1}[1]{%
 \def\rep@title{#2 \ref{##1}}%
 \begin{rep@theorem}}%
 {\end{rep@theorem}}}
\makeatother

\theoremstyle{theorem}

\newreptheorem{theorem}{Theorem}
\newtheorem{thm}{Theorem}[section]
\newtheorem*{thm*}{Theorem}
\theoremstyle{definition}
\newtheorem{prop}[thm]{Proposition}
\newtheorem*{prop*}{Proposition}

\newtheorem{lem}[thm]{Lemma}

\newtheorem*{cor*}{Corollary}
\theoremstyle{remark}

\newtheorem{rem}[thm]{Remark}

\title{\vspace*{-1.5cm} Geometry at the infinity of the space of positive metrics:
\\
test configurations, geodesic rays and chordal distances
}

\author
{Siarhei Finski
}

\date{}

\usepackage[%
    left=1in,%
    right=1in,%
    top=1.1in,%
    bottom=0.8in,%
    paperheight=11in,%
    paperwidth=8.5in%
]{geometry}

\newcommand{\imun} {\sqrt{-1}}

\newcommand{\res}{{\rm{Res}}}

\newcommand{\comp}{\mathbb{C}}
\newcommand{\real}{\mathbb{R}}

\newcommand{\nat}{\mathbb{N}}
\newcommand{\integ}{\mathbb{Z}}

\newcommand{\enmr}[1]{\text{End}{(#1)}}

\newcommand{\ccal}{\mathscr{C}}

\newcommand{\dbar}{ \overline{\partial} }

\newcommand{\rk}[1]{{\rm{rk}} ( #1 )}
\newcommand{\tr}[1]{{\rm{Tr}} \big[ #1 \big]}

\makeatletter
\DeclareFontFamily{OMX}{MnSymbolE}{}
\DeclareSymbolFont{MnLargeSymbols}{OMX}{MnSymbolE}{m}{n}
\SetSymbolFont{MnLargeSymbols}{bold}{OMX}{MnSymbolE}{b}{n}
\DeclareFontShape{OMX}{MnSymbolE}{m}{n}{
    <-6>  MnSymbolE5
   <6-7>  MnSymbolE6
   <7-8>  MnSymbolE7
   <8-9>  MnSymbolE8
   <9-10> MnSymbolE9
  <10-12> MnSymbolE10
  <12->   MnSymbolE12
}{}
\DeclareFontShape{OMX}{MnSymbolE}{b}{n}{
    <-6>  MnSymbolE-Bold5
   <6-7>  MnSymbolE-Bold6
   <7-8>  MnSymbolE-Bold7
   <8-9>  MnSymbolE-Bold8
   <9-10> MnSymbolE-Bold9
  <10-12> MnSymbolE-Bold10
  <12->   MnSymbolE-Bold12
}{}

\let\llangle\@undefined
\let\rrangle\@undefined
\DeclareMathDelimiter{\llangle}{\mathopen}%
                     {MnLargeSymbols}{'164}{MnLargeSymbols}{'164}
\DeclareMathDelimiter{\rrangle}{\mathclose}%
                     {MnLargeSymbols}{'171}{MnLargeSymbols}{'171}
\makeatother

\newenvironment{sciabstract}{}

\cftsetindents{section}{0em}{1.7em}

\sectionfont{\large}

\titlespacing*{\section}
{0pt}{5pt}{5pt}

\begin{document}

\maketitle

\begin{sciabstract}
  \textbf{Abstract.}
  From the work of Phong and Sturm in 2007, for a polarised projective manifold and an ample test configuration, one can associate the geodesic ray of plurisubharmonic metrics on the polarising line bundle using the solution of the Monge-Ampère equation on an equivariant resolution of singularities of the test configuration.  We prove that the Mabuchi chordal distance between the geodesic rays associated with two ample test configurations coincides with the spectral distance between the associated filtrations on the section ring. 
  \par 
  This gives an algebraic description of the boundary at the infinity of the space of positive metrics, viewed — as it is usually done for spaces of negative curvature — through geodesic rays.
\end{sciabstract}

\pagestyle{fancy}
\lhead{}
\chead{Geometry at the infinity of the space of positive metrics}
\rhead{\thepage}
\cfoot{}


\newcommand{\Addresses}{{
  \bigskip
  \footnotesize
  \noindent \textsc{Siarhei Finski, CNRS-CMLS, École Polytechnique F-91128 Palaiseau Cedex, France.}\par\nopagebreak
  \noindent  \textit{E-mails }: \texttt{finski.siarhei@gmail.com} $\quad$ or  $\quad$  \texttt{siarhei.finski@polytechnique.edu}.
}} 

\vspace*{0.25cm}

\par\noindent\rule{1.25em}{0.4pt} \textbf{Table of contents} \hrulefill

\vspace*{-1.5cm}

\tableofcontents

\vspace*{-0.2cm}

\noindent \hrulefill


\section{Introduction}\label{sect_intro}
	The main goal of this article is to study the geometry at the infinity of the space of positive metrics on an ample line bundle over a given projective manifold.
	Here we view the infinity in terms of geodesic rays; this point of view goes in line with the general philosophy advocated by Donaldson \cite{DonaldSymSp} that the space of positive metrics on an ample line bundle is as an infinite-dimensional manifold of non-positive sectional curvature. In this perspective, our study here is similar to the study of Tits boundary of ${\rm{CAT}}(0)$ spaces, cf. \cite{HaefligerBrid}.
	\par 
	More precisely, let $X$ be a complex projective manifold, and let $L$ be an ample line bundle over $X$.
	We denote by $\mathcal{H}^L$ (or simply $\mathcal{H}$ for brevity) the space of positive Hermitian metrics on $L$.
	\par 
	For any $p \in [1, +\infty]$, one can introduce on $\mathcal{H}$ a collection of $L^p$-type Mabuchi metrics, see Section \ref{sect_pp_thr}.
	Using these Finsler metrics, we introduce the path length metric structures $(\mathcal{H}, d_p)$.
	By \cite{DarvasFinEnerg}, the metric completions $(\mathcal{E}^p, d_p)$ of $(\mathcal{H}, d_p)$ are complete geodesic metric spaces, which means that between any two points of $\mathcal{E}^p$, there is a geodesic of $(\mathcal{E}^p, d_p)$ connecting them.
	\par 
	By definition, a \textit{geodesic ray} in $(\mathcal{E}^p, d_p)$ is the distinguished geodesic segment, closed at one extremity and of infinite length, which can be constructed as some psh envelope, see (\ref{eq_geod_as_env}), or, alternatively, as a solution to a certain Monge-Ampère equation, see (\ref{eq_ma_geod}).
	For $p \in ]1, +\infty[$, by the result of Darvas-Lu \cite[Theorem 1.2]{DarLuGeod}, the space $(\mathcal{E}^p, d_p)$ is uniquely geodesic, and the above notion of geodesic rays coincides with the respective notion in the sense of metric spaces.
	\par 
	By \cite[Theorem 1.3]{DarLuGeod}, the space of geodesic rays satisfies Euclid's 5th postulate for half-lines, meaning that geodesic rays departing from different initial points are in bijective correspondence.
	From now on, we consider geodesic rays departing from a fixed initial point.
	\par 
	Following Darvas-Lu \cite[(3)]{DarLuGeod}, we define the chordal $L^p$-distance, $d_p(\{ h^{L, 1}_{t} \}, \{ h^{L, 2}_{t} \})$, $p \in [1, +\infty]$, between two geodesic rays $h^{L, 1}_{t}, h^{L, 2}_{t}$, $t \in [0, + \infty[$, in the following way
	\begin{equation}\label{eq_defn_chordal}
		d_p(\{ h^{L, 1}_{t} \}, \{ h^{L, 2}_{t} \}) := 
		\lim_{t \to \infty} \frac{d_p(h^{L, 1}_{t}, h^{L, 2}_{t} )}{t}.
	\end{equation}
	The limit is finite by the triangle inequality and it exists by the fact that metric spaces $(\mathcal{E}^p, d_p)$ are Buseman convex, see Chen-Cheng \cite[Theorem 1.5]{ChenCheng3}, cf. (\ref{eq_bus_conv_mab}) and after Lemma \ref{eq_lim_d_p_inf}.
	Darvas-Lu in \cite[Theorem 1.3]{DarLuGeod} proved that the chordal distance is indeed a distance on the space of geodesic rays departing from the same initial point (in particular, it separates the geodesic rays).
	\par 
	Geodesic rays have recently found applications in several areas of complex geometry.
	Most notably, many results towards Yau-Tian-Donaldson conjecture, studying the existence of constant scalar curvature Kähler metrics in a given Kähler class, rely substantially on geodesic rays, see Phong-Ross-Sturm \cite{PhongRossSturm}, Paul-Tian \cite{PaulTianII}, Berman-Boucksom-Jonsson \cite{BerBouckJonYTD} or Li \cite{ChiLiENS}.
	\par 
	Part of the reason for this is that while the points of the space $\mathcal{H}$ parametrize geometric objects, (a subset of) points on the boundary at the infinity of $\mathcal{H}$ are parametrized by ample test-configurations -- some special degenerations of manifolds, algebraic in nature.
	This proves useful in relating the existence of a certain metric on the line bundle to some algebraic obstruction.
	\par 
	The main goal of the current article is to further investigate the geometry at the infinity of the space of positive metrics on a given ample line bundle by studying chordal distances between pairs of geodesic rays.
	As we show, for geodesic rays generated by ample test configurations, this chordal distance coincides with the spectral distance on filtrations on the section ring associated with the test configurations.
	This fulfills the general philosophy of Boucksom-Hisamoto-Jonsson \cite[Definition 3.4]{BouckHisJon} for the distance functional, saying that the limiting behavior of a functional on the boundary of the space of positive metrics should be related with an appropriate functional defined on the space of non-Archimedean metrics on the line bundle.
	\par 
	Remark that the chordal distance is a complex-geometric quantity, defined using complex pluripotential theory, and spectral distance on the filtrations is a purely algebrogeometric quantity.
	Our result, hence, lies on the interface of the three domains.
	\par 
	To describe our main statement in more details, recall that on the geometric side, to any ample test configuration $\mathcal{T}$ of $(X, L)$ and a fixed positive metric $h^L_0$ on $L$, Phong-Sturm in \cite[Theorem 3]{PhongSturmDirMA} associated a geodesic ray $h^{\mathcal{T}}_t$, $t \in [0, + \infty[$, of plurisubharmonic metrics on $L$ emanating from $h^L_0$ by considering the solution of the Dirichlet problem for a Monge-Ampère equation over a $\comp^*$-equivariant resolution of singularities of the test configuration with boundary conditions prescribed by the initial point of the ray, see Section \ref{sect_filt} for details.
	By the results of Chu-Tosatti-Weinkove \cite{ChuTossVeinC11}, based on the previoius work of Phong-Sturm \cite{PhongSturmRegul}, the metrics $h^{\mathcal{T}}_t$ are $\mathscr{C}^{1, 1}$; from toric examples \cite{ChenTangAster}, \cite{SongZeld}, we cannot hope for a better regularity in general. 
	\par 
	On the algebraic side, recall that Witt Nystr{\"o}m in \cite[Lemma 6.1]{NystOkounTest} associated with any test configuration $\mathcal{T}$ a submultiplicative filtration $\mathcal{F}^{\mathcal{T}}$ on the \textit{section ring}
	\begin{equation}
		R(X, L) := \oplus_{k = 1}^{\infty} H^0(X, L^k),
	\end{equation}
	by considering the vanishing order along the central fiber of $\mathcal{T}$ of the $\comp^*$-equivariant meromorphic extension of a section from $R(X, L)$, see Section \ref{sect_filt} for details.
	\par 
	Now, for any two filtrations $\mathcal{F}_1, \mathcal{F}_2$ on a finitely dimensional vector space $V$, and any $p \in [1, +\infty]$, we define spectral distances $d_p(\mathcal{F}_1, \mathcal{F}_2)$ using the $l^p$-norms of the joint spectrum of the filtrations $\mathcal{F}_1, \mathcal{F}_2$, see (\ref{eq_dp_filtr}) for details.
	It was established by Chen-Maclean \cite[Theorem 4.3]{ChenMaclean}, cf. also Boucksom-Jonsson \cite[Theorem 3.3 and \S 3.4]{BouckJohn21}, that for the filtrations $\mathcal{F}^{\mathcal{T}_1}, \mathcal{F}^{\mathcal{T}_2}$ associated with ample test configurations $\mathcal{T}_1, \mathcal{T}_2$, and any $p \in [1, +\infty[$, the following limit exists
	\begin{equation}\label{eq_spec_dist}
		d_p(\mathcal{F}^{\mathcal{T}_1}, \mathcal{F}^{\mathcal{T}_2}) := 
		\lim_{k \to \infty} \frac{d_p(\mathcal{F}^{\mathcal{T}_1}_k, \mathcal{F}^{\mathcal{T}_2}_k)}{k},
	\end{equation}
	where $\mathcal{F}^{\mathcal{T}_1}_k, \mathcal{F}^{\mathcal{T}_2}_k$, $k \in \nat$, are the restrictions of $\mathcal{F}^{\mathcal{T}_1}, \mathcal{F}^{\mathcal{T}_2}$ on the graded pieces $H^0(X, L^k)$.
	We shall prove, cf. Remark \ref{rem_chen_mclean_alt_pf}, that the limit also exists for $p = +\infty$.
	\par 
	Our main result of this article says that the geometric and algebraic viewpoints on the distances associated with ample test configurations are compatible.
	\begin{thm}\label{thm_dist_na}
		For any ample test configurations $\mathcal{T}_1, \mathcal{T}_2$ and any $p \in [1, +\infty]$, we have
		\begin{equation}\label{eq_dist_na}
			d_p \big(\{ h^{\mathcal{T}_1}_t \}, \{ h^{\mathcal{T}_2}_t \} \big)
			=
			d_p \big( \mathcal{F}^{\mathcal{T}_1}, \mathcal{F}^{\mathcal{T}_2} \big).
		\end{equation}
	\end{thm}
	\begin{rem}
		a) A relation between the two distances was speculated and conjectured in the literature, see Darvas-Lu \cite[p. 3 and 7]{DarLuGeod}, Zhang \cite[Remark 6.12]{KeweiZhangVal} and Remark \ref{rem_dist_na_sm} for details. 
		\par
		b) 
		When one of the test configurations is trivial and $p \in [1, +\infty[$, (\ref{eq_dist_na}) is equivalent to the result of Hisamoto \cite{HisamSpecMeas}, which followed the work of Witt Nystr{\"o}m \cite[Theorems 1.1 and 1.4]{NystOkounTest}, see Remark \ref{rem_spec_meas} for details.
		For $p = 1$, (\ref{eq_dist_na}) is due to Reboulet \cite[Theorem 4.1.1]{ReboulGeodDeger}, see Remark \ref{rem_geod_rays}.
		Our proof of Theorem \ref{thm_dist_na} is new even in these special cases.
	\end{rem}
	\par 
	\par 
	Let us now briefly describe the main idea behind the proof of Theorem \ref{thm_dist_na}.
	It relies in an essential way on the well-known observation that one can naturally interpret the space of filtrations on a given finitely dimensional vector space as the boundary at the infinity (viewed in terms of geodesic rays) of the space of Hermitian norms on the vector space, see \cite[\S II.10]{HaefligerBrid}.
	This result can be viewed as a finitely-dimensional analogue of Theorem \ref{thm_dist_na}.
	To pass from this finitely-dimensional picture to the infinitely-dimensional one, we rely on the methods of geometric quantization.
	Previous works of Phong-Sturm  \cite{PhongSturmDirMA}, \cite{PhongSturmRegul} and the author \cite{FinSecRing}, \cite{FinNarSim}, lie in the heart of our approach.
	\par 
	More precisely, recall that Phong-Sturm in \cite{PhongSturmDirMA} constructed for any ample test configuration a ray of Hermitian norms on $R(X, L)$ which quantizes the geodesic ray of metrics on $L$ associated with the test configuration (in the sense that the Fubini-Study metric of the ray of norms is related to the ray of metrics), see Theorem \ref{thm_phong_sturm_limit}.
	Recall further that author in \cite{FinNarSim} established that the Fubini-Study map, when restricted to the set of submultiplicative norms, is an isometry with respect to the natural distances, see Theorem \ref{thm_d_p_norm_fs_rel}.
	These two results as well as the fact that the space of Hermitian norms endowed with the natural distances is Buseman convex, see (\ref{eq_toponogov}), and the fact that the geodesic ray of Hermitian norms, constructed by Phong-Sturm is “almost submultiplicative" in the sense which will be made precise in Section \ref{sect_part1_pf}, allow us to establish one part of Theorem \ref{thm_dist_na}, showing that the left-hand side of (\ref{eq_dist_na}) is no bigger than the right-hand side.
	\par 
	Establishing the opposite bound, showing that the right-hand side of (\ref{eq_dist_na}) is no bigger than the left-hand side, is much more intricate and requires a more detailed analysis of the geodesic ray at the infinity.
	We first compare in Theorem \ref{thm_2_step} the geodesic ray of Hermitian norms on the section ring and the ray of $L^2$-norms associated with the geodesic ray of the test configuration.
	To do so, we rely on the results of Phong-Sturm \cite{PhongSturmRegul} about the boundness of geodesic rays, see Theorem \ref{thm_ph_st_regul}, and on a refinement of our previous work on the study of the metric structure of section ring, \cite{FinSecRing}, showing that our results can be extended in the degenerating family setting.	
	Then we prove that it is sufficient to assume that the singularities of the central fibers of the test configurations are mild enough.
	Then we show that for test configurations with mild singularities, it is possible to estimate the distance between the $L^2$-norms of geodesic rays of metrics in terms of the distance between the geodesic rays of metrics themselves.
	This is done by using quantized maximum principle of Berndtsson \cite{BerndtProb} and by relying on the techniques of Dai-Liu-Ma \cite{DaiLiuMa} and Ma-Marinescu \cite{MaHol} on the study Bergman kernels, which we generalize to the setting of degerating families of manifolds.
	In total, we establish another part of Theorem \ref{thm_dist_na}, showing that the right-hand side of (\ref{eq_dist_na}) is no bigger than the left-hand side, finishing the proof of Theorem \ref{thm_dist_na}.
	\par 
	This article is organized as follows.
	In Section \ref{sect_prel}, we recall the necessary preliminaries for Theorem \ref{thm_dist_na} and provide some applications.
	In Section \ref{sect_first}, we establish one part of Theorem \ref{thm_dist_na}, showing that the left-hand side of (\ref{eq_dist_na}) is no bigger than the right-hand side, and in Section \ref{sect_second}, we establish the opposite bound.
	\par 
	\textbf{Notations}. 
	We denote by $\mathbb{D}_n(r)$ (resp. $\mathbb{D}^*_n(r)$) the (resp. punctured) euclidean ball in $\comp^n$ of radius $r > 0$, and by $\mathbb{D}_n(r_1, r_2)$ the euclidean annulus in $\comp^n$ of interior radius $r_1 > 0$ and exterior radius $r_2 > r_1$.
	When $n = 1$ or $r = 1$, we omit them from the notation.
	\par 
	On a metric space $(X, d)$, for $x \in X$, $r > 0$, we denote by $B(x, r)$ the ball of radius $r$ around $x$.
	For a function $f : X \to \real$, defined on $(X, d)$, we denote by $f_*$ the lower-semicontinuous regularization of $f$, given by $f_*(x) := \lim_{\epsilon \to 0} \inf_{y \in B(x, \epsilon)} f(y)$.
	We similarly define the upper-semicontinuous regularization and we extend these notations to metrics on line bundles.
	\par 
	Let $(X, \omega)$ be a compact Kähler manifold.
	By $\partial \dbar$-lemma, the space $\mathcal{H}_{[\omega]}$ of Kähler metrics on $X$ cohomologous to $\omega$ can be identified with the space $\mathcal{H}_{\omega}$ of Kähler potentials, consisting of $u \in \ccal^{\infty}(X, \real)$, such that $\omega_u := \omega + \imun \partial \dbar u$ is strictly positive.
	Assume that there is a holomorphic line bundle $L$, such that the De Rham class $[\omega]$ of $\omega$ is related with the first Chern class $c_1(L)$ of $L$ as $[\omega] = 2 \pi c_1(L)$.
	Then the space $\mathcal{H}_{\omega}$ can be viewed as the space of positive Hermitian metrics $\mathcal{H}^L$ on $L$ upon the identification 
	\begin{equation}\label{eq_pot_metr_corr}
		u \mapsto h^L := e^{-u} \cdot h^L_0, 
	\end{equation}
	where $h^L_0$ is a positive Hermitian metric on $L$, verifying $\omega = 2 \pi c_1(L, h^L_0)$. 
	The function $u$ is called the potential of $h^L$.
	These identifications will be implicit later on, and we sometimes use the letter $\mathcal{H}$ to designate $\mathcal{H}_{\omega}, \mathcal{H}_{[\omega]}$ or $\mathcal{H}^L$.
	\par 
	We denote by ${\rm{PSH}}(X, \omega)$ the set of $\omega$-psh potentials; these are upper semicontinuous functions $u \in L^1(X, \real \cup \{ -\infty \})$, such that 
	\begin{equation}
		\omega_u := \omega + \imun \partial \dbar u
	\end{equation}
	is positive as a $(1, 1)$-current.
	We say a (singular) metric $h^L$ on $L$ is psh if its potential is $\omega$-psh.
	\par 
	A Hermitian metric $h^L$ on a line bundle $L$ over a compact manifold is called \textit{bounded} if for any (or some) smooth metric $h^L_0$ on $L$, there is $C > 0$, such that $\exp(-C) \cdot h^L_0 \leq h^L \leq \exp(C)  \cdot h^L_0$.
	We denote by $d_{+ \infty}(h^L_0, h^L)$ the smallest constant $C > 0$, verifying above inequality.
	\par 
	For a fixed Hermitian metric $h^L$ on a line bundle $L$ over a manifold $X$ (resp. and a measure $\mu$ on $X$), we denote by ${\rm{Ban}}^{\infty}_k(h^L) = \| \cdot \|_{L^{\infty}_k(h^L)}$ (resp. ${\rm{Hilb}}_k(h^L, \mu) = \| \cdot \|_{L^2_k(h^L, \mu)}$), $k \in \nat$, the $L^{\infty}$-norm (resp. $L^2$-norm) on $H^0(X, L^k)$ induced by $h^L$ (resp. and $\mu$), i.e. for any $f \in H^0(X, L^k)$, we define $\| f \|_{L^{\infty}_k(h^L)} = \sup_{x \in X} |f(x)|_{h^L}$ (resp. $\| f \|_{L^2_k(h^L, \mu)} = \int_{x \in X} |f(x)|^2_{h^L} d \mu(x)$).
	We denote by ${\rm{Ban}}^{\infty}(h^L) = \sum_{k = 0}^{\infty} {\rm{Ban}}^{\infty}_k(h^L)$ and ${\rm{Hilb}}(h^L, \mu)  = \sum_{k = 0}^{\infty} {\rm{Hilb}}_k(h^L, \mu)$ the induced graded norms on $R(X, L)$.
	When $h^L$ is bounded psh and $\mu$ is given by $\frac{1}{n!} c_1(L, h^L)^n$, where the power is interpreted in Bedford-Taylor sense \cite{BedfordTaylor}, we omit $\mu$ from the notation.
	When the volume form $\mu$ is the symplectic volume $\frac{\omega^n}{n!}$ of some Kähler form $\omega$ on $X$, we denote ${\rm{Hilb}}(h^L, \mu)$ by ${\rm{Hilb}}(h^L, \omega)$.
	\par 
	\textbf{Acknowledgement}. 
	I would like to thank Rémi Reboulet and Lars Martin Sektnan for their invitation to University of Gothenburg; in particular, Rémi who drew my attention to the problem of this article during my visit and shared some of his ideas.
	I also thank Sébastien Boucksom for many enlightening discussions on non-Archimedean pluripotential theory and related fields.
	Finally, I would like to acknowledge the support of CNRS and École polytechnique.

\section{Norms, filtrations, metrics and degenerations}\label{sect_prel}
	The main goal of this section is to recall the necessary preliminaries for Theorem \ref{thm_dist_na} and to describe some applications.
	More precisely, in Section \ref{sect_append} we introduce natural metric structures on the sets Hermitian norms and filtrations on a given finitely dimensional vector space.
	In Section \ref{sect_pp_thr}, we recall the basics of pluripotential theory.
	In Section \ref{sect_filt}, we recall the basics of test configurations.
	Finally, in Section \ref{sect_spec_mes}, we describe some applications of Theorem \ref{thm_dist_na}.
	
	\subsection{Metric structures on Hermitian norms and filtrations}\label{sect_append}
	The main goal of this section is to introduce natural distances on the spaces of Hermitian norms and filtrations on a given finitely dimensional vector space.
	\par 
	Let $V$ be a complex vector space, $\dim V = n$.
	We denote by $\mathcal{H}_V$ the space of Hermitian norms $H$ on $V$, viewed as an open subset of the Hermitian operators ${\rm{Herm}}(V)$. 
	Let $\lambda_1, \ldots, \lambda_n$ be the ordered spectrum of $h \in {\rm{Herm}}(V)$ with respect to a norm $H \in \mathcal{H}_V$.
	For $p \in [1, +\infty[$, we define
	\begin{equation}
		\| h \|^H_p
		:=
		\sqrt[p]{\frac{\sum_{i = 1}^{\dim V} |\lambda_i|^p}{\dim V}},
	\end{equation}
	and we let $\| h \|^H_{+ \infty} := \max |\lambda_i|$.
	By Ky Fan inequality, one can establish that $\| \cdot \|^H_p$, $p \in [1, +\infty]$, is a Finsler norm for any $H$, i.e. it satisfies the triangle inequality, cf. \cite[Lemma 1.1]{BouckICM}.
	\par 
	We then define the length metric $d_p(H_0, H_1)$, $H_0, H_1 \in \mathcal{H}_V$, as usual through the infimum of the length $l(\gamma) := \int_0^1 \|\gamma'(t)\|^{\gamma(t)}_p dt$, where $\gamma$ is a piecewise smooth path in $\mathcal{H}_V$ joining $H_0, H_1$.
	\par 
	One can verify, cf. \cite[Theorem 3.1]{BouckErik21}, that this metric admits the following explicit description.
	Let $T \in {\rm{Herm}}(V)$, be the \textit{transfer map} between Hermitian norms $H_0, H_1 \in \mathcal{H}_V$, i.e. the Hermitian products $\langle \cdot, \cdot \rangle_{H_0}, \langle \cdot, \cdot \rangle_{H_1}$ induced by $H_0$ and $H_1$, are related as $\langle \cdot, \cdot \rangle_{H_1} = \langle T \cdot, \cdot \rangle_{H_0}$, then
	\begin{equation}\label{eq_dist_transf}
		d_p(H_0, H_1)
		=
		\sqrt[p]{\frac{{\rm{Tr}}[|\log T|^p]}{\dim V}},
	\end{equation}
	for any $p \in [1, +\infty[$ and $d_{+ \infty}(H_0, H_1) = \|\log T \|$, where $\| \cdot \|$ is the operator norm with respect to $H_0$.
	Moreover, the Hermitian norms $H_t$, $t \in [0, 1]$, corresponding to the scalar products $\langle \cdot, \cdot \rangle_{H_t} := \langle T^t \cdot, \cdot \rangle_{H_0}$ are geodesics in $(\mathcal{H}_V, d_p)$, $p \in [1, +\infty]$.
	Later on, we call them the \textit{distinguished geodesics}.
	For $p \in ]1, +\infty[$, it is possible to verify that $(\mathcal{H}_V, d_p)$ is a uniquely geodesic space, cf. \cite[Theorem 6.1.6]{BhatiaBook}, and hence these are the only geodesic segments between $H_0$ and $H_1$; see, however, \cite[Lemma 2.4]{FinNarSim} for a counterexample of the analogous statement for $p = 1, +\infty$.
	\par 
	Let us now discuss the non-Archimedean part of the story.
	A \textit{filtration} $\mathcal{F}$ of a vector space $V$ is a map from $\real$ to vector subspaces of $V$, $t \mapsto \mathcal{F}^t V$, verifying $\mathcal{F}^t V \subset \mathcal{F}^s V$ for $t > s$, and such that $\mathcal{F}^t V  = V$ for sufficiently small $t$ and $\mathcal{F}^t V = \{0\}$ for sufficiently big $t$.
	We assume that this map is left-continuous, i.e. for any $t \in \real$, there is $\epsilon_0 > 0$, such that $\mathcal{F}^t V = \mathcal{F}^{t - \epsilon} V $ for any $0 < \epsilon < \epsilon_0$.
	We define the \textit{jumping numbers} $e_{\mathcal{F}}(j)$, $j = 1, \ldots, n$, of the filtration $\mathcal{F}$ as follows
	\begin{equation}\label{eq_defn_jump_numb}
		e_{\mathcal{F}}(j) := \sup \Big\{ t \in \real : \dim \mathcal{F}^t V \geq j \Big\}.
	\end{equation}
	\par 
	Filtrations $\mathcal{F}$ on $V$ are in bijection with functions $\chi_{\mathcal{F}} : V \to [0, +\infty[$, defined as
	\begin{equation}\label{eq_filtr_norm}
		\chi_{\mathcal{F}}(s) := \exp(- w_{\mathcal{F}}(s)).
	\end{equation}
	where $w_{\mathcal{F}}(s)$ is the weight associated with the filtration, defined as
	$
		w_{\mathcal{F}}(s) := \sup \{ \lambda \in \real : s \in \mathcal{F}^{\lambda} V \}
	$.
	An easy verification shows that $\chi_{\mathcal{F}}$ is a non-Archimedean norm on $V$ with respect to the trivial absolute value on $\comp$, i.e. it satisfies the following axioms
	\begin{enumerate}
		\item $\chi_{\mathcal{F}}(f) = 0$ if and only if $f = 0$,
		\item $\chi_{\mathcal{F}}(\lambda f) = \chi_{\mathcal{F}}(f)$, for any $\lambda \in \comp^*$, $f \in V$,
		\item $\chi_{\mathcal{F}}(f + g) \leq \max \{ \chi_{\mathcal{F}}(f), \chi_{\mathcal{F}}(g) \}$, for any $f, g \in V$.
	\end{enumerate}
	\par 
	As it was established for example in \cite[Lemma II.10.80]{HaefligerBrid}, for any two filtrations $\mathcal{F}_1, \mathcal{F}_2$ on $V$, there is a basis $e_1, \ldots, e_n$ of $V$, which jointly diagonalizes $\chi_{\mathcal{F}_1}$ and $\chi_{\mathcal{F}_2}$, i.e. for any $\lambda_1, \ldots, \lambda_n \in \comp$ and $j = 1, 2$, we have
	\begin{equation}\label{eq_sim_diag_nna}
		\chi_{\mathcal{F}_j} \Big(\sum_{i = 1}^{n} \lambda_i e_i \Big) = \max{}_{i = 1}^{n} \big\{ \chi_{\mathcal{F}_j}(\lambda_i e_i) \big\}.
	\end{equation}
	Analogously to (\ref{eq_dist_transf}), for $p \in [1, +\infty[$, we define using this basis
	\begin{equation}\label{eq_dp_filtr}
		d_p(\mathcal{F}_1, \mathcal{F}_2)
		=
		\sqrt[p]{\frac{\sum_{i = 1}^{\dim V} |w_{\mathcal{F}_1}(e_i) - w_{\mathcal{F}_2}(e_i)|^p}{\dim V}},
	\end{equation}
	and we let $d_{+ \infty}(\mathcal{F}_1, \mathcal{F}_2) := \max_{x \in V \setminus \{ 0 \}} |w_{\mathcal{F}_1}(x) - w_{\mathcal{F}_2}(x)|$.
	\par 
	As we recall later in (\ref{eq_na_nm_interpol}), the space of non-Archimedean norms on a finitely dimensional vector space can be viewed as the boundary at the infinity of the space of Hermitian norms. 
	The metric structures (\ref{eq_dist_transf}) and (\ref{eq_dp_filtr}) are compatible under this identification, see (\ref{eq_d_p_fil_norms_herm}).
	Our main result, Theorem \ref{thm_dist_na}, is an analogue of this statement in the infinitely-dimensional setting.

	\subsection{Pluripotential theory and geodesic segments between positive metrics}\label{sect_pp_thr}
	The main goal of this section is to recall some basic facts from complex pluripotential theory, emphasizing the metric part and in particular the study of geodesic segments.
	\par 
	Let us fix a Kähler form $\omega$ on $X$.
	One can introduce on the space of Kähler potentials $\mathcal{H}_{\omega}$ a collection of $L^p$-type Finsler metrics, $p \in [1, +\infty[$, defined as follows.
	If $u \in \mathcal{H}_{\omega}$ and $\xi \in T_u \mathcal{H}_{\omega} \simeq \ccal^{\infty}(X, \real)$, then the $L^p$-length of $\xi$ is given by the following expression
	\begin{equation}\label{eq_finsl_dist_fir}
		\| \xi \|_p^u
		:=
		\sqrt[p]{
		\frac{1}{\int \omega^n}
		 \int_X |\xi(x)|^p \cdot \omega_u^n(x)}.
	\end{equation}
	For $p = 2$, this was introduced by Mabuchi \cite{Mabuchi}, and for $p \in [1, +\infty[$ by Darvas \cite{DarvasFinEnerg}.
	For brevity, we omit $\omega$ from our further notations.
	\par 
	Darvas in \cite{DarvasFinEnerg} studied the completion $(\mathcal{E}^p, d_p)$ of the path length metric structures $(\mathcal{H}, d_p)$ associated with (\ref{eq_finsl_dist_fir}), and proved that these completions are geodesic metric spaces and have a vector space structure.
	Certain geodesic segments of $(\mathcal{E}^p, d_p)$ can be constructed as upper envelopes of quasi-psh functions. 
	More precisely, we identify paths $u_t \in \mathcal{E}^p$, $t \in [0, 1]$, with rotationally-invariant functions $\hat{u}$ over $X \times \mathbb{D}(e^{-1}, 1)$ through the following formula 
	\begin{equation}\label{eq_defn_hat_u}
		\hat{u}(x, \tau) = u_{t}(x), \quad \text{where} \quad x \in X \, \text{ and } \, t = - \log |\tau|.
	\end{equation}
	We say that a curve $[0,1] \ni t \to v_t \in \mathcal{E}^p$ is a \textit{weak subgeodesic} connecting $u_0, u_1 \in \mathcal{E}^p$ if $d_p(v_t, u_i) \to 0$, as $t \to 0$ for $i = 0$ and $t \to 1$ for $i = 1$, and $\hat{u}$  is $\pi^* \omega$-psh on $X \times \mathbb{D}(e^{-1}, 1)$.
	As shown in \cite[Theorem 2]{DarvasFinEnerg}, the following envelope 
	\begin{equation}\label{eq_geod_as_env}
		u_t := \sup \Big\{ 
			v_t \, : \, t \to v_t \,  \text{ is a weak subgeodesic connecting } \, v_0 \leq u_0 \text{ and } v_1 \leq u_1
		\Big\},
	\end{equation}
	is a $d_p$-geodesic connecting $u_0, u_1$.
	It will be later called the \textit{distinguished geodesic segment}.
	\par 
	According to Chen-Cheng \cite[Theorem 1.5]{ChenCheng3}, the metric spaces $(\mathcal{E}^p, d_p)$, $p \in [1, +\infty[$, are Buseman convex, i.e. for any distinguished geodesic segments $u_t, v_t \in \mathcal{E}^p$, $t \in [0, 1]$, departing from the same initial point, for any $s \in [0, 1]$, we have
	\begin{equation}\label{eq_bus_conv_mab}
		\frac{d_p(u_s, v_s)}{s} \leq d_p(u_1, v_1).
	\end{equation}
	The space $(\mathcal{E}^2, d_2)$ is, moreover, $\rm{CAT}(0)$ by the result of Darvas \cite[Theorem 1]{DarvasMabCompl}, building on the previous work of Calabi-Chen \cite[Theorem 1.1]{CalabiChen}.
	\par 
	It is well-known, cf. Guedj-Zeriahi \cite[Exercise 10.2]{GuedjZeriahBook}, that
	\begin{equation}\label{eq_inter_ep}
		\underset{p \in [1, +\infty[}{\cap}  \mathcal{E}^p
		=
		{\rm{PSH}}(X, \omega) \cap L^{\infty}(X).
	\end{equation}
	When $u_0, u_1 \in {\rm{PSH}}(X, \omega) \cap L^{\infty}(X)$, Berndtsson \cite[\S 2.2]{BernBrunnMink} in \cite[\S 2.2]{BernBrunnMink} proved that $u_t$, $t \in [0, 1]$, defined by (\ref{eq_geod_as_env}), verifies $u_t \in L^{\infty}(X)$ and it can be described as the only path connecting $u_0$ to $u_1$, so that $\hat{u}$ is the solution of the following Monge-Ampère equation
	\begin{equation}\label{eq_ma_geod}
		(\pi^* \omega + \imun \partial \dbar \hat{u})^{n + 1} = 0,
	\end{equation}
	where the wedge power is interpreted in Bedford-Taylor sense \cite{BedfordTaylor}.
	For smooth geodesic segments in $(\mathcal{H}, d_2)$, Semmes \cite{Semmes} and Donaldson \cite{DonaldSymSp} have made similar observations before.
	The uniqueness of the solution of (\ref{eq_ma_geod}) is assured by \cite[Lemma 5.25]{GuedjZeriahBook}.
	Remark, in particular, that for any $u_0, u_1 \in {\rm{PSH}}(X, \omega) \cap L^{\infty}(X)$, the distinguished weak geodesic connecting them is the same if we view $u_0, u_1$ as elements in any of $\mathcal{E}^p$, $p \in [1, + \infty[$.
	\par 
	Now, we define the space $\mathcal{E}^{+ \infty}$ as the completion of $\mathcal{H}$ with respect to the distance $d_{+\infty}$.
	More explicitly, by a version of Demailly's regularization theorem, see \cite{Dem82}, \cite{DemRegul}, $\mathcal{E}^{+ \infty}$ can be identified with the space of continuous psh metrics on $L$, cf. \cite[Theorem 8.1]{GuedZeriGeomAnal} and \cite[\S 4.2]{FinSecRing}. 
	From \cite[Theorem 4.5]{FinSecRing}, the metric $d_{+\infty}$ on $\mathcal{E}^{+ \infty}$ can be alternatively defined as the path length metric structure associated with the $L^{\infty}$-length, defined in the notations of (\ref{eq_finsl_dist_fir}) as $\| \xi \|_{+\infty}^u := \sup |\xi(x)|$, in a direct analogy with the definitions of $d_p$, $p \in [1, +\infty[$.
	The following result is undoubtedly well-known to the experts in the field.
	We present its proof later this section.
	\begin{lem}\label{eq_lim_d_p_inf}
		For any $u_0, u_1 \in \mathcal{E}^{+ \infty}$, we have
		$
			\lim_{p \to \infty} d_p(u_0, u_1) = d_{+ \infty}(u_0, u_1)
		$.
	\end{lem}
	From Lemma \ref{eq_lim_d_p_inf}, the space $(\mathcal{E}^{+ \infty}, d_{+ \infty})$ is Buseman convex in the sense (\ref{eq_bus_conv_mab}).
	Hence, the chordal distance (\ref{eq_defn_chordal}) between rays of continuous metrics is well-defined for any $p \in [1, +\infty]$.
	\par 
	\begin{sloppypar}
	Recall that the distance between two given elements $u_0, u_1 \in {\rm{PSH}}(X, \omega) \cap L^{\infty}(X)$  can be expressed in terms of the distinguished geodesics $u_t$, $t \in [0, 1]$, connecting them, see (\ref{eq_geod_as_env}).
	More precisely, Berndtsson in \cite[\S 2.2]{BernBrunnMink} proved that $u_t \in L^{\infty}(X)$ and the limits $\lim_{t \to 0} u_t = u_0$, $\lim_{t \to 1} u_t = u_1$ hold in the uniform sense.
	Since $u_t$ is a weak subgeodesic, for fixed $x \in X$, the function $u_t(x)$ is convex in $t \in [0, 1]$, see \cite[Theorem I.5.13]{DemCompl}.
	Hence, one-sided derivatives $\dot{u}_t^{-}$, $\dot{u}_t^{+}$ of $u_t$ are well-defined for $t \in ]0, 1[$ and they increase in $t$.
	We denote $\dot{u}_0 := \lim_{t \to 0} \dot{u}_t^{-} = \lim_{t \to 0} \dot{u}_t^{+}$.
	From \cite[\S 2.2]{BernBrunnMink}, we know that $\dot{u}_0$ is bounded and by Darvas \cite[Theorem 1]{DarvWeakGeod}, we, moreover, have
	\begin{equation}\label{eq_bnd_darvas_sup}
		\sup |\dot{u}_0| \leq \sup |u_1 - u_0|.
	\end{equation}
	According to Darvas-Lu-Rubinstein \cite[Lemma 4.5]{DarvLuRub}, refining previous result of Chen \cite{ChenGeodMab} and Darvas \cite{DarvasMabCompl}, for any $u_0 \in \mathcal{H}_{\omega}$, $u_1 \in {\rm{PSH}}(X, \omega) \cap L^{\infty}(X)$, $p \in [1, +\infty[$, we have
	\begin{equation}\label{eq_d_p_berndss}
		d_p(u_0, u_1)
		=
		\sqrt[p]{ \frac{1}{\int \omega^n} \int_X |\dot{u}_0(x)|^p \cdot \omega_{u_0}^n(x)}.
	\end{equation}
	From Darvas \cite[Theorem 7.2]{DarvasMabCompl}, we actually know that (\ref{eq_d_p_berndss}) holds also for $u_0, u_1 \in {\rm{PSH}}(X, \omega) \cap L^{\infty}(X)$, such that $\Delta u_0, \Delta u_1 \in L^{\infty}(X)$, where $\Delta$ is the Laplace operator on $X$ (we say in this case that $u_0, u_1 \in \mathscr{C}^{1, \overline{1}}(X)$).
	\begin{proof}[Proof of Lemma \ref{eq_lim_d_p_inf}]
		Let us first assume that $u_0, u_1$ are smooth and positive, i.e. $u_0, u_1 \in \mathcal{H}_{\omega}$.
		Then by (\ref{eq_d_p_berndss}),
		$
			\lim_{p \to + \infty} d_p(u_0, u_1) = \sup |\dot{u}_0|
		$.
		However, for geodesics with smooth extremities, we have $\sup |\dot{u}_0| = \sup |u_1 - u_0|$, cf. \cite[Lemma 4.8]{FinSecRing}. This proves Lemma \ref{eq_lim_d_p_inf} in that case.
		\par 
		By Demailly's regularization theorem, see \cite{Dem82}, \cite{DemRegul}, cf. \cite[Theorem 8.1]{GuedZeriGeomAnal}, for any $u_0, u_1 \in \mathcal{E}^{+ \infty}$, there are sequences $u_{0, i}, u_{1, i}  \in \mathcal{H}_{\omega}$, $i \in \nat^*$, which converge uniformly, as $i \to \infty$, to $u_0$ and $u_1$ respectively.
		Since by above, Lemma \ref{eq_lim_d_p_inf} holds for $u_{0, i}, u_{1, i}$, it holds in full generality.
	\end{proof}
	\end{sloppypar}
	\par 
	We say that the path $u_t \in {\rm{PSH}}(X, \omega) \cap L^{\infty}(X)$, $t \in [0, 1]$, is $\mathscr{C}^{1, \overline{1}}$ if $\hat{u}, \Delta \hat{u} \in L^{\infty}(X \times \mathbb{D}(e^{-1}, 1))$, where $\Delta$ is the Laplace operator on $X \times \mathbb{D}(e^{-1}, 1)$.
	By standard regularity results, we then see that $u \in \mathscr{C}^{1, \alpha}(X \times \mathbb{D}(e^{-1}, 1))$ for any $\alpha < 1$.
	Hence, the two-sided derivatives $\dot{u}_t^{-}$ and $\dot{u}_t^{+}$ coincide, and we denote them by $\dot{u}_t$.
	Berndtsson in \cite[Proposition 2.2]{BerndtProb}, cf. also Darvas \cite[Theorem 7.2, (49) and (50)]{DarvasMabCompl}, established that for $\mathscr{C}^{1, \overline{1}}$ geodesic rays $u_t$, $t \in [0, 1]$, the following bounded measure on the real line
	\begin{equation}\label{eq_berndt_meas}
		\mu_t
		:=
		(\dot{u}_t )_* \big( \omega_{u_t}^n \big),
	\end{equation}
	doesn't depend on $t \in [0, 1]$.
	From (\ref{eq_d_p_berndss}), we see that the absolute moments of this measure, $\int |x|^p d\mu_t(x)$, coincide with $d_p(u_0, u_1)$, $p \in [1, +\infty[$.
	
	\subsection{Test configurations, submultiplicative filtrations and geodesic rays}\label{sect_filt}
	The main goal of this section is to recall the definition of a test configuration and the relation between test configurations, submultiplicative filtrations and geodesic rays of metrics.
	Recall first that a test configuration $\mathcal{T} = (\pi: \mathcal{X} \to \comp, \mathcal{L})$ for $(X, L)$ consists of
	\begin{enumerate}
		\item A scheme $\mathcal{X}$ with a $\comp^*$-action $\rho$,
		\item A $\comp^*$-equivariant line bundle $\mathcal{L}$ over $\mathcal{X}$,
		\item A flat $\comp^*$-equivariant projection $\pi : \mathcal{X} \to \comp$, where $\comp^*$ acts on $\comp$ by multiplication, such that if we denote its fibers by $X_{\tau} := \pi^{-1}(\tau)$, $\tau \in \comp$, then $(X_1, \mathcal{L}|_{X_1})$ is isomorphic to $(X, L)$.
	\end{enumerate}
	Remark that our definition differs slightly from the usual one, requiring $(X_1, \mathcal{L}|_{X_1})$ to be isomorphic with $(X, L^r)$ for some $r \in \nat^*$.
	\par 
	We say that a test configuration is \textit{(semi)ample} if $\mathcal{L}$ is relatively (semi)ample.
	We say that it is \textit{normal} if $\mathcal{X}$ is normal.
	Remark that the $\comp^*$-action induces the canonical isomorphisms
	\begin{equation}\label{eq_can_ident_test}
		\mathcal{X} \setminus X_0 \simeq \comp^* \times X, \qquad \mathcal{L}|_{\mathcal{X} \setminus X_0} \simeq p^* L,
	\end{equation}
	where $p : \comp^* \times X \to X$ is the natural projection.
	\par 
	Now, a collection of filtrations on the graded pieces $A_k$ of a graded vector space $A := \oplus_{k = 0}^{+\infty} A_k$ is called a (graded) filtration on $A$.
	We say that a graded filtration $\mathcal{F}$ is \textit{bounded} if there is $C > 0$, such that for any $k \in \nat^*$, $\mathcal{F}^{ C k} A_k = \{0\}$.
	A graded filtration $\mathcal{F}$ on a ring is called \textit{submultiplicative} if for any $t, s \in \real$, $k, l \in \nat$, we have 
	\begin{equation}\label{eq_subm_filt}
		\mathcal{F}^t A_k \cdot \mathcal{F}^s A_l \subset \mathcal{F}^{t + s} A_{k + l}.
	\end{equation}
	Remark that for any submultiplicative graded filtration on a finitely generated ring $A$, there is $C > 0$, such that for any $k \in \nat^*$, $\mathcal{F}^{- C k} A_k = A_k$.
	We say that a filtration $\mathcal{F}$ is a \textit{$\mathbb{Z}$-filtration} if its weights are integral.
	A $\mathbb{Z}$-filtration on a finitely generated ring $A$ is called a \textit{filtration of finite type} if the associated $\comp[\tau]$-algebra ${\rm{Rees}}(\mathcal{F}) := \sum_{(\lambda, k) \in \mathbb{Z} \times \mathbb{N}} \tau^{- \lambda} \mathcal{F}^{\lambda} A_k$, also called the Rees algebra, is finitely generated.
	Remark that one can reconstruct a filtration $\mathcal{F}$ from the $\comp[\tau]$-algebra structure of ${\rm{Rees}}(\mathcal{F})$.
	Clearly, filtrations of finite type are bounded.
	\par 
	Following Witt Nystr{\"o}m \cite[Lemma 6.1]{NystOkounTest}, let us construct a submultiplicative filtration  $\mathcal{F}^{\mathcal{T}}$ on $R(X, L)$ associated with a test configuration $\mathcal{T}$ of $(X, L)$ as follows.
	Pick an element $s \in H^0(X, L^k)$, $k \in \nat^*$, and consider the section $\tilde{s} \in H^0(\mathcal{X} \setminus X_0, \mathcal{L}^k)$, obtained by the application of the $\comp^*$-action to $s$.
	By the flatness of $\pi$, the section $\tilde{s}$ extends to a meromorphic section over $\mathcal{X}$, cf. Witt Nystr{\"o}m \cite[Lemma 6.1]{NystOkounTest}.
	In other words, there is $l \in \integ$, such that for a coordinate $\tau$ on $\comp$, we have $\tilde{s} \cdot \tau^l \in H^0(\mathcal{X}, \mathcal{L}^k)$.
	We define the restriction $\mathcal{F}^{\mathcal{T}}_k$ of the filtration $\mathcal{F}^{\mathcal{T}}$ to $H^0(X, L^k)$ as
	\begin{equation}\label{eq_defn_filt_test}
		\mathcal{F}^{\mathcal{T} \lambda}_k H^0(X, L^k)
		:=
		\Big\{
			s \in H^0(X, L^k) : \tau^{- \lceil \lambda \rceil} \cdot \tilde{s} \in H^0(\mathcal{X}, \mathcal{L}^k)
		\Big\}, 
		\quad \lambda \in \real.
	\end{equation}
	Alternatively, for any $k \in \nat^*$, consider the embedding $H^0(\mathcal{X}, \mathcal{L}^k) \to H^0(X, L^k) \otimes \comp[\tau, \tau^{-1}]$, induced by (\ref{eq_can_ident_test}).
	An easy verification shows that $\mathcal{F}^{\mathcal{T}}$ is defined in such a way that under this embedding the $\comp[\tau]$-algebras $R(\mathcal{X}, \mathcal{L})$ and ${\rm{Rees}}(\mathcal{F}^{\mathcal{T}})$ are isomorphic, cf. \cite[(A.2)]{BouckJohn21}.
	As for ample $\mathcal{L}$, the $\comp[\tau]$-algebra $R(\mathcal{X}, \mathcal{L})$ is finitely generated, the filtration $\mathcal{F}^{\mathcal{T}}$ is of finite type for ample test configurations $\mathcal{T}$, cf. \cite[(9)]{NystOkounTest} or \cite[\S A.2]{BouckJohn21}.
	\par 
	From Rees construction, for any filtration $\mathcal{F}$ of finite type on $R(X, L)$, there is $d \in \nat^*$ and an ample test configuration $\mathcal{T}$ of $(X, L^d)$, such that the restriction of $\mathcal{F}$ to $R(X, L^d) \subset R(X, L)$ coincides with $\mathcal{F}^{\mathcal{T}}$.
	For completeness, let us recall this construction in details.
	Since $\mathcal{F}$ is of finite type, the associated Rees algebra ${\rm{Rees}}(\mathcal{F})$ is a finitely generated $\comp[\tau]$-algebra.
	Let $d \in \nat^*$ be such that ${\rm{Rees}}(\mathcal{F})^{(d)} := \sum_{(\lambda, k) \in \mathbb{Z} \times \mathbb{N}} \tau^{- \lambda} \mathcal{F}^{\lambda} H^0(X, L^{dk})$ is generated in degree one.
	Consider $\mathcal{X} := \rm{Proj}_{\comp[\tau]}({\rm{Rees}}(\mathcal{F})^{(d)})$, $\mathcal{L} := \mathscr{O}(1)$ with the natural map $\pi : \mathcal{X} \to \comp = \rm{Spec}(\comp[\tau])$.
	Clearly, $\mathcal{L}$ is ample, cf. \cite[Proposition 7.10]{HartsBook}.
	Remark that ${\rm{Rees}}(\mathcal{F})^{(d)}$ is torsion free, and so by \cite[Proposition 6.3]{EisenbudBook}, it is a flat $\comp[\tau]$-algebra, which means that $\pi$ is flat.
	There is also a natural equivariant $\comp^*$-action on $\mathcal{X}$, which intervenes with $\pi$.
	The fiber $X_1$ of $\pi$ at $\tau = 1$ equals to $\rm{Proj}({\rm{Rees}}(\mathcal{F})^{(d)} \otimes_{\comp[\tau]} \comp[\tau] / (\tau - 1))$, cf. \cite[p. 89]{HartsBook}, and since ${\rm{Rees}}(\mathcal{F})^{(d)} \otimes_{\comp[\tau]} \comp[\tau] / (\tau - 1)$ is isomorphic to $R(X, L^d)$, we have $(X_1, \mathcal{L}|_{X_1}) = (X, L^d)$ by \cite[Exercise II.5.13]{HartsBook}. 
	Hence, the pair $\mathcal{T} := (\pi : \mathcal{X} \to \comp, \mathcal{L})$ is a test configuration.
	In fact, by \cite[Exercise II.5.9b)]{HartsBook}, when restricted to elements of sufficiently large degree, we have an isomorphism between the $\comp[\tau]$-algebras $R(\mathcal{X}, \mathcal{L})$ and ${\rm{Rees}}(\mathcal{F})^{(d)}$. 
	Hence, the filtration associated with $\mathcal{T}$, restricted to elements of sufficiently big degree in $R(X, L^d)$, coincides with the restriction of the filtration $\mathcal{F}$.
	Moreover, from \cite[Proposition 2.15]{BouckHisJonsFourier}, this is a one-to-one correspondence between filtrations of finite type on $R(X, L)$ (considered modulo the restrictions as above) and ample test configurations of $(X, L^d)$ for $d \in \nat^*$.
	\par 
	Let us now recall some operations on ample test configurations, $\mathcal{T} = (\pi: \mathcal{X} \to \comp, \mathcal{L})$.
	Consider the normalization $p_0 : \widetilde{\mathcal{X}} \to \mathcal{X}$ of $\mathcal{X}$, and denote $\widetilde{\mathcal{L}} := p_0^* \mathcal{L}$, $\widetilde{\pi} := \pi \circ p_0$.
	Since $p_0$ is finite, $\widetilde{\mathcal{L}}$ is ample, cf. \cite[Proposition 4.4]{HartsAmple}.
	By the universal property of the normalization, the $\comp^*$-action on $\mathcal{X}$ can be lifted to the $\comp^*$-action on $\widetilde{\mathcal{X}}$.
	From \cite[Theorem III.9.7]{HartsBook} and \cite[Proposition 2.6]{BouckHisJonsFourier}, the map $\widetilde{\pi}$ is flat.
	Hence, the pair $\widetilde{\mathcal{T}} = (\widetilde{\pi}: \widetilde{\mathcal{X}} \to \comp, \widetilde{\mathcal{L}})$ is an ample test configuration of $(X, L)$, cf. \cite[\S 1.4]{BouckJonsGlobal}.
	By an abuse of notation, we call $\widetilde{\mathcal{T}}$ the normalization  of $\mathcal{T}$.
	\par 
	By equivariant Hironaka's resolution of singularities theorem, cf. Koll{\'a}r \cite[Proposition 3.9.1]{KollarResol}, $\widetilde{\mathcal{X}}$ admits a $\comp^*$-equivariant resolution $p : \mathcal{X}' \to \widetilde{\mathcal{X}}$ of singularities. We let $\mathcal{L}' := p^* \widetilde{\mathcal{L}}$, $\pi' := \widetilde{\pi} \circ p$.
	By \cite[Proposition II.7.16 and Theorems II.7.17, III.9.7]{HartsBook}, the map $\pi': \mathcal{X}' \to \comp$ is flat, and, hence, the pair $\mathcal{T}' := (\pi': \mathcal{X}' \to \comp, \mathcal{L}')$ is a (semiample) test configuration of $(X, L)$.
	By an abuse of notation, we call $\mathcal{T}'$ a resolution of singularities of $\mathcal{T}$.
	\par 
	Remark now that test configurations of $(X, L)$ form a category, where a morphism between $\mathcal{T} = (\mathcal{X}, \mathcal{L})$ and $\mathcal{T}' = (\mathcal{X}', \mathcal{L}')$ is given by a $\comp^*$-equivariant morphism  $p : \mathcal{X} \to \mathcal{X}'$ over $\comp$, compatible with the isomorphisms $\mathcal{X}'|_1 \simeq X \simeq \mathcal{X}|_1$. 
	There is at most one morphism between any two given test configurations, and we say that $\mathcal{T}$ dominates $\mathcal{T}'$ when it exists.
	Clearly, any morphism between test configurations is a birational map, and hence by \cite[Theorem II.7.17]{HartsBook} it is isomorphic to the blowup along a sheaf of ideals, which is trivial away from the central fiber.
	Remark that for any test configuration $\mathcal{T} = (\pi: \mathcal{X} \to \comp, \mathcal{L})$, by definition, there is a $\comp^*$-equivariant birational map $\comp \times X \dashrightarrow \mathcal{X}$.
	From this, by taking a $\comp^*$-equivariant resolution of indeterminacies, any two $\comp^*$-equivariant resolutions can be dominated by a third one.
	\par 
	We say that two test configurations are \textit{equivalent} if they are dominated by a third test configuration, so that the pull-backs of the line bundles of the two initial test configurations coincide. 
	According to \cite[Proposition 2.30]{BouckJonsGlobal}, every semiample test configuration is equivalent to a unique normal ample test configuration.
	By Zariski's main theorem, cf. \cite[Lemma A.4]{BouckJohn21}, equivalent normal test configurations produce the same filtrations on the section ring.
	\par 
	Let us now recall, following Phong-Sturm \cite{PhongSturmDirMA}, a construction of geodesic rays of metrics associated with an ample test configuration $\mathcal{T} = (\pi: \mathcal{X} \to \comp, \mathcal{L})$.
	\par 
	Consider the restriction $\pi': \mathcal{X}'_{\mathbb{D}} \to \mathbb{D}$ of a resolution of singularities $\mathcal{T}' := (\pi': \mathcal{X}' \to \comp, \mathcal{L}')$ of $\mathcal{T}$ to the unit disc $\mathbb{D}$ and denote $\mathcal{L}'_{\mathbb{D}} := \mathcal{L}'|_{\mathcal{X}'_{\mathbb{D}}}$.
	Phong-Sturm in \cite[Theorem 3]{PhongSturmDirMA} established that for any fixed smooth positive metric $h^L_0$ on $L$, there is a rotation-invariant bounded psh metric $h^{\mathcal{L}'}_{\mathbb{D}}$ over $\mathcal{L}'_{\mathbb{D}}$, verifying in the Bedford-Taylor sense, cf. \cite{BedfordTaylor}, the Monge-Ampère equation
	\begin{equation}\label{eq_ma_geod_dir}
		c_1(\mathcal{L}'_{\mathbb{D}}, h^{\mathcal{L}'}_{\mathbb{D}})^{n + 1} = 0,
	\end{equation}
	and such that its restriction over $\partial \mathcal{X}'_{\mathbb{D}}$ coincides with the rotation-invariant metric obtained from the fixed metric $h^L_0$ on $L$.
	Under the identification (\ref{eq_can_ident_test}), we then construct a ray $h^{\mathcal{T}}_t$, $t \in [0, + \infty[$, of metrics on $L$, such that $\hat{h}^{\mathcal{T}} = h^{\mathcal{L}'}_{\mathbb{D}}$ in the notations (\ref{eq_defn_hat_u}).
	Due to the equation (\ref{eq_ma_geod_dir}) and the description of the geodesic ray as in (\ref{eq_geod_as_env}), we see that the ray of metrics $h^{\mathcal{T}}_t$, $t \in [0, + \infty[$, is a geodesic ray emanating from $h^L_0$.
	This ray of metrics is only $\mathscr{C}^{1, 1}$ in general, see \cite{ChuTossVeinC11}.
	\par
	Recall that Phong-Sturm in \cite[Theorem 5]{PhongSturmRegul} established that there is a unique bounded psh solution to (\ref{eq_ma_geod_dir}).
	Since a pull-back of a solution (\ref{eq_ma_geod_dir}) from one resolution of singularities will be a solution on a dominating resolution of singularities, the geodesic ray $h^{\mathcal{T}}_t$, $t \in [0, + \infty[$, is independent of the choice of the $\comp^*$-equivariant resolution of singularities. Similarly, equivalent test configurations produce the same geodesic rays.
	As we shall see in Remark \ref{rem_ph_st_regul}, a result of Phong-Sturm \cite{PhongSturmDirMA} shows, moreover, that two ample test configurations produce the same geodesic ray of metrics \textit{if and only if} they are equivalent.

	\subsection{Convergence of spectral measures and maximal geodesic rays}\label{sect_spec_mes}
	The main goal of this section is to refine Theorem \ref{thm_dist_na} from distance convergence to convergence on the level of spectral measures and from submultiplicative filtrations associated with ample test configurations to general bounded submultiplicative filtrations.
	\par 
	We use below the notations from Theorem \ref{thm_dist_na}.
	For any $t \in [0, +\infty[$, we denote by $h^{\mathcal{T}_1 \mathcal{T}_2}_{t, s}$, $s \in [0, 1]$, the distinguished geodesic segment between $h^{\mathcal{T}_1}_t$ and $h^{\mathcal{T}_2}_t$.
	By the regularity result of Chu-Tosatti-Weinkove \cite{ChuTossVeinC11}, Chen \cite{ChenGeodMab} and Darvas \cite[Corollary 2]{DarvasMabCompl}, for any $t \in [0, +\infty[$, the path $h^{\mathcal{T}_1 \mathcal{T}_2}_{t, s}$, $s \in [0, 1]$, is $\mathscr{C}^{1, \overline{1}}$.
	In particular, by (\ref{eq_berndt_meas}), the following measure 
	\begin{equation}
		\mu_t^{\mathcal{T}_1 \mathcal{T}_2}
		:=
		\Big(\frac{1}{t} \frac{\partial}{\partial s} h^{\mathcal{T}_1 \mathcal{T}_2}_{t, s} \Big)_* \Big( c_1(L, h^{\mathcal{T}_1 \mathcal{T}_2}_{t, s})^n \Big),
	\end{equation}
	on $\real$ doesn't depend on $s \in [0, 1]$, as suggested by the notations.
	From (\ref{eq_bnd_darvas_sup}), this is a bounded measure.
	Moreover, by (\ref{eq_d_p_berndss}), the absolute moments of $\mu_t^{\mathcal{T}_1 \mathcal{T}_2}$ are related with $d_p$-distances as
	\begin{equation}\label{eq_abs_mom_1}
		\sqrt[p]{\int |x|^p \cdot d \mu_t^{\mathcal{T}_1 \mathcal{T}_2}(x)}
		=
		\frac{d_p(h^{\mathcal{T}_1}_t, h^{\mathcal{T}_2}_t)}{t}.
	\end{equation}
	\par 
	Now, on the algebraic side, for any $k \in \nat$, we denote by $e^k_1, \ldots, e^k_{N_k}$, $N_k := \dim H^0(X, L^k)$, the basis of $H^0(X, L^k)$, which jointly diagonalizes $\chi_{\mathcal{F}^{\mathcal{T}_1}_k}$ and $\chi_{\mathcal{F}^{\mathcal{T}_2}_k}$ as in (\ref{eq_sim_diag_nna}). 
	We define the sequence of probability measures $\mu^{\mathcal{F}^{\mathcal{T}_1} \mathcal{F}^{\mathcal{T}_2}}_{k}$, $k \in \nat^*$, $N_k \neq 0$, on $\real$ as follows
	\begin{equation}\label{eq_jump_meas_defn}
		\mu^{\mathcal{F}^{\mathcal{T}_1} \mathcal{F}^{\mathcal{T}_2}}_{k} 
		:= 
		\frac{1}{\dim H^0(X, L^k)} \sum_{i = 1}^{\dim H^0(X, L^k)} \delta \bigg[
		\frac{w_{\mathcal{F}^{\mathcal{T}_1}_k}(e^k_i) - w_{\mathcal{F}^{\mathcal{T}_2}_k}(e^k_i)}{k} \bigg], 
	\end{equation}
	where $\delta[x]$ is the Dirac mass at $x \in \real$. 
	Clearly, by the definition of $d_p$ from (\ref{eq_dp_filtr}), we have 
	\begin{equation}\label{eq_abs_mom_2}
		\sqrt[p]{\int |x|^p \cdot d \mu^{\mathcal{F}^{\mathcal{T}_1} \mathcal{F}^{\mathcal{T}_2}}_{k} (x)}
		=
		\frac{d_p(\mathcal{F}^{\mathcal{T}_1}_k, \mathcal{F}^{\mathcal{T}_2}_k)}{k}.
	\end{equation}
	\begin{thm}\label{thm_spec_meas}
		As $t \to \infty$ (resp. $k \to \infty$), the sequence of measures $\mu_t^{\mathcal{T}_1 \mathcal{T}_2}$ (resp. $\mu^{\mathcal{F}^{\mathcal{T}_1} \mathcal{F}^{\mathcal{T}_2}}_{k}$) converges weakly a bounded measure on $\real$.
		Moreover, the following identity holds
		\begin{equation}\label{eq_spec_meas1}
			\lim_{t \to \infty} \mu_t^{\mathcal{T}_1 \mathcal{T}_2}
			=
			\lim_{k \to \infty} \mu^{\mathcal{F}^{\mathcal{T}_1} \mathcal{F}^{\mathcal{T}_2}}_{k}.
		\end{equation} 
		Also, the following limits exist, they are finite, and we have
		\begin{equation}\label{eq_spec_meas2}
		\begin{aligned}
			&
			\lim_{t \to \infty} 
			\frac{1}{t} \max_{x \in X} \log \Big( \frac{h^{\mathcal{T}_2}_t(x)}{h^{\mathcal{T}_1}_t(x)} \Big)
			=
			\lim_{k \to \infty} 
			\frac{1}{k} \max_{i = 1, \ldots, N_k} \Big( w_{\mathcal{F}^{\mathcal{T}_1}_k}(e^k_i) - w_{\mathcal{F}^{\mathcal{T}_2}_k}(e^k_i) \Big),
			\\
			&
			\lim_{t \to \infty} 
			\frac{1}{t} \min_{x \in X} \log \Big( \frac{h^{\mathcal{T}_2}_t(x)}{h^{\mathcal{T}_1}_t(x)} \Big)
			=
			\lim_{k \to \infty} 
			\frac{1}{k} \min_{i = 1, \ldots, N_k} \Big( w_{\mathcal{F}^{\mathcal{T}_1}_k}(e^k_i) - w_{\mathcal{F}^{\mathcal{T}_2}_k}(e^k_i) \Big).
		\end{aligned}
		\end{equation}
	\end{thm}
	\begin{rem}\label{rem_spec_meas}
		a) The limiting measure on the left-hand side of (\ref{eq_spec_meas1}) is the chordal analogue of the probability measure constructed by Berndtsson \cite{BerndtProb} for geodesics.
		\par 
		b) The existence of the limit on the right-hand side of (\ref{eq_spec_meas1}) is due to Chen-Maclean \cite[Theorem 4.3]{ChenMaclean}, cf. also Boucksom-Jonsson \cite[Theorem 3.3 and \S 3.4]{BouckJohn21}.
		Our proof is independent of their result and it provides a different way of establishing the existence of the limiting measure.
		\par 
		c)
		When one of the test configurations is trivial (and hence the corresponding geodesic ray is constant), (\ref{eq_spec_meas1}) is due to Hisamoto \cite{HisamSpecMeas}. Witt Nystr{\"o}m \cite[Theorems 1.1 and 1.4]{NystOkounTest} also previously established (\ref{eq_spec_meas1}) under an additional assumption that the second test configuration is a product test configuration.
	\end{rem}
	To establish Theorem \ref{thm_spec_meas}, we need the following statement. We defer its proof to Section \ref{sect_part1_pf}.
	\begin{lem}\label{lem_gr_geod_ray}
		For any ample test configuration $\mathcal{T}$, the associated geodesic ray grows at most exponentially. 
		In other words, there is $C > 0$, such that for any $t \in [0, +\infty[$, we have
		\begin{equation}\label{eq_ray_norms_decart21}
			d_{+ \infty}(h^{\mathcal{T}}_0, h^{\mathcal{T}}_t)
			\leq 
			C t.
		\end{equation}
		Moreover, one can take $C := \limsup_{k \geq 1} \frac{1}{k} \sup_{x \in H^0(X, L^k) \setminus \{0\}} |w_{\mathcal{F}_k^{\mathcal{T}}}(x)|$. The latter constant is finite since the filtration $\mathcal{F}^{\mathcal{T}}$ is of finite type.
	\end{lem}
	We also need to study how geodesic rays and filtrations change under the shift operator, defined for any test configuration $\mathcal{T} = (\pi: \mathcal{X} \to \comp, \mathcal{L})$ as  $\mathcal{T}[m] := (\pi: \mathcal{X} \to \comp, \mathcal{L} \otimes \mathscr{O}(m X_0))$.
	Directly from the definitions, for any $k \in \nat^*$, the weights of the restrictions $\mathcal{F}[m]_k$, $\mathcal{F}_k$ of $\mathcal{F}[m]$ and $\mathcal{F}$ to $H^0(X, L^k)$, are related as 
	\begin{equation}\label{eq_shift_test_1}
		w_{\mathcal{F}[m]_k} = w_{\mathcal{F}_k} + mk.
	\end{equation}
	Similarly, for ample $\mathcal{T}$ and a fixed smooth positive metric $h^L_0$ on $L$, the geodesic rays $h^{\mathcal{T}[m]}_t$, $h^{\mathcal{T}}_t$, $t \in [0, +\infty[$,  emanating from $h^L_0$ and associated with $\mathcal{T}[m]$, $\mathcal{T}$ are related as 
	\begin{equation}\label{eq_shift_test_2}
		h^{\mathcal{T}[m]}_t = \exp(- t m) h^{\mathcal{T}}_t,
	\end{equation}	
	where we implicitly identified $\mathcal{L} \otimes \mathscr{O}(m X_0)|_{\mathcal{X} \setminus X_0}$ with $\mathcal{L}|_{\mathcal{X} \setminus X_0}$ using the canonical trivialization of the line bundle $\mathscr{O}(m X_0)|_{\mathcal{X} \setminus X_0}$.
	\begin{proof}[Proof of Theorem \ref{thm_spec_meas}]
		Let us first assume that for any $t \in [0, +\infty[$, we have $h^{\mathcal{T}_1}_t \leq h^{\mathcal{T}_2}_t$ and for any $k \in \nat$, we have $w_{\mathcal{F}^{\mathcal{T}_1}_k} \geq w_{\mathcal{F}^{\mathcal{T}_2}_k}$.
		Then by Lemma \ref{lem_gr_geod_ray}, (\ref{eq_bnd_darvas_sup}) and the fact that the filtrations $\mathcal{F}^{\mathcal{T}_1}$, $\mathcal{F}^{\mathcal{T}_2}$ are of finite type, the measures $\mu_t^{\mathcal{T}_1 \mathcal{T}_2}$, $t \in [0, +\infty[$, and $\mu^{\mathcal{F}^{\mathcal{T}_1} \mathcal{F}^{\mathcal{T}_2}}_{k}$, $k \in \nat^*$, have support in a fixed compact interval in $[0, +\infty[$.
		The weak convergence of the sequence of measures $\mu_t^{\mathcal{T}_1 \mathcal{T}_2}$ (resp. $\mu^{\mathcal{F}^{\mathcal{T}_1} \mathcal{F}^{\mathcal{T}_2}}_{k}$) holds since by Theorem \ref{thm_dist_na} their absolute moments (which coincide with moments in this case) converge.
		Since these moments of the limiting measures coincide by Theorem \ref{thm_dist_na}, applied for $p \in \nat^*$, we deduce (\ref{eq_spec_meas1}).
		An easy verification shows that Theorem \ref{thm_dist_na} for $p = +\infty$ also gives us exactly the first identity from (\ref{eq_spec_meas2}).
		Similarly, if for any $t \in [0, +\infty[$, we have $h^{\mathcal{T}_1}_t \geq h^{\mathcal{T}_2}_t$, and for any $k \in \nat$, we have $w_{\mathcal{F}^{\mathcal{T}_1}_k} \leq w_{\mathcal{F}^{\mathcal{T}_2}_k}$, Theorem \ref{thm_dist_na} for $p = +\infty$ gives us exactly the second identity from (\ref{eq_spec_meas2}).
		\par 
		Now, let $\mathcal{T}_1$ and $\mathcal{T}_2$ be arbitrary ample test configurations.
		Directly from the description (\ref{eq_geod_as_env}), we obtain that for any $m \in \nat$, we have  $h^{\mathcal{T}_1 \mathcal{T}_2[m]}_{t, s} = \exp(- s m t) h^{\mathcal{T}_1 \mathcal{T}_2}_{t, s}$, $s \in [0, 1]$, $t \in [0, +\infty[$. 
		Hence, for any $t \in [0, +\infty[$, we have $\mu_t^{\mathcal{T}_1 \mathcal{T}_2[m]} = S[m]_* \mu_t^{\mathcal{T}_1 \mathcal{T}_2}$, where $S[m] : \real \to \real$, $x \mapsto x - m$ is the shift operator.
		Similarly, by (\ref{eq_shift_test_1}), we have $\mu^{\mathcal{F}^{\mathcal{T}_1} \mathcal{F}^{\mathcal{T}_2[m]}}_{k} = S[m]_* \mu^{\mathcal{F}^{\mathcal{T}_1} \mathcal{F}^{\mathcal{T}_2}}_{k}$.
		From this, we obtain that Theorem \ref{thm_spec_meas} holds for $\mathcal{T}_1$ and $\mathcal{T}_2$ if and only if it holds for $\mathcal{T}_1$ and $\mathcal{T}_2[m]$ for some $m \in \nat$.
		\par 
		But from the boundness of the filtrations $\mathcal{F}^{\mathcal{T}_1}$, $\mathcal{F}^{\mathcal{T}_2}$ and from Lemma \ref{lem_gr_geod_ray}, we can always make $h^{\mathcal{T}_1}_t \leq h^{\mathcal{T}_2[m]}_t$, $t \in [0, +\infty[$, and $w_{\mathcal{F}^{\mathcal{T}_1}_k} \geq w_{\mathcal{F}^{\mathcal{T}_2[m]}_k}$, $k \in \nat$, by taking $m$ sufficiently big.
		Similarly, by making $m$ sufficiently small, the opposite inequalities will be satisfied.
		As described above, this implies Theorem \ref{thm_spec_meas}.
	\end{proof}
	\par 
	We will now show that Theorem \ref{thm_dist_na} can be used to study arbitrary bounded submultiplicative filtrartions $\mathcal{F}$ on $R(X, L)$.
	\par 
	To explain this, we will first need to explain that to any bounded submultiplicative filtration one can naturally associate a geodesic ray.
	For this, for simplicity, we assume that $\mathcal{F}$ is a $\mathbb{Z}$-filtration.
 	Following Székelyhidi \cite{SzekeTestConf}, recall that for any given bounded submultiplicative $\mathbb{Z}$-filtration $\mathcal{F}$ on $R(X, L)$, and any $k \in \nat^*$ big enough so that $H^0(X, L^k)$ generates $R(X, L^k)$, we can define a sequence of canonical apprixomations, $\mathcal{F}(k)$, which are filtrations of finite type on $R(X, L^k)$ generated by the restriction of $\mathcal{F}$ to $H^0(X, L^k)$.
 	By Rees correspondence, cf. Section \ref{sect_filt}, for any $k \in \nat^*$, there is $d_k \in \nat^*$, divisible by $k$, and an ample test configuration $\mathcal{T}(k) := (\pi_k : \mathcal{X}_k \to \comp, \mathcal{L}_k)$ of $L^{d_k}$, so that the restriction of $\mathcal{F}(k)$ to $R(X, L^{d_k}) \subset R(X, L^k)$ coincides with the filtration associated with $\mathcal{T}(k)$. 
 	We denote by $h_t^{\mathcal{F}(k)}$, $t \in [0, +\infty[$, the geodesic ray on $L$ emanating from $h^L_0$, defined $h_t^{\mathcal{F}(k)} := (h^{\mathcal{T}(k)}_t)^{\frac{1}{d_k}}$, where $h^{\mathcal{T}(k)}_t$ is the geodesic ray on $L^{d_k}$ associated with $\mathcal{T}(k)$ and emanating from $(h^L_0)^{d_k}$.
 	\par 	
 	The following result was established in \cite[Theorems 5.5, 5.7 and 5.10]{FinNarSim} using the works of Berman-Boucksom-Jonsson \cite{BerBouckJonYTD} and Phong-Sturm \cite{PhongSturmRegul}.
 	\begin{prop}\label{prop_filtr_geod_ray}
 		For any $t \in [0, +\infty[$, the sequence of metrics $h_t^{\mathcal{F}(k)}$, is uniformly bounded over $k \in \nat^*$. When restricted over multiplicative subsequences of $\nat^*$ (as for example $k = 2^l$, $l \in \nat^*$), for any $t \in [0, +\infty[$, the sequence of metrics $h_t^{\mathcal{F}(k)}$ is decreasing, and the ray of metrics $h_t^{\mathcal{F}} := (\lim_{k \to \infty} h_t^{\mathcal{F}(k)})_*$ is a geodesic ray departing from $h^L_0$.
 	\end{prop}
 	We can now state the following generalization of Theorem \ref{thm_dist_na}.
 	\begin{thm}\label{cor_dist_na_sm}
 		For any bounded submultiplicative filtrations $\mathcal{F}_1, \mathcal{F}_2$ on $R(X, L)$ and any $p \in [1, +\infty[$, the following identity holds
		\begin{equation}\label{eq_dist_na_sm}
			d_p \big(\{ h_t^{\mathcal{F}_1} \}, \{ h_t^{\mathcal{F}_2} \} \big)
			=
			d_p \big( \mathcal{F}_1, \mathcal{F}_2 \big).
		\end{equation}
 	\end{thm}
 	 \begin{rem}\label{rem_dist_na_sm}
 		a) Theorem \ref{cor_dist_na_sm} responds to a question from Zhang \cite[Remark 6.12]{KeweiZhangVal}.
 		\par 
 		b) From the proof of Theorem \ref{thm_spec_meas}, we see that the analogue of (\ref{eq_spec_meas1}) holds if the geodesic rays $h_t^{\mathcal{F}_1}$, $h_t^{\mathcal{F}_2}$ are $\mathscr{C}^{1, \overline{1}}$.
 		For general bounded submultiplicative filtrations, this regularity cannot be expected by \cite[Proposition 3.9 and Remark 3.10]{FinNarSim}.
 	\end{rem}
 	\par 
 	To prove Theorem \ref{cor_dist_na_sm}, we need to use a result of Boucksom-Jonsson \cite[Theorem 3.18]{BouckJohn21} stating that finite-type approximations of submultiplicative filtrations are continuous with respect to the $d_p$-metrics for $p \in [1, +\infty[$. For an alternative proof of this result, see \cite[Theorem 5.6]{FinNarSim}.
 	\begin{prop}\label{prop_finite_type_approx}
 		For any $p \in [1, +\infty[$, we have $\lim_{k \to \infty} \widetilde{d}_p(\mathcal{F}(k), \mathcal{F}) = 0$.
 	\end{prop}
 	As another ingredient in the proof of Theorem \ref{cor_dist_na_sm}, we need to show that Theorem \ref{thm_dist_na} can be used to give an algebraic formula for chordal distances between \textit{maximal geodesic rays}.
	\par 
	To describe this, we say that a sequence of geodesic rays $h^L_{i, t}$, $i \in \nat$, $t \in [0, +\infty[$, of bounded metrics on $L$ approximate $h^L_t$ from below (resp. approximate $h^L_t$ almost everywhere from above) if for any $t \in [0, +\infty[$, $i \leq j$, $h^L_{i, t} \leq h^L_{j, t}$ (resp. $h^L_{i, t} \geq h^L_{j, t}$) and $\lim_{i \to \infty} h^L_{i, t} = h^L_t$ (resp. $(\lim_{i \to \infty} h^{\mathcal{T}_i}_t)_* = h^L_t$).
	\par 
	Also, for bounded submultiplicative filtrations $\mathcal{F}_i$, $i = 0, 1$, on $R(X, L^{d_i})$, $d_i \in \nat^*$ and any $p \in [1, +\infty[$, we denote $\widetilde{d}_p(\mathcal{F}_0, \mathcal{F}_1) := \frac{1}{d_0 d_1} d_p(\mathcal{F}_0|_{R(X, L^{d_0 d_1})}, \mathcal{F}_1|_{R(X, L^{d_0 d_1})})$.
	\par 
	Recall that for a geodesic ray of Hermitian metrics $h^L_t \in \mathcal{E}^1$, $t \in [0, +\infty[$, one can define its \textit{non-Archimedean potential} (which is a function on the Berkovich analytification of $X$) by studying the singularities of the ray at $t = +\infty$, see \cite[\S B.6]{BerBouckJonYTD}.
	Following Berman-Boucksom-Jonsson \cite[Definition 6.5]{BerBouckJonYTD}, we say that a geodesic ray of Hermitian metrics $h^L_t \in \mathcal{E}^1$, $t \in [0, +\infty[$, is \textit{maximal}, if its potential is maximal among all geodesic rays departing from the same initial point and having the same non-Archimedean potential.
	Alternatively, according to \cite[proof of Theorem 6.6]{BerBouckJonYTD}, a geodesic ray $h^L_t$, $t \in [0, +\infty[$, is maximal if and only if there is a sequence of ample test configurations $\mathcal{T}_i$, $i \in \nat$, of $(X, L)$, such that the associated geodesic rays $h^{\mathcal{T}_i}_t$, $t \in [0, +\infty[$, approximate from below $h^L_t$.
	\begin{rem}\label{rem_max_hf}
 		It was established in \cite[Theorems 1.11]{FinNarSim} that for any bounded submultiplicative filtration $\mathcal{F}$ on $R(X, L)$, the ray $\{ h_t^{\mathcal{F}} \}$ from Proposition \ref{prop_filtr_geod_ray} is maximal, and it corresponds to the non-Archimedean potential $FS(\mathcal{F})$ prescribed by $\mathcal{F}$ as in \cite[Definition 2.13]{BouckJohn21}.
 	\end{rem}
	\par 
	Now, we fix a smooth positive metric $h^L_0$ on $L$, $p \in [1, + \infty[$, and consider the set $\mathcal{R}^p$ of geodesic rays $\{ h^L_t \}$, $h^L_t \in \mathcal{E}^p$, $t \in [0, +\infty[$, departing from $h^L_0$.
	We denote by $\mathcal{R}^p_{\max} \subset \mathcal{R}^p$ the subset of maximal geodesic rays.
	By \cite[Example 6.10]{BerBouckJonYTD}, the set $\mathcal{R}^p \setminus \mathcal{R}^p_{\max}$ is not empty.
	\par 
	\begin{prop}\label{cor_geod_rays}
		We fix $p \in [1, + \infty[$, and consider two maximal geodesic rays $\{ h^{L, i}_t \} \in \mathcal{R}^p_{\max}$, $i = 0, 1$.
		Let $\mathcal{T}^i_j$, $j \in \nat$, be two sequences of ample test configurations of $(X, L^{r^i_j})$, $r^i_j \in \nat^*$, such that the geodesic rays $\{ h^{L, i}_{j, t} \} := \{ (h^{\mathcal{T}^i_j}_t)^{\frac{1}{r^i_j}} \}$, approximate from below (or almost everywhere from above) the rays $\{ h^{L, i}_t \}$.
		Then
		\begin{equation}\label{eq_dist_na_max}
			d_p \big(\{ h^{L, 0}_t \}, \{ h^{L, 1}_t \} \big)
			=
			\lim_{j \to \infty} \widetilde{d}_p \big( \mathcal{F}^{\mathcal{T}^0_j}, \mathcal{F}^{\mathcal{T}^1_j} \big).
		\end{equation}
	\end{prop}
	\begin{proof}
		By Theorem \ref{thm_dist_na}, it suffices to establish that $\lim_{j \to \infty} d_p \big(\{ h^{L, i}_t \}, \{ h^{L, i}_{j, t} \} \big) = 0$.
		This is a direct consequence of \cite[Lemma 4.3]{DarLuGeod}, saying that this holds for any sequences of geodesic rays approximating a given geodesic ray from below (resp. almost everywhere from above).
	\end{proof}
 	\begin{rem}\label{rem_geod_rays}
 		a) When $p = 1$ and one geodesic ray is trivial, Proposition \ref{cor_geod_rays} reduces to \cite[Corollary 6.7]{BerBouckJonYTD}.
 		Using pluripotential theory, Reboulet in \cite[Theorem 4.1.1 and Remark 4.4.4]{ReboulGeodDeger} established that for $p = 1$, (\ref{eq_dist_na_max}) reduces to the case when one geodesic ray is trivial. 
 		\par 
 		b) Proposition \ref{cor_geod_rays} suggests that the right-hand side of (\ref{eq_dist_na_max}) is the analogue of $d_p$-distance between non-Archimedean potentials of the geodesic rays. 
 		It is interesting if one can get a formula for it in the spirit of (\ref{eq_d_p_berndss}), probably using the construction of the Monge-Ampère measure from \cite[\S 2.6]{BouckHisJon} and geodesics between non-Archimedean potentials from Reboulet \cite{ReboulGeod}.
 		\par 
 		c) For non-maximal geodesic rays, by \cite[Corollary 6.7]{BerBouckJonYTD} and \cite[Example 6.10]{BerBouckJonYTD}, there is no hope that the chordal distance can be expressed in terms of the non-Archimedean potentials of the geodesic rays. 
 		It would be interesting to understand the difference between the left-hand side and the right-hand side of (\ref{eq_dist_na_max}) in this case.
 	\end{rem}	
 	 \begin{proof}[Proof of Theorem \ref{cor_dist_na_sm}.]
 	 	Directly by Propositions \ref{prop_filtr_geod_ray},  \ref{cor_geod_rays}, and maximality of $\{ h_t^{\mathcal{F}_1} \}, \{ h_t^{\mathcal{F}_2} \}$, see Remark \ref{rem_max_hf}, we see that there is a sequence $d_k \in \nat^*$, $k \in \nat^*$, such that 
 	 	\begin{equation}
 	 		d_p \big(\{ h_t^{\mathcal{F}_1} \}, \{ h_t^{\mathcal{F}_2} \} \big)
			=
			\lim_{k \to \infty} \frac{1}{d_k} d_p \big( \mathcal{F}_1(k)|_{R(X, L^{d_k})}, \mathcal{F}_2(k)|_{R(X, L^{d_k})} \big),
 	 	\end{equation}
 	 	where the limit is taken over a multiplicative subsequence of $\nat^*$ (as for example $k = 2^l$, $l \in \nat^*$).
 	 	A combination of this with Proposition \ref{prop_finite_type_approx} yields
 	 	\begin{equation}\label{eq_cor_dist_na_sm_1}
 	 		d_p \big(\{ h_t^{\mathcal{F}_1} \}, \{ h_t^{\mathcal{F}_2} \} \big)
			=
			\lim_{k \to \infty} \frac{1}{d_k} d_p \big( \mathcal{F}_1|_{R(X, L^{d_k})}, \mathcal{F}_2|_{R(X, L^{d_k})} \big),
 	 	\end{equation}
 	 	where the limit is again taken over a multiplicative subsequence of $\nat^*$.
 	 	It is only left now to apply the result of Chen-Maclean \cite[Theorem 4.3]{ChenMaclean}, cf. also Boucksom-Jonsson \cite[Theorem 3.3 and \S 3.4]{BouckJohn21}, saying that for the restrictions $\mathcal{F}_{1, k}, \mathcal{F}_{2, k}$ of $\mathcal{F}_1, \mathcal{F}_2$ to $H^0(X, L^k)$, the limit $\lim_{k \to \infty} \frac{1}{k} d_p ( \mathcal{F}_{1, k}, \mathcal{F}_{2, k} )$ exists.
 	 	In particular, the right-hand side of (\ref{eq_cor_dist_na_sm_1}) coincides with it.
 	\end{proof}

\section{Quantization, Buseman convexity and submultiplicative norms}\label{sect_first}
	To establish Theorem \ref{thm_dist_na}, we prove that the left-hand side of (\ref{eq_dist_na}) is not smaller than the right-hand side and then the opposite bound (we call these statements lower and upper bounds of Theorem \ref{thm_dist_na} later on).
	The main goal of this section is to establish the upper bound.
	\par 
	More precisely, in Section \ref{sect_geom_ray_fin_dim}, we study the geometry of the space of Hermitian norms and various constructions of rays of norms.
	In Section \ref{sec_fs_bdem}, we recall the definition of the Fubini-Study map, and then a statement from \cite{FinNarSim}, concerning its isometry properties.
	Finally, in Section \ref{sect_part1_pf}, by relying on this, we establish the upper bound of Theorem \ref{thm_dist_na}.

	\subsection{Geometry of geodesic rays on the space of norms}\label{sect_geom_ray_fin_dim}	
	The main goal of this section is to recall the relation between filtrations and Hermitian norms on a finitely dimensional vector space and then to discuss the metric properties of this correspondence.
	\par
	Let us first recall that it is possible to view the space of filtrations on a given finitely dimensional vector space as the boundary at the infinity of the space of Hermitian norms, where the latter space is interpreted in terms of geodesic rays.
	For this, for any filtration $\mathcal{F}$ on a finitely dimensional vector space $V$, we associate a ray of Hermitian norms.
	More precisely, we fix a Hermitian norm $H_V := \| \cdot \|_H$ on $V$ and consider an orthonormal basis $s_1, \ldots, s_n$, of $(V, H_V)$, adapted to the filtration $\mathcal{F}$, i.e. verifying $s_i \in \mathcal{F}^{e_{\mathcal{F}}(i)} V$, where $e_{\mathcal{F}}(i)$, $i = 1, \ldots, n$, are the jumping numbers of the filtration $\mathcal{F}$, defined in (\ref{eq_defn_jump_numb}).
	We define the ray of Hermitian norms $H_t^{\mathcal{F}} := \| \cdot \|_{t}^{\mathcal{F}}$, $t \in [0, +\infty[$, on $V$ by declaring the basis 
	\begin{equation}\label{eq_bas_st}
		(s_1^t, \ldots, s_n^t) := \big( e^{t e_{\mathcal{F}}(1)} s_1, \ldots, e^{t e_{\mathcal{F}}(n)} s_n \big),
	\end{equation}
	to be orthonormal with respect to $H_t^{\mathcal{F}}$.
	It is clear from (\ref{eq_dist_transf}) that $H_t^{\mathcal{F}}$ is a geodesic ray with respect to the metrics $d_p$, $p \in [1, +\infty]$.
	Moreover, for any $t, s \in [0, +\infty[$, $p \in [1, +\infty[$, we have
	\begin{equation}
		d_p(H_t^{\mathcal{F}}, H_s^{\mathcal{F}})
		=
		|t - s| 
		\cdot
		\sqrt[p]{\frac{\sum_{i = 1}^{\dim V} |e_{\mathcal{F}}(i)|^p}{\dim V}},
	\end{equation}
	and $d_{+ \infty}(H_t^{\mathcal{F}}, H_s^{\mathcal{F}}) = |t - s| \cdot \max |e_{\mathcal{F}}(i)|$.
	Since for $p \in ]1, +\infty[$, the space $(\mathcal{H}_V, d_p)$ is a uniquely geodesic space, this gives us a complete description of geodesic rays with respect to $d_p$.
	Hence, filtrations are in bijective correspondence with geodesic rays.
	\par 
	Remark the following relation between the non-Archimedean norm $\chi_{\mathcal{F}}$, defined in (\ref{eq_filtr_norm}), and the ray of norms $H_t^{\mathcal{F}}$, $t \in [0, +\infty[$: for any $f \in V$, we have
	\begin{equation}\label{eq_na_nm_interpol}
		\log \chi_{\mathcal{F}}(f) 
		=
		 \lim_{t \to +\infty} \frac{\log \| f \|_{t}^{\mathcal{F}}}{t}.
	\end{equation}
	Remark that (\ref{eq_na_nm_interpol}) can be though as the finitely-dimensional analogue of Theorem \ref{thm_dist_na}.
	\par 
	It is well-known, cf. \cite[(1-2)]{BouckICM}, that the correspondence (\ref{eq_na_nm_interpol}) respects the metric structures (\ref{eq_dist_transf}) and (\ref{eq_dp_filtr}).
	In other words, the geodesic rays $H_t^{\mathcal{F}_0}, H_t^{\mathcal{F}_1}$, $t \in [0, +\infty[$, associated with the filtrations $\mathcal{F}_0, \mathcal{F}_1$ and emanating from a fixed Hermitian norm $H_0$ for any $p \in [1, +\infty]$ verify
	\begin{equation}\label{eq_d_p_fil_norms_herm}
		d_p(\mathcal{F}_1, \mathcal{F}_2)
		=
		\lim_{t \to \infty} \frac{d_p(H_t^{\mathcal{F}_0}, H_t^{\mathcal{F}_1})}{t}.
	\end{equation}
	\par 
	It is also well-known, cf. \cite[Exercise 6.5.4]{BhatiaBook}, that the space of Hermitian norms endowed with $d_p$-distances is Buseman convex. 
	More precisely, for any $0 < s < t$, $p \in [1, +\infty]$, we have
	\begin{equation}\label{eq_toponogov}
		\frac{d_p(H_s^{\mathcal{F}_0}, H_s^{\mathcal{F}_1})}{s} \leq \frac{d_p(H_t^{\mathcal{F}_0}, H_t^{\mathcal{F}_1})}{t},
	\end{equation}
	which gives an alternative way to see that the limit in (\ref{eq_d_p_fil_norms_herm}) exists.
	\par 
	In this article, we sometimes deal with non-Hermitian norms. 
	Due to this, we will generalize the distances $d_p$, $p \in [1, +\infty]$, from (\ref{eq_dist_transf}), to this more broad context.
	\par 
	\begin{sloppypar}
	More precisely, let $N_i = \| \cdot \|_i$, $i = 0, 1$, be two norms on $V$.
	We define the \textit{logarithmic relative spectrum} of $N_0$ with respect to $N_1$ as a non-increasing sequence $\mu_j := \mu_j(N_0, N_1)$, $j = 1, \ldots, \dim V$, defined as follows
	\begin{equation}\label{eq_log_rel_spec}
		\mu_j
		:=
		\sup_{\substack{W \subset V \\ \dim W = j}} 
		\inf_{w \in W \setminus \{0\}} \log \frac{\| w \|_1}{\| w \|_0}.
	\end{equation}
	We then define for $p \in [1, +\infty[$, the following quantity
	\begin{equation}
		d_p(N_0, N_1) = \sqrt[p]{\frac{\sum_{i = 1}^{\dim V} |\mu_i|^p}{\dim V}},
	\end{equation}
	and we let $d_{+ \infty}(N_0, N_1) = \max |\mu_i|$.
	By (\ref{eq_dist_transf}), it coincides with our previous definition if both $N_1$ and $N_2$ are Hermitian. 
	Also, $d_{+ \infty}(N_0, N_1)$ is the \textit{multiplicative gap} between $N_0, N_1$, i.e. it is the minimal constant $C > 0$, such that $N_0 \leq \exp(C) \cdot N_1$ and $N_1 \leq \exp(C) \cdot N_0$.
	\par 
	Remark also that John ellipsoid theorem, cf. \cite[\S 3]{PisierBook}, says that for any normed vector space $(V, N_V)$, there is a Hermitian norm $H_V$ on $V$, verifying 
	\begin{equation}\label{eq_john_ellips}
		H_V \leq N_V \leq \sqrt{\dim V} \cdot H_V.
	\end{equation}
	From (\ref{eq_john_ellips}), the fact that $d_p$, $p \in [1, +\infty]$, satisfy triangle inequality when restricted to Hermitian norms, Minkowski inequality and the usual monotonicity properties of the logarithmic relative spectrum, cf. \cite[(2.10), (2.11)]{FinNarSim}, we deduce that for any norms $N_0, N_1, N_2$ on $V$, the following weak version of triangle inequality holds
	\begin{equation}\label{eq_tr_weak}
		d_p(N_0, N_2) \leq d_p(N_0, N_1) + d_p(N_1, N_2) + \log \dim V.
	\end{equation}
	\end{sloppypar}
	\par 
	Now, we fix a finitely-dimensional normed vector space $(V, N_V)$, $\| \cdot \|_V := N_V$ and a filtration $\mathcal{F}$ of $V$.
	We define the non-Archimedean norm $\chi_{\mathcal{F}}$ associated with $\mathcal{F}$ as in (\ref{eq_filtr_norm}).
	Following \cite[(2.18)]{FinNarSim}, we construct a ray of norms $N^{\mathcal{F}}_t := \| \cdot \|^{\mathcal{F}}_t$, $t \in [0, +\infty[$, emanating from $N_V$, as follows
	\begin{equation}\label{eq_ray_norm_defn0}
		\| f \|^{\mathcal{F}}_t :=
		\inf 
		\Big\{
			\sum
			\| f_i \|_V
			\cdot
			\chi_{\mathcal{F}}(f_i)^t
			\,
			:
			\,
			f = \sum f_i
		\Big\}.
	\end{equation}
	Remark that even if the initial norm $N_V$ is Hermitian, the above ray is certainly not.
	Let us, nevertheless, recall the following compatibility result between the two definitions of rays of norms.
	\begin{lem}[{ \cite[Lemma 2.8]{FinNarSim}}]\label{lem_two_norms_comp0}
		For any (resp. Hermitian) norm $N_V$ (resp. $H_V$) on $V$ and any $t \in [0, +\infty[$, the rays of norms $N_t^{\mathcal{F}}$ (resp. $H_t^{\mathcal{F}}$) associated with the filtration $\mathcal{F}$ as in (\ref{eq_ray_norm_defn0}) (resp. (\ref{eq_bas_st})) and emanating from $N_V$ (resp. $H_V$) are related as follows: for any $t \in [0, +\infty[$, we have
		\begin{equation}\label{eq_two_norms_comp01}
			d_{+ \infty}(N^{\mathcal{F}}_t, H^{\mathcal{F}}_t)
			\leq
			d_{+ \infty}(N_V, H_V)
			+
			\log \dim V.
		\end{equation}
	\end{lem}
	\begin{rem}\label{rem_ray_sm_anal_hlds}
		From (\ref{eq_d_p_fil_norms_herm}), (\ref{eq_tr_weak}) and (\ref{eq_two_norms_comp01}), the analogue of (\ref{eq_d_p_fil_norms_herm}) holds for the rays as in (\ref{eq_ray_norm_defn0}).
	\end{rem}
	\par
	Now, the essential reason for introducing the ray of norms (\ref{eq_ray_norm_defn0}) instead of (\ref{eq_bas_st}) is that it behaves better in comparison with (\ref{eq_bas_st}) when defined on a graded ring instead of a vector space.
	To explain this, we fix a graded ring $A$ and a graded filtration $\mathcal{F}$ on $A$.
	We assume that $\mathcal{F}$ is submultiplicative in the sense of (\ref{eq_subm_filt}). 
	We fix a graded norm $N = \sum N_k$, $N_k := \| \cdot \|_k$, over $A$, which we assume to be submultiplicative in the following sense: for any $k, l \in \nat^*$, $f \in A_k$, $g \in A_l$, we have
	\begin{equation}\label{eq_subm_s_ring}
		\| f \cdot g \|_{k + l} \leq 
		\| f \|_k \cdot
		\| g \|_l.
	\end{equation}
	A trivial verification shows that the following lemma holds.
	\begin{lem}[{\cite[\S 5.1]{FinNarSim}}]\label{lem_subm_ray_norms}
		The ray of norms $N^{\mathcal{F}}_t$, $t \in [0, +\infty[$, emanating from $N$ and constructed as in (\ref{eq_ray_norm_defn0}), is a ray of submultiplicative norms.
	\end{lem}

	\subsection{Fubini-Study metrics of submultiplicative norms}\label{sec_fs_bdem}
	\begin{sloppypar}
	In this article, we constantly pass from the study of graded norms on $R(X, L)$ to metrics on $L$.
	The fundamental tool for this is the Fubini-Study map.
	In this section, we recall its definition and its isometric properties.
	\par 
	We fix an ample line bundle $L$ over a compact complex manifold $X$.
	For $k_0 \in \nat$ so that $L^{k_0}$ is very ample, Fubini-Study map associates with any norm $N_k = \| \cdot \|_k$ on $H^0(X, L^k)$, $k \geq k_0$, a continuous metric $FS(N_k)$ on $L^k$, constructed as follows.
	Consider the Kodaira embedding 
	\begin{equation}\label{eq_kod}
		{\rm{Kod}}_k : X \hookrightarrow \mathbb{P}(H^0(X, L^k)^*).
	\end{equation}
	The evaluation maps provide the isomorphism $
		L^{-k} \to {\rm{Kod}}_k^* \mathscr{O}(-1),
	$
	where $\mathscr{O}(-1)$ is the \textit{tautological line bundle} over $\mathbb{P}(H^0(X, L^k)^*)$.
	We endow $H^0(X, L^k)^*$ with the dual norm $N_k^*$ and denote by $FS^{\mathbb{P}}(N_k)$ the induced metric on the \textit{hyperplane line bundle} $\mathscr{O}(1) := \mathscr{O}(-1)^*$ over $\mathbb{P}(H^0(X, L^k)^*)$. 
	We define the metric $FS(N_k)$ on $L^k$ as the only metric verifying under the dual of the above isomorphism the identity
	\begin{equation}\label{eq_fs_defn}
		FS(N_k) = {\rm{Kod}}_k^* ( FS^{\mathbb{P}}(N_k) ).
	\end{equation}
	A statement below can be seen as an alternative definition of $FS(N_k)$.
	\end{sloppypar}
	\begin{lem}\label{lem_fs_inf_d}
		For any $x \in X$, $l \in L^k_x$, the following identity takes place
		\begin{equation}\label{eq_fs_norm}
			|l|_{FS(N_k)}
			=
			\inf_{\substack{s \in H^0(X, L^k) \\ s(x) = l}}
			\| s \|_k.
		\end{equation}
	\end{lem}
	\begin{proof}
		An easy verification, cf. Ma-Marinescu \cite[Theorem 5.1.3]{MaHol}.
	\end{proof}
	When the norm $N_k$ comes from a Hermitian product on $H^0(X, L^k)$, the definition of the Fubini-Study map is standard, and explicit evaluation shows that in this case $c_1(\mathscr{O}(1) , FS^{\mathbb{P}}(N_k))$ coincides up to a positive constant with the Kähler form of the Fubini-Study metric on $\mathbb{P}(H^0(X, L^k)^*)$ induced by $N_k$.
	In particular, $c_1(\mathscr{O}(1) , FS^{\mathbb{P}}(N_k))$ is a positive $(1, 1)$-form.
	From Kobayashi \cite{KobFinNeg}, for general norms $N_k$, the $(1, 1)$-current $c_1(\mathscr{O}(1) , FS^{\mathbb{P}}(N_k))$ is positive, cf. \cite[\S 2.1]{FinNarSim} for details.
	In particular, the metric $FS(N_k)$ is positive for any norm $N_k$ on $H^0(X, L^k)$.
	\par 
	We will now study the properties of the Fubini-Study map on the space of graded norms.
	Let $N, N'$ be graded norms on the section ring $R(X, L)$.
	For $p \in [1, +\infty]$, we define 
	\begin{equation}\label{eq_dp_graded}
		d_p(N, N') := \limsup_{k \to \infty} \frac{d_p(N_k, N'_k)}{k},
	\end{equation}
	where $N_k, N'_k$ are the restrictions of $N, N'$ to $H^0(X, L^k)$.
	\par 
	A trivial verification, based on (\ref{eq_fs_defn}), shows that the Fubini-Study map is 1-Lipschitz with respect to the $d_{+ \infty}$-metric. 
	In other words, we have
	\begin{equation}\label{eq_fs_contact}
		d_{+ \infty}(FS(N_k), FS(N_k'))
		\leq
		d_{+ \infty}(N_k, N_k').
	\end{equation}
	For other $d_p$-metrics, $p \in [1, +\infty[$, no relation between the distances of graded norms and distances of their Fubini-Study metrics exists, see \cite[Proposition 3.7]{FinNarSim}.
	But from the work of the author \cite{FinNarSim}, we know that there is such a relation under an additional submultiplicativity assumption, (\ref{eq_subm_s_ring}).
	\par 
	More precisely, from Lemma \ref{lem_fs_inf_d}, it is easy to verify that the sequence of Fubini-Study metrics $FS(N_k)$, $k \geq k_0$, is submultiplicative for any submultiplicative graded norm $N = \sum N_k$ on $R(X, L)$.
	By this we mean that for any $k, l \geq k_0$, $FS(N_{k + l}) \leq FS(N_k) \cdot FS(N_l)$.
	In particular, by Fekete's lemma, the sequence of metrics $FS(N_k)^{\frac{1}{k}}$ on $L$ converges, as $k \to \infty$, to a (possibly only bounded from above and even null) upper semicontinuous metric, which we denote by $FS(N)$.
	We say that $N$ is \textit{bounded} if $FS(N)$ is bounded, and we denote by $FS(N)_*$ the lower semicontinuous regularization of $FS(N)$, which is psh, cf. \cite[Proposition I.4.24]{DemCompl}.
	\begin{thm}[{ \cite[Corollary 3.6]{FinNarSim} }]\label{thm_d_p_norm_fs_rel}
		For any bounded submultiplicative graded norms $N, N'$ on $R(X, L)$, and any $p \in [1, +\infty[$, we have
		\begin{equation}\label{eq_d_p_norm_fs_rel}
			d_p \big( FS(N)_*, FS(N')_* \big)
			=
			d_p(N, N').
		\end{equation}
		Moreover, we have $\lim$ instead of $\limsup$ in (\ref{eq_dp_graded}) in this case. If, moreover, $FS(N)$ and $FS(N')$ are continuous, then one can take $p = +\infty$ above.
	\end{thm}
	\par 
	Let us recall, finally, that a result of Tian \cite{TianBerg}, states that for smooth positive metrics $h^L_0$ on $L$, as $k \to \infty$, the following uniform convergence takes place
	\begin{equation}\label{eq_tian_conv}
		FS({\rm{Hilb}}_k(h^L_0))^{\frac{1}{k}} \to h^L_0.
	\end{equation}
	Directly from Lemma \ref{lem_fs_inf_d}, we see that (\ref{eq_tian_conv}) can be restated in the following way.
	For any $\epsilon > 0$, there is $k_0 \in \nat$, such that for any $k \geq k_0$, we have 
	\begin{equation}\label{eq_tian_dinf_norms}
		d_{+ \infty}\Big({\rm{Ban}}^{\infty}_k(h^L_0), {\rm{Hilb}}_k(h^L_0) \Big)
		\leq
		\epsilon k.
	\end{equation}
	More detailed analysis, using the fact that $h^L_0$ is positive and smooth, cf. Catlin \cite{Caltin}, Zelditch \cite{ZeldBerg}, Dai-Liu-Ma \cite{DaiLiuMa} and Ma-Marinescu \cite{MaHol}, shows that we can improve (\ref{eq_tian_dinf_norms}) by replacing the right-hand side by $(n + \epsilon) \log k$.

	\subsection{Isometry properties of the quantization scheme of Phong-Sturm}\label{sect_part1_pf}
	
	The main goal of this section is to establish the upper bound in Theorem \ref{thm_dist_na}.
	Our proof relies in an essential way on the fact that the finitely dimensional version of Theorem \ref{thm_dist_na} holds, see (\ref{eq_d_p_fil_norms_herm}).
	To pass from this finitely-dimensional picture to the infinitely-dimensional one of Theorem \ref{thm_dist_na}, we rely on the methods of geometric quantization using the quantization scheme introduced by Phong-Sturm for geodesic rays associated with test configurations.
	The central point is then to prove that this quantization scheme preserves distances in a reasonable sense.
	\par 
	More precisely, we fix an ample test configuration $\mathcal{T}$ of a polarized projective manifold $(X, L)$, and let $\mathcal{F}^{\mathcal{T}}_k$, $k \in \nat$, be the filtrations on the graded pieces $H^0(X, L^k)$ of the section ring $R(X, L)$ induced by the test configuration as in Section \ref{sect_filt}.
	We fix a smooth positive metric $h^L_0$ on $L$, and for any $t \in [0, +\infty[$, $k \in \nat$, we define, following Phong-Sturm \cite{PhongSturmTestGeodK}, $H^{\mathcal{T}}_{t, k}$ as the (geodesic) ray of Hermitian norms on $H^0(X, L^k)$ associated with the filtration $\mathcal{F}^{\mathcal{T}}_k$ and emanating from ${\rm{Hilb}}_k(h^L_0)$ as in (\ref{eq_bas_st}).
	We denote by $H^{\mathcal{T}}_t = \sum_{k = 0}^{\infty} H^{\mathcal{T}}_{t, k}$ the associated graded norm on $R(X, L)$.
	We denote by $h^{\mathcal{T}}_t$ the geodesic ray of metrics on $L$, constructed from the test configuration $\mathcal{T}$ as in Section \ref{sect_filt}.
	\par 
	The following result will be central in our approach to the upper bound in Theorem \ref{thm_dist_na}.
	\begin{thm}\label{thm_dp_ray_norms_herm}
		For any ample test configurations $\mathcal{T}_1, \mathcal{T}_2$ of a polarized projective manifold $(X, L)$ and any $t \in [0, + \infty[$, $p \in [1, +\infty[$, we have the following metric relation between the quantized geodesic rays of norms and geodesic rays of metrics 
		\begin{equation}\label{eq_dp_ray_norms_herm}
			d_p(h^{\mathcal{T}_1}_t, h^{\mathcal{T}_2}_t)
			=
			d_p ( 
				H^{\mathcal{T}_1}_t, H^{\mathcal{T}_2}_t
			),
		\end{equation}
		and we have $\lim$ instead of $\limsup$ in the definition (\ref{eq_dp_graded}) corresponding to the right-hand side of (\ref{eq_dp_ray_norms_herm}).
		Moreover, for $p = +\infty$, we have
		\begin{equation}
			d_{+ \infty}(h^{\mathcal{T}_1}_t, h^{\mathcal{T}_2}_t)
			\leq
			\liminf_{k \to \infty}
			\frac{d_{+ \infty} ( 
				H^{\mathcal{T}_1}_{t, k}, H^{\mathcal{T}_2}_{t, k}
			)}{k}
		\end{equation}
	\end{thm}
	
	To establish Theorem \ref{thm_dp_ray_norms_herm}, the following result of Phong-Sturm is indispensable.
	\begin{thm}[Phong-Sturm { \cite[Theorem 1]{PhongSturmTestGeodK} }]\label{thm_phong_sturm_limit}
		The geodesic ray $h^{\mathcal{T}}_t$, $t \in [0, + \infty[$, associated with the test configuration $\mathcal{T}$ is related to the geodesic ray of Hermitian norms $H^{\mathcal{T}}_t$ as follows
		\begin{equation}
			h^{\mathcal{T}}_t
			=
			\lim_{k \to \infty}
			\big(
			\inf_{l \geq k}
			FS(H^{\mathcal{T}}_{t, l})^{\frac{1}{l}}
			\big)_*.
		\end{equation}
	\end{thm}
	\begin{proof}[Proof of Lemma \ref{lem_gr_geod_ray}]
		By the definition of the geodesic ray of norms $H_{t, k}^{\mathcal{F}}$ and the boundness of the filtration associated with an ample test configuration, we conclude that there is $C > 0$, such that for any $k \in \nat^*$, $t \in [0, +\infty[$, we have
		\begin{equation}\label{eq_ray_norms_decart}
			d_{+ \infty}(H_{t, k}^{\mathcal{F}}, H_{0, k}^{\mathcal{F}})
			\leq 
			C t k.
		\end{equation}
		From the second part of Theorem \ref{thm_dp_ray_norms_herm} and (\ref{eq_ray_norms_decart}), we deduce Lemma \ref{lem_gr_geod_ray}.
	\end{proof}
	\begin{proof}[Proof of Theorem \ref{thm_dp_ray_norms_herm}]
		The main idea of the proof is to replace the rays of norms $H^{\mathcal{T}_1}_t, H^{\mathcal{T}_2}_t$, $t \in [0, +\infty[$, by their submultiplicative analogues (\ref{eq_ray_norm_defn0}), to which we can apply Theorem \ref{thm_d_p_norm_fs_rel}.
		\par 
		More precisely, for an ample test configuration $\mathcal{T}$ of a polarized pair $(X, L)$, we denote by $N^{\mathcal{T}}_{t, k}$, $t \in [1, +\infty[$, the ray of norms emanating from ${\rm{Ban}}^{\infty}_k(h^L_0)$ associated with $\mathcal{F}^{\mathcal{T}}_k$ as in (\ref{eq_ray_norm_defn0}).
		We denote by $N^{\mathcal{T}}_t = \sum_{k = 0}^{\infty} N^{\mathcal{T}}_{t, k}$ the associated graded ray of norms on $R(X, L)$.
		\par 
		The crucial point about the graded norms $N^{\mathcal{T}}_t$, $t \in [0, +\infty[$, is that they are submultiplicative in the sense (\ref{eq_subm_s_ring}).
		This follows from Lemma \ref{lem_subm_ray_norms}, the fact that the norm ${\rm{Ban}}^{\infty}(h^L_0)$ is submultiplicative and the fact that the filtration $\mathcal{F}^{\mathcal{T}}$ is submultiplicative, see Section \ref{sect_filt}.
		\par 
		Since the filtration $\mathcal{F}^{\mathcal{T}}$ is bounded, we see that $N^{\mathcal{T}}_{t}$, $t \in [0, +\infty[$, is bounded in the sense described before Theorem \ref{thm_d_p_norm_fs_rel}, cf. \cite[after (5.10)]{FinNarSim}.
		We conclude by Theorem \ref{thm_d_p_norm_fs_rel} that for any $p \in [1, +\infty[$, the following relation holds
		\begin{equation}\label{eq_thm_dp_ray_norms_herm0}
			d_p \big( FS(N^{\mathcal{T}_1}_{t})_*, FS(N^{\mathcal{T}_2}_{t})_* \big)
			=
			d_p \big(N^{\mathcal{T}_1}_{t}, N^{\mathcal{T}_2}_{t} \big),
		\end{equation}
		and we have $\lim$ instead of $\limsup$ in the definition (\ref{eq_dp_graded}) of the right-hand side of (\ref{eq_thm_dp_ray_norms_herm0}).
		\par 
		\begin{sloppypar}
		By (\ref{eq_fs_contact}), we have $d_{+ \infty} ( FS(N^{\mathcal{T}_1}_{t, k}), FS(N^{\mathcal{T}_2}_{t, k}) ) \leq d_{+ \infty} (N^{\mathcal{T}_1}_{t, k}, N^{\mathcal{T}_2}_{t, k} ).$
		Remark that $d_{+ \infty}$-distance is lower semicontinuous with respect to the pointwise convergence, i.e. for a sequence of metrics $h^L_{1, l}$, $h^L_{2, l}$, $l \in \nat$, on $L$ converging pointwise to some bounded metrics $h^L_1$, $h^L_2$, we have $d_{+ \infty}(h^L_1, h^L_2) \leq \liminf_{l \to \infty} d_{+ \infty}(h^L_{1, l}, h^L_{2, l})$.
		Also, lower-semicontinuous regularization is 1-Lipschitz with respect to $d_{+ \infty}$-distance, i.e. for any bounded metrics $h^L_1, h^L_2$ on $L$, we have $d_{+ \infty}(h^L_{1 *}, h^L_{2 *}) \leq d_{+ \infty}(h^L_1, h^L_2)$.
		From all these observations, we conclude
		\begin{equation}
			d_{+ \infty} \big( FS(N^{\mathcal{T}_1}_{t})_*, FS(N^{\mathcal{T}_2}_{t})_* \big)
			\leq
			\liminf_{k \to \infty} \frac{d_{+ \infty} (N^{\mathcal{T}_1}_{t, k}, N^{\mathcal{T}_2}_{t, k} )}{k}.
		\end{equation}
		\end{sloppypar}
		\par 
		\begin{sloppypar}
		Now, it is only left to relate the rays of norms $N^{\mathcal{T}_i}_t$, $t \in [0, +\infty[$, $i = 1, 2$, to $H^{\mathcal{T}_i}_t$, and the rays of metrics $FS(N^{\mathcal{T}_i}_{t, k})^{\frac{1}{k}}$, $t \in [0, +\infty[$, $k \in \nat^*$, to $FS(H^{\mathcal{T}_i}_{t, k})^{\frac{1}{k}}$.
		For this, by Lemma \ref{lem_two_norms_comp0}, (\ref{eq_tr_weak}) and (\ref{eq_tian_dinf_norms}), for any $p \in [1, +\infty]$, we have
		\begin{equation}\label{eq_thm_dp_ray_norms_herm1}
			\liminf_{k \to \infty} \frac{d_p (N^{\mathcal{T}_1}_{t, k}, N^{\mathcal{T}_2}_{t, k} )}{k}
			=
			\liminf_{k \to \infty} \frac{d_p (H^{\mathcal{T}_1}_{t, k}, H^{\mathcal{T}_2}_{t, k} )}{k}.
		\end{equation}
		\end{sloppypar}
		Remark also that the sequence of metrics $FS(N^{\mathcal{T}_i}_{t, k})^{\frac{1}{k}}$, $i = 1, 2$, $k \in \nat^*$, is submultiplicative by the discussion before Theorem \ref{thm_d_p_norm_fs_rel}, and, hence, by Fekete's lemma, its limit $FS(N^{\mathcal{T}_i}_{t})$ coincides with the infimum of $FS(N^{\mathcal{T}_i}_{t, k})^{\frac{1}{k}}$, $k \in \nat^*$.
		From this and Lemma \ref{lem_two_norms_comp0}, we obtain
		\begin{equation}\label{eq_thm_dp_ray_norms_herm2}
			FS(N^{\mathcal{T}_i}_{t})_*
			=
			\lim_{k \to \infty}
			\big(
			\inf_{l \geq k}
			FS(H^{\mathcal{T}_i}_{t, l})^{\frac{1}{l}}
			\big)_*.
		\end{equation}
		We conclude by Theorem \ref{thm_phong_sturm_limit}, (\ref{eq_thm_dp_ray_norms_herm0}), (\ref{eq_thm_dp_ray_norms_herm1}) and (\ref{eq_thm_dp_ray_norms_herm2}). 
	\end{proof}
	Now, we have everything ready to prove a part of Theorem \ref{thm_dist_na}.
	\begin{proof}[Proof of the upper bound of Theorem \ref{thm_dist_na}.]
		First of all, for any $t \in [0, + \infty[$, $p \in [1, +\infty]$, $k \in \nat$, by the finitely-dimensional analogue (\ref{eq_d_p_fil_norms_herm}) of Theorem \ref{thm_dist_na}, we have 
		\begin{equation}\label{eq_thm_dist_na_0}
			\frac{d_p \big( 
				H^{\mathcal{T}_1}_{t, k}, H^{\mathcal{T}_2}_{t, k}
			\big)}{t}
			\leq
			d_p(\mathcal{F}^{\mathcal{T}_1}_k, \mathcal{F}^{\mathcal{T}_2}_k).
		\end{equation}
		We now divide both sides of (\ref{eq_thm_dist_na_0}) by $k$, take the limit $k \to \infty$ and use Theorem \ref{thm_dp_ray_norms_herm} along with (\ref{eq_spec_dist}) to conclude that we have
		\begin{equation}\label{eq_thm_dist_na_1}
			\frac{d_p(h^{\mathcal{T}_1}_t, h^{\mathcal{T}_2}_t)}{t}
			\leq
			\liminf_{k \to \infty} \frac{ d_p(\mathcal{F}^{\mathcal{T}_1}_k, \mathcal{F}^{\mathcal{T}_2}_k)}{k}.
		\end{equation}
		By taking now limit $t \to \infty$ in (\ref{eq_thm_dist_na_1}), we obtain the upper bound of Theorem \ref{thm_dist_na}.
	\end{proof}

\section{Uniform submultiplicativity, Toeplitz operators and snc models}	\label{sect_second}
	The main goal of this section is to prove that the left-hand side of (\ref{eq_dist_na}) is not smaller than the right-hand side, i.e. that the lower bound of Theorem \ref{thm_dist_na} holds.
	This with the fact that we already established the opposite bound in Section \ref{sect_part1_pf} would give us a complete proof of Theorem \ref{thm_dist_na}.
	\par 
	Similarly to the proof of the upper bound from Section \ref{sect_part1_pf}, the proof here relies on the geometric quantization procedure of Phong-Sturm. 
	But otherwise it is rather different.
	\par 
	We first make a comparison between the geodesic ray of norms on the section ring, introduced before Theorem \ref{thm_dp_ray_norms_herm}, and the ray of $L^2$-norms of the geodesic ray of metrics associated with the test configurations as introduced in Section \ref{sect_filt}.
	We then show that it is sufficient to assume that the singularities of the central fibers of test configurations are mild enough.
	Finally, for test configurations with mild singularities, we estimate the distance between the $L^2$-norms associated with the geodesic rays of metrics in terms of the distance between the geodesic rays of metrics themselves.
	Combining all these results with a result from Section \ref{sect_geom_ray_fin_dim}, saying that the finitely-dimensional analogue of Theorem \ref{thm_dist_na} holds, leads to a proof of the lower bound from Theorem \ref{thm_dist_na}.
	\par 
 	More precisely, recall that before Theorem \ref{thm_dp_ray_norms_herm} we defined, following Phong-Sturm, a geodesic ray of graded Hermitian norms $H^{\mathcal{T}}_t = \sum_{k = 0}^{\infty} H^{\mathcal{T}}_{t, k}$, $t \in [0, +\infty[$, on $R(X, L)$ associated with a test configuration $\mathcal{T}$.
	Let $h^{\mathcal{T}}_t$, $t \in [0, +\infty[$, be the geodesic ray of metrics on $L$ associated with $\mathcal{T}$ as in Section \ref{sect_filt}.
	Let $\omega$ be a Kähler form on $X$.
	In Sections \ref{sect_quot}, \ref{sect_ohs}, by relying on the results of Phong-Sturm \cite{PhongSturmRegul} and the methods from the previous works of the author, \cite{FinSecRing}, \cite{FinNarSim}, we establish the following result, relating rays $H^{\mathcal{T}}_t$ and $h^{\mathcal{T}}_t$, $t \in [0, + \infty[$.
	\begin{thm}\label{thm_2_step}
		There are $C > 0$, $k_0 \in \nat^*$, such that for any $t \in [0, +\infty[$, $k \geq k_0$, the ray of norms $H^{\mathcal{T}}_t$ compares to the $L^2$-norms associated with the geodesic ray of metrics $h^{\mathcal{T}}_t$ as follows
		\begin{equation}
			d_{+ \infty} \big(
			H^{\mathcal{T}}_{t, k},
			{\rm{Hilb}}_k(h^{\mathcal{T}}_t, \omega)
			\big)
			\leq
			C(t + k).
		\end{equation}
		Similarly, for $L^{\infty}$-norms, we have
		\begin{equation}
			d_{+ \infty} \big(
			H^{\mathcal{T}}_{t, k},
			{\rm{Ban}}_k^{\infty}(h^{\mathcal{T}}_t)
			\big)
			\leq
			C(t + k).
		\end{equation}
	\end{thm}
 	\par 
 	Now, we say that a proper holomorphic map  $\pi : \mathcal{X} \to \mathbb{C}$ (or $\pi : \mathcal{X} \to \mathbb{D}$) is a \textit{snc model} if $\mathcal{X}$ is smooth, the central fiber $X_0$ is a simple normal crossing divisor in $\mathcal{X}$, and the intersections of irreducible components of $X_0$ are either irreducible or empty.
	If, furthermore, the central fiber is reduced, we say that it is a \textit{semistable snc model}.
	When $\mathcal{X}$ is endowed with an ample line bundle $\mathcal{L}$, the pair $(\pi, \mathcal{L})$ is called an \textit{ample semistable snc model}.
	In Section \ref{sect_nat_oper}, by relying on the results of Phong-Sturm \cite{PhongSturmRegul} and Boucksom-Jonsson \cite{BouckJohn21}, we establish the following result.
	\begin{thm}\label{thm_1_step}
		In order to prove Theorem \ref{thm_dist_na}, it suffices to establish it for $\mathcal{T}_1 = (\pi: \mathcal{X} \to \comp, \mathcal{L}_1)$ and $\mathcal{T}_2 = (\pi: \mathcal{X} \to \comp, \mathcal{L}_2)$, where $(\pi, \mathcal{L}_1)$, $(\pi, \mathcal{L}_2)$ are ample semistable snc models.
	\end{thm}
	\par 
	Let us now fix two test configurations $\mathcal{T}_1, \mathcal{T}_2$ as in Theorem \ref{thm_1_step}.
	We fix a smooth positive metric $h^L_0$ on $L$, and let $h^{\mathcal{T}_1}_t, h^{\mathcal{T}_2}_t$, $t \in [0, +\infty[$, be the geodesic rays of metrics on $L$ associated with $\mathcal{T}_1, \mathcal{T}_2$ and emanating from $h^L_0$.
	In Sections \ref{sect_dist_snc}, \ref{sect_toepl}, by relying on the methods of Dai-Liu-Ma \cite{DaiLiuMa}, Ma-Marinescu \cite{MaHol}, \cite{MaMarToepl}, Darvas-Lu-Rubinstein \cite{DarvLuRub}, and the results of Berndtsson \cite{BerndtProb}, we establish the following result.
	\begin{thm}\label{thm_3_step}
		For any $\epsilon > 0$, $p \in [1, +\infty[$, there are $C > 0, k_0 \in \nat$, such that for any $t \in [0, +\infty[$, $k \geq k_0$, the following bound holds
		\begin{equation}\label{eq_3_step}
			d_p \big( {\rm{Hilb}}_k(h^{\mathcal{T}_1}_t, \omega), {\rm{Hilb}}_k(h^{\mathcal{T}_2}_t, \omega) \big)
			\leq
			k \cdot d_p(h^{\mathcal{T}_1}_t, h^{\mathcal{T}_2}_t)
			+
			C(k + t)
			+
			\epsilon k t.
		\end{equation}
		Moreover, for $p = +\infty$, we have
		\begin{equation}
			d_{+\infty} \big( {\rm{Hilb}}_k(h^{\mathcal{T}_1}_t, \omega), {\rm{Hilb}}_k(h^{\mathcal{T}_2}_t, \omega) \big)
			\leq
			k \cdot d_{+\infty}(h^{\mathcal{T}_1}_t, h^{\mathcal{T}_2}_t).
		\end{equation}
	\end{thm}
	We will now show how to assemble these results to finally establish Theorem \ref{thm_dist_na}.
	\begin{proof}[Proof of the lower bound of Theorem \ref{thm_dist_na}]
		By Theorem \ref{thm_1_step}, without loss of generality, we may assume that $\mathcal{T}_1 = (\pi: \mathcal{X} \to \comp, \mathcal{L}_1)$ and $\mathcal{T}_2 = (\pi: \mathcal{X} \to \comp, \mathcal{L}_2)$, where $(\pi, \mathcal{L}_1)$, $(\pi, \mathcal{L}_2)$ are ample semistable snc models.
		We use the notations introduced before Theorem \ref{thm_3_step}.
		From Theorems \ref{thm_2_step}, \ref{thm_3_step} and (\ref{eq_tr_weak}), we conclude that for any $\epsilon > 0$, $p \in [1, +\infty]$, there are $C > 0$, $k_0 \in \nat^*$, such that for any $t \in [0, +\infty[$ and $k \geq k_0$, we have
		\begin{equation}\label{eq_pf_2_0}
			d_p \big( H^{\mathcal{T}_1}_{t, k}, H^{\mathcal{T}_2}_{t, k} \big)
			\leq
			k \cdot d_p(h^{\mathcal{T}_1}_t, h^{\mathcal{T}_2}_t)
			+
			C(k + t)
			+
			\epsilon k t.
		\end{equation}		
		By dividing both sides of (\ref{eq_pf_2_0}) by $t$ and taking limit $t \to \infty$, we conclude that 
		\begin{equation}\label{eq_pf_2_1}
			d_p \big( \mathcal{F}^{\mathcal{T}_1}_{k}, \mathcal{F}^{\mathcal{T}_2}_{k} \big)
			\leq
			k \cdot d_p \big(\{ h^{\mathcal{T}_1}_t \}, \{ h^{\mathcal{T}_2}_t \} \big)
			+
			C
			+
			\epsilon k.
		\end{equation}	
		By dividing both sides of (\ref{eq_pf_2_1}) by $k$ and taking limit $k \to \infty$, we conclude that 
		\begin{equation}\label{eq_pf_2_2}
			\limsup_{k \to \infty} \frac{d_p \big( \mathcal{F}^{\mathcal{T}_1}_{k}, \mathcal{F}^{\mathcal{T}_2}_{k} \big)}{k}
			\leq
			d_p \big(\{ h^{\mathcal{T}_1}_t \}, \{ h^{\mathcal{T}_2}_t \} \big)
			+
			\epsilon.
		\end{equation}	
		Since $\epsilon > 0$ was chosen arbitrarily, we obtain the lower bound of Theorem \ref{thm_dist_na}.
	\end{proof}
	\begin{sloppypar}
	\begin{rem}\label{rem_chen_mclean_alt_pf}
		From our proof of Theorem \ref{thm_dist_na}, we see that the limit of $d_p ( \mathcal{F}^{\mathcal{T}_1}_{k}, \mathcal{F}^{\mathcal{T}_2}_{k} ) / k$, $p \in [1, +\infty]$, exists as $k \to \infty$.
		For $p \in [1, +\infty[$, a different proof of this fact was given by Chen-Maclean \cite[Theorem 4.3]{ChenMaclean}, cf. also Boucksom-Jonsson \cite[Theorem 3.3 and \S 3.4]{BouckJohn21}.
	\end{rem}
	\end{sloppypar}

	\subsection{Geodesic rays of Hermitian norms as the ray of $L^2$-norms}\label{sect_quot}
	The main goal of this section is to compare on a section ring geodesic rays of Hermitian norms associated with an ample test configuration with $L^2$-norms, i.e. to establish Theorem \ref{thm_2_step}.
	\par 
	\begin{sloppypar}
	Before all, let us introduce some notations from linear algebra.
	Recall first that a norm (Archimedean or non-Archimedean) $N_V = \| \cdot \|_V$ on a finitely dimensional vector space $V$ naturally induces the norm $\| \cdot \|_Q := [N_V]$ on any quotient $Q$, $\pi : V \to Q$ of $V$ as follows
	\begin{equation}\label{eq_defn_quot_norm}
		\| f \|_Q
		:=
		\inf \Big \{
		 \| g \|_V
		 :
		 \quad
		 g \in V, 
		 \pi(g) = f
		\Big\},
		\qquad f \in Q.
	\end{equation}
	Clearly, if $N_V$ is Hermitian, the quotient $N_Q$ is Hermitian as well.
	\end{sloppypar}
	\par 
	Let $V$ (resp. $W$) be a finitely dimensional vector space with a Hermitian norm $H_V$ (resp. $H_W$).
	We denote by ${\rm{Sym}}^l H_V$, $l \in \nat$, (resp. $H_V \otimes H_W$) the Hermitian norm on ${\rm{Sym}}^l V$ (resp. $V \otimes W$) associated with the scalar product induced by $H_V$ (resp. $H_V$ and $H_W$).
	\par 
	Now, for a polarized projective manifold $(X, L)$ and any $l, k \in \nat^*$, we define the multiplication
	\begin{equation}\label{eq_mult_map}
		{\rm{Mult}}_{l, k}^{{\rm{Sym}}} : {\rm{Sym}}^l H^0(X, L^k)
		\to
		H^0(X, L^{kl}),
	\end{equation}	
	as $f_1 \otimes \cdots \otimes f_l \mapsto f_1 \cdots f_l$. 
	Similarly, we define 
	\begin{equation}\label{eq_mult_map2}
		{\rm{Mult}}_{l, k} : H^0(X, L^l) \otimes H^0(X, L^k)
		\to
		H^0(X, L^{l + k}).
	\end{equation}	
	The core of the proof of Theorem \ref{thm_2_step}, from which we conserve the notations, lies in the following several results.
	\begin{thm}\label{thm_2_step_1}
		There is $k_0 \in \nat$, such that for any $k \geq k_0$, there are $C > 0$, $l_0 \in \nat$, such that for any $l \geq l_0$, $t \in [0, +\infty[$,  under the map (\ref{eq_mult_map}), the following inequality takes place
		\begin{equation}
			\exp(C(l + t)) \cdot
			{\rm{Hilb}}_{kl}(h^{\mathcal{T}}_t, \omega)
			\geq
			\big[ {\rm{Sym}}^l H^{\mathcal{T}}_{t, k} \big].
		\end{equation}
	\end{thm}
	
	\begin{thm}\label{thm_2_step_121}
		There are $C > 0$, $k_0 \in \nat$, such that for any $l, k \geq k_0$, $t \in [0, +\infty[$,  under the map (\ref{eq_mult_map2}), the following inequality takes place
		\begin{equation}
			\exp(C t + C) \cdot
			{\rm{Hilb}}_{l + k}(h^{\mathcal{T}}_t, \omega)
			\geq
			\big[ {\rm{Hilb}}_l(h^{\mathcal{T}}_t, \omega) \otimes {\rm{Hilb}}_k(h^{\mathcal{T}}_t, \omega) \big].
		\end{equation}
	\end{thm}
	\begin{rem}
		Theorems \ref{thm_2_step_1} and \ref{thm_2_step_121} are uniform weak versions of \cite[Theorems 1.1, 4.18]{FinSecRing}.
	\end{rem}
	The proofs of Theorems \ref{thm_2_step_1}, \ref{thm_2_step_121}, which will be presented in Section \ref{sect_ohs}, rely on Ohsawa-Takegoshi extension theorem.
	\begin{thm}\label{thm_2_step_2}
		For any $\epsilon > 0$, there is $k_0 \in \nat$ such that for any $l, k \geq k_0$, $t \in [0, +\infty[$,  under the map (\ref{eq_mult_map}), the following inequality takes place
		\begin{equation}
			\big[ {\rm{Sym}}^l H^{\mathcal{T}}_{t, k} \big]
			\geq
			\exp(- \epsilon kl) \cdot
			H^{\mathcal{T}}_{t, kl}.
		\end{equation}
		Similarly, under the map (\ref{eq_mult_map2}), the following inequality takes place
		\begin{equation}\label{eq_2_step_212}
			\big[ H^{\mathcal{T}}_{t, l} \otimes H^{\mathcal{T}}_{t, k} \big]
			\geq
			\exp(- \epsilon (l + k)) \cdot
			H^{\mathcal{T}}_{t, l + k}.
		\end{equation}
	\end{thm}
	\par 
	Before describing the proof of Theorem \ref{thm_2_step_2}, let us explain how along with Theorems \ref{thm_2_step_1}, \ref{thm_2_step_121}, they entail Theorem \ref{thm_2_step}.
	For this, we need to have a better understanding of the uniform quantization properties of the geodesic ray of Hermitian norms on the section ring, and the following result of Phong-Sturm will be of paramount significance for this.
	\par 
	Let us fix an arbitrary ample test configuration $\mathcal{T} = (\pi: \mathcal{X} \to \comp, \mathcal{L})$ and an arbitrary $\comp^*$-equivariant resolution of singularities $\mathcal{T}' = (\pi': \mathcal{X}' \to \comp, \mathcal{L}')$ of $\mathcal{T}$.
	We fix an arbitrary smooth metric $h^{\mathcal{L}'}$ on $\mathcal{L}'$, and denote by $h^{\mathcal{T} {\rm{sm}}}_t$, $t \in [0, +\infty[$, the ray of smooth positive metrics on $L$ associated with $h^{\mathcal{L}'}$ in the same way as the geodesic ray $h^{\mathcal{T}}_t$, $t \in [0, +\infty[$, of metrics on $L$ was associated with a solution of the Monge-Ampère equation (\ref{eq_ma_geod_dir}).
	\begin{thm}\label{thm_ph_st_regul}
		There are $C > 0$, $k_0 \in \nat$, such that for any $t \in [0, +\infty[$, $k \geq k_0$, we have
		\begin{equation}\label{eq_ph_st_regul011}
			\exp(-C) \cdot h^{\mathcal{T} {\rm{sm}}}_t  \leq FS(H^{\mathcal{T}}_{t, k})^{\frac{1}{k}} \leq \exp(C)  \cdot h^{\mathcal{T} {\rm{sm}}}_t.
		\end{equation}
		In particular, by Theorem \ref{thm_phong_sturm_limit}, we have
		\begin{equation}
			\exp(-C) \cdot h^{\mathcal{T} {\rm{sm}}}_t \leq h^{\mathcal{T}}_t \leq \exp(C)  \cdot h^{\mathcal{T} {\rm{sm}}}_t.
		\end{equation}
	\end{thm}
	\begin{rem}\label{rem_ph_st_regul}
		From the second part of Theorem \ref{thm_ph_st_regul}, we see that two ample test configurations give rise to the same geodesic ray of metrics if and only if they are equivalent.
	\end{rem}
	\begin{proof}
		The only new statement is the validity of the upper bound of (\ref{eq_ph_st_regul011}).
		The rest was established by Phong-Sturm in the proof of \cite[Lemma 4]{PhongSturmDirMA}, including the validity of the upper bound of (\ref{eq_ph_st_regul011}) for $k \in \nat$ , divisible by $k_0 \in \nat$, where $k_0$ is any sufficiently big natural number.
		\par 
		Let us now fix $k_0, k_1 \in \nat$ sufficiently big and relatively prime.
		For a given $k \in \nat$, $k \geq 2 k_0 k_1$, we decompose $k = k_0 r + k_1 s$, $r, s \in \nat$.
		An easy calculation, cf. Lemma \cite[Lemma 4.11]{FinSecRing}, shows that $FS(\big[ H^{\mathcal{T}}_{t, k_0 r} \otimes H^{\mathcal{T}}_{t, k_1 s} \big]) = FS(H^{\mathcal{T}}_{t, k_0 r}) \cdot FS(H^{\mathcal{T}}_{t, k_1 s})$.
		By (\ref{eq_2_step_212}) and the validity of the upper bound (\ref{eq_ph_st_regul011}) for $k := k_0 r$, $k := k_1 s$, we deduce its validity for all sufficiently big $k$.
	\end{proof}
	We are now finally ready to prove the main result of this section.
	\begin{proof}[Proof of Theorem \ref{thm_2_step}]
		Directly from the lower bound of (\ref{eq_ph_st_regul011}), by Lemma \ref{lem_fs_inf_d}, there are $C > 0$, $k_0 \in \nat$, such that for any $k \geq k_0$, $t \in [0, +\infty[$, we have
		\begin{equation}\label{eq_low_bnd_inf1}
			H^{\mathcal{T}}_{t, k}
			\geq
			{\rm{Ban}}^{\infty}_k \big( \exp(- C) \cdot h^{\mathcal{T}}_t \big).
		\end{equation}
		Since the $L^{\infty}$-norm dominates the $L^2$-norm, we deduce that
		\begin{equation}\label{eq_low_bnd_inf12}
			H^{\mathcal{T}}_{t, k}
			\geq
			\exp(- C k)
			\cdot
			{\rm{Hilb}}_k(h^{\mathcal{T}}_t, \omega).
		\end{equation}
		On another hand, directly from Theorem \ref{thm_2_step_1} and the first part of Theorem \ref{thm_2_step_2}, for any $k_0, k_1 \in \nat$ sufficiently big, there is $C > 0$, such that for any $t \in [0, +\infty[$, $k$ divisible by $k_0$ or $k_1$, we have
		\begin{equation}\label{eq_low_bnd_inf2}
			{\rm{Hilb}}_k(h^{\mathcal{T}}_t, \omega)
			\geq
			\exp(- C (k + t))
			\cdot
			H^{\mathcal{T}}_{t, k}.
		\end{equation}
		If we now fix $k_0, k_1 \in \nat^*$, which are relatively prime and big enough, and apply Theorem \ref{thm_2_step_121}, the second part of Theorem \ref{thm_2_step_2} and (\ref{eq_low_bnd_inf2}) for $k = k_0 r$ and $k = k_1 s$, where $r, s \in \nat$ are big enough, we obtain that (\ref{eq_low_bnd_inf2}) holds for all $k$ sufficiently large, since any such number can be written as $k_0 r + k_1 s$, $r, s \in \nat$.
		Then a combination of (\ref{eq_low_bnd_inf12}) and (\ref{eq_low_bnd_inf2}) gives us a proof of the first part of Theorem \ref{thm_2_step}.
		Also, from (\ref{eq_low_bnd_inf1}) and (\ref{eq_low_bnd_inf2}), there are $C > 0$, $k_0 \in \nat$, such that for any $t \in [0, +\infty[$, $k \geq k_0$, we have
		\begin{equation}
			\exp(Ck) \cdot {\rm{Hilb}}_k(h^{\mathcal{T}}_t, \omega)
			\geq
			{\rm{Ban}}^{\infty}_k(h^{\mathcal{T}}_t).
		\end{equation}
		The proof of the second part of Theorem \ref{thm_2_step} follows this and the first part.
	\end{proof}
	\par 
	We will now prove Theorem \ref{thm_2_step_2}.
	This result in essence says that the construction of geodesic rays of Hermitian norms emanating from $L^2$-norms respects the multiplicative structure of the section ring.
	The proof of Theorem \ref{thm_2_step_2} decomposes into three statements.
	The first statement from \cite{FinSecRing} shows that the construction of $L^2$-norms respects the multiplicative structure of the section ring.
	The second statement shows that geodesic rays of Hermitian norms on finitely dimensional vector spaces behave reasonably under taking quotients.
	The third statement shows that formation of geodesic rays on finitely dimensional vector spaces is compatible with tensor products.
	We begin by recalling the first result.
	\begin{lem}[{\cite[Theorem 1.5]{FinSecRing}}]\label{thm_sec_ring}
		For any continuous psh metric $h^L$ on $L$ and any $\epsilon > 0$, there is $k_0 \in \nat$, such that for any $l, k \geq k_0$,  under the map (\ref{eq_mult_map}), the following inequality holds
		\begin{equation}\label{eq_sec_ring}
			\big[  {\rm{Sym}}^l {\rm{Hilb}}_k(h^L) \big]
			\geq
			\exp(- \epsilon k l)
			\cdot
			{\rm{Hilb}}_{kl}(h^L).
		\end{equation}
		Similarly, for any $\epsilon > 0$, there is $k_0 \in \nat$, such that for any $l, k \geq k_0$,  under the map (\ref{eq_mult_map2}), the following inequality holds
		\begin{equation}\label{eq_sec_ring2}
			\big[  {\rm{Hilb}}_l(h^L) \otimes {\rm{Hilb}}_k(h^L) \big]
			\geq
			\exp(- \epsilon (l + k))
			\cdot
			{\rm{Hilb}}_{l + k}(h^L).
		\end{equation}
	\end{lem}
	The first linear algebra ingredient in the proof of Theorem \ref{thm_2_step_2} goes as follows.
	\par 
	\begin{prop}[Non-Archimedean interpolation theorem of Stein-Weiss]\label{prop_interpol}
		Let $H_0$ (resp. $H_1$) be a fixed Hermitian norm on $V$ (resp. $Q$) and $\mathcal{F}$ (resp. $\mathcal{G}$) is a filtration on $V$ (resp. $Q$).
		We assume that $[H_0] \geq H_1$ and $[\chi_{\mathcal{F}}] \geq \chi_{\mathcal{G}}$.
		Then the geodesic ray $H_t^{\mathcal{F}}$, $t \in [0, +\infty[$, of Hermitian norms on $V$ associated with $\mathcal{F}$ and emanating from $H_0$ compares to the geodesic ray $H_t^{\mathcal{G}}$ of Hermitian norms on $Q$ associated with $\mathcal{G}$ and emanating from $H_1$ as follows 
		\begin{equation}\label{eq_interpol}
			[H_t^{\mathcal{F}}]
			\geq
			H_t^{\mathcal{G}}.
		\end{equation}
	\end{prop}
	\begin{proof}
		Remark first that the conditions $[H_0] \geq H_1$ and $[\chi_{\mathcal{F}}] \geq \chi_{\mathcal{G}}$ are equivalent to the fact that the quotient map $\pi : V \to Q$ is 1-Lipschitz, where $V$ is endowed with the norm $H_0$ (resp. $\chi_{\mathcal{F}}$) and $Q$ is endowed with the norm $H_1$ (resp. $\chi_{\mathcal{G}}$).
		In this perspective Proposition \ref{prop_interpol} is a non-Archimedean version of the interpolation theorem of Stein-Weiss, cf. \cite[Theorem 5.4.1]{InterpSp}, saying, as we recall below, that a similar statement holds for geodesics between two fixed Hermitian norms.
		\par 
		The proof of Proposition \ref{prop_interpol} proceeds in two steps.
		Let us denote by $N_t^{\mathcal{F}}$ (resp. $N_t^{\mathcal{G}}$) the ray of norms on $V$ (resp. $Q$) associated with $\mathcal{F}$ (resp. $\mathcal{G}$) and emanating from $H_0$ (resp. $H_1$) as in (\ref{eq_ray_norm_defn0}).
		Directly from (\ref{eq_ray_norm_defn0}) and our assumptions on the relation between $H_0$ and $H_1$, $\mathcal{F}$ and $\mathcal{G}$, the following inequality is satisfied
		$
			[N_t^{\mathcal{F}}]
			\geq
			N_t^{\mathcal{G}}
		$.
		From this and Lemma \ref{lem_two_norms_comp0}, we conclude 
		\begin{equation}\label{eq_interpol0}
			\dim V^2 \cdot
			[H_t^{\mathcal{F}}]
			\geq
			H_t^{\mathcal{G}}.
		\end{equation}
		We will now show that this estimate can be bootstrapped to (\ref{eq_interpol}).
		\par 
		In fact, recall that in \cite[Lemma 4.21]{FinSecRing} we proved, by essentially reformulating interpolation theorem of Stein-Weiss, that for any two Hermitian norms $H^V_0, H^V_1$ on $V$ and any two Hermitian norms $H^Q_0, H^Q_1$ on $Q$, verifying $[H^V_0] \geq H^Q_0$ and $[H^V_1] \geq H^Q_1$, for the geodesic rays of norms $H^V_t$ (resp. $H^Q_t$) between $H^V_0$ and $H^V_1$ (resp. $H^Q_0$ and $H^Q_1$), we have $[H^V_t] \geq H^Q_t$, for any $t \in [0, 1]$.
		Now, for fixed $h > 0$, by (\ref{eq_interpol0}), we can apply this result for $H^V_0 := H_0$, $H^V_1 := H_h^{\mathcal{F}}$ and $H^Q_0 := H_1$, $H^Q_1 := \frac{1}{\dim V^2} H_h^{\mathcal{G}}$.
		For $t \in [0, h]$, it gives us
		$
			(\dim V^2)^{\frac{t}{h}} \cdot
			[H_t^{\mathcal{F}}]
			\geq
			H_t^{\mathcal{G}}
		$.
		As $h$ can be chosen as large as we wish, we deduce (\ref{eq_interpol}).
	\end{proof}
	\par 
	Now, to state the last linear-algebraic ingredient, let us fix a finitely dimensional vector space $V$, endowed with a Hermitian norm $H_V$ and a filtration $\mathcal{F}$.
	We denote by ${\rm{Sym}}^k \mathcal{F}$ the filtration on ${\rm{Sym}}^k V$ induced from the filtration $\mathcal{F}$ on $V$, and by $H_t^{{\rm{Sym}}^l \mathcal{F}}$ the geodesic ray of Hermitian norms on ${\rm{Sym}}^l V$ emanating from ${\rm{Sym}}^l H_V$ associated with the filtration ${\rm{Sym}}^l \mathcal{F}$.
	Similarly, we fix another finitely dimensional vector space $W$, endowed with a Hermitian norm $H_W$ and a filtration $\mathcal{G}$.
	Recall that the filtration  $\mathcal{F} \otimes \mathcal{G}$ on $V \otimes W$ is defined so that in terms of the associated non-Archimedean norms, defined as in (\ref{eq_filtr_norm}), we have
	\begin{equation}
		\chi_{\mathcal{F} \otimes \mathcal{G}}(h) = \min \max_{i = 1, \cdots, N} \chi_{\mathcal{F}}(f_i) \cdot \chi_{\mathcal{F}}(g_i),
	\end{equation}
	where the minimum is taken over all possible decompositions $h = \sum_{i = 1}^{N} f_i \otimes g_i$, $N \in \nat$.
	We denote by $H_t^{\mathcal{F} \otimes \mathcal{G}}$ the geodesic ray of Hermitian norms on $V \otimes W$ emanating from $H_V \otimes H_W$ and associated with $\mathcal{F} \otimes \mathcal{G}$.
	\begin{lem}\label{lem_p_expl_calc1}
		The construction of geodesic rays of norms is compatible with the symmetrization and tensor products.
		In other words, for any $l \in \nat^*$ and $t \in [0, +\infty[$, we have
		\begin{equation}\label{eq_p_expl_calc1}
			H_t^{{\rm{Sym}}^l \mathcal{F}}
			=
			{\rm{Sym}}^l H_t^{\mathcal{F}}, 
			\qquad 
			H_t^{\mathcal{F} \otimes \mathcal{G}}
			=
			H_t^{\mathcal{F}}
			\otimes
			H_t^{\mathcal{G}}.
		\end{equation}
	\end{lem}
	\begin{proof}
		The proofs of the both statements are identical, so we only concentrate on the part concerning the symmetric powers.
		We denote $n = \dim V$ and let $e_1, \ldots, e_n$ be an orthonormal basis of $(V, H_V)$, adapted to the filtration $\mathcal{F}$ as in (\ref{eq_bas_st}).
		Then we see that for multiindices $\alpha \in \nat^n$, $|\alpha| = l$, $\alpha = (\alpha_1, \ldots, \alpha_n)$, the basis $\sqrt{\frac{k!}{\alpha!}} e^{\alpha} := \sqrt{\frac{l!}{\alpha_1! \cdots \alpha_n!}} e_1^{\alpha_1} \cdots e_n^{\alpha_n}$ is an adapted basis for the filtration ${\rm{Sym}}^l \mathcal{F}$ on the Hermitian vector space $({\rm{Sym}}^l V, {\rm{Sym}}^l H_V)$.
		From this, we deduce (\ref{eq_p_expl_calc1}).
	\end{proof}
	
	\begin{proof}[Proof of Theorem \ref{thm_2_step_2}]
		The proofs of the both statements are identical, so we only concentrate on the part concerning the symmetric powers.
		By Lemma \ref{lem_p_expl_calc1}, we see that ${\rm{Sym}}^l H^{\mathcal{T}}_{t, k}$, $t \in [0, +\infty[$, can be interpreted as the geodesic ray of Hermitian norms, emanating from  ${\rm{Sym}}^l H^{\mathcal{T}}_{0, k} = {\rm{Sym}}^l {\rm{Hilb}}_k(h^L_0)$, associated with the filtration ${\rm{Sym}}^l \mathcal{F}^{\mathcal{T}}_k$.
		By Lemma \ref{thm_sec_ring} and submultiplicativity of $\mathcal{F}$, we see that all the assumptions of Proposition \ref{prop_interpol} are satisfied for $H_t^{\mathcal{F}} := {\rm{Sym}}^l H^{\mathcal{T}}_{t, k}$ and $H_t^{\mathcal{G}} := \exp(- \epsilon k l) H^{\mathcal{T}}_{t, kl}$. 
		Proposition \ref{prop_interpol} then in our context gives us exactly Theorem \ref{thm_2_step_2}.
	\end{proof}

	\subsection{Ohsawa-Takegoshi extension theorem and quotients of geodesic rays}\label{sect_ohs}
	The main goal of this section is to prove Theorems \ref{thm_2_step_1},  \ref{thm_2_step_121}. 
	The proofs are based on a version of Ohsawa-Takegoshi extension theorem with a uniform constant, which we recall below.
	\par 
	We fix a compact complex manifold $X$ of dimension $n$ with an ample line bundle $L$ over it, endowed with a smooth positive metric $h^L_0$.
	Let $L_0, L_1$ be two line bundles on $X$, endowed with smooth semipositive metrics $h^{L_0}_0, h^{L_1}_0$, such that $(L_0 \otimes L_1, h^{L_0}_0 \otimes h^{L_1}_0)$ is a positive line bundle.
	Let $Y$ be a closed submanifold of $X$ of dimension $m$.
	Let $\omega$ be a fixed Kähler form on $X$.
	 \begin{thm}\label{thm_ot_asymp}
	 	There are $c, C > 0$, $k_0 \in \nat$, such that for any $k \geq k_0$, any psh metric $h^L$ and any section $f \in H^0(Y, L|_Y^k)$, there is a holomorphic extension $\tilde{f} \in H^0(X, L^k)$ of $f$, such that the following $L^2$-bound is satisfied
	 	\begin{equation}\label{eq_ot_asymp}
	 		\int_X | \tilde{f}(x) |_{h^L} \cdot \omega^n(x) 
	 		\leq
	 		C
	 		\cdot
	 		\exp \big( c d_{+ \infty}(h^L, h^L_0) \big)
	 		\cdot
	 		\int_Y | f(y) |_{h^L} \cdot \omega^m(y). 
	 	\end{equation}
	 	Similarly, there are $c, C > 0$, $k_0 \in \nat$, such that for any $k, l \geq k_0$, any psh metrics $h^{L_0}$, $h^{L_1}$ on $L_0$, $L_1$, and any section $f \in H^0(Y, L_0|_Y^k \otimes L_1|_Y^l)$, there is a holomorphic extension $\tilde{f} \in H^0(X, L_0^k \otimes L_1^l)$ of $f$, such that the following $L^2$-bound is satisfied
	 	\begin{multline}\label{eq_ot_asymp2}
	 		\int_X | \tilde{f}(x) |_{h^{L_0}, h^{L_1}} \cdot \omega^n(x) 
	 		\leq
	 		C
	 		\cdot
	 		\exp \Big( c  \big( d_{+ \infty}(h^{L_0}, h^{L_0}_0) + d_{+ \infty}(h^{L_1}, h^{L_1}_0) \big) \Big)
	 		\cdot
	 		\\
	 		\cdot
	 		\int_Y | f(y) |_{h^{L_0}, h^{L_1}} \cdot \omega^m(y),
	 	\end{multline}
	 	where $| \cdot |_{h^{L_0}, h^{L_1}}$ is the pointwise norm induced by $h^{L_0}$ and $h^{L_1}$.
	 \end{thm}
	 \begin{proof}
	 	See the proof of \cite[Theorem 2.5]{FinSecRing}, which is a rather direct adaptation of more general and refined results of Demailly \cite[Theorem 2.8]{DemExtRed} and Ohsawa \cite{OhsawaV}.
	 \end{proof}
	 \par 
	We will now reformulate Theorem \ref{thm_ot_asymp} in a form, which is better suited for our needs.
	 Consider the restriction operator
	 \begin{equation}\label{eq_res_map}
	 	{\rm{Res}}_k : H^0(X, L^k) \to H^0(Y, L|_Y^k), \quad {\rm{Res}}_{k, l} : H^0(X, L_0^k \otimes L_1^l) \to H^0(Y, L_0|_Y^k \otimes L_1|_Y^l).
	 \end{equation}
	 In the language of quotient norms from (\ref{eq_defn_quot_norm}), considered with respect to the maps (\ref{eq_res_map}), Theorem \ref{thm_ot_asymp} can be restated as follows: the maps (\ref{eq_res_map}) are surjective, and the following bound holds
	 \begin{equation}\label{eq_reform_thm_ot_asymp}
	 \begin{aligned}
	 	&
 		[{\rm{Hilb}}^X_k(h^L, \omega)] 
 		\leq 
 		C
 		\cdot
 		\exp(c d_{+ \infty}(h^L, h^L_0))
 		\cdot 
 		{\rm{Hilb}}^Y_k(h^L, \omega|_Y),
 		\\
 		&
 		[{\rm{Hilb}}^X_{k, l}(h^{L_0}, h^{L_1}, \omega)] 
 		\leq 
 		C
 		\cdot
 		\exp \big( c  ( d_{+ \infty}(h^{L_0}, h^{L_0}_0) + d_{+ \infty}(h^{L_1}, h^{L_1}_0) ) \big)
 		\cdot
 		\\
 		&
 		\qquad \qquad \qquad \qquad \qquad \qquad \qquad \qquad \qquad \qquad \qquad \qquad \qquad 
 		\cdot 
 		{\rm{Hilb}}^Y_{k, l}(h^{L_0}, h^{L_1}, \omega|_Y),
	 \end{aligned}
	 \end{equation}
	where ${\rm{Hilb}}^X_k(h^L, \omega)$ and ${\rm{Hilb}}^Y_k(h^L, \omega|_Y)$ (resp. ${\rm{Hilb}}^X_{k, l}(h^{L_0}, h^{L_1}, \omega)$ and ${\rm{Hilb}}^Y_{k, l}(h^{L_0}, h^{L_1}, \omega|_Y)$) stand for the $L^2$-norms on $H^0(X, L^k)$ and $H^0(Y, L|_Y^k)$ (resp. $H^0(X, L_0^k \otimes L_1^l)$ and $H^0(Y, L_0|_Y^k \otimes L_1|_Y^l)$) induced by $h^L$ (resp. $h^{L_0}$, $h^{L_1}$) and $\omega$.
	\par 
	To apply Theorem \ref{thm_ot_asymp} in the proof of Theorem \ref{thm_2_step_1}, we need to interpret the symmetric tensor product norm in terms of the $L^2$-norm.
	The following well-known result gives us exactly that.
	\begin{lem}\label{lem_p_expl_calc}
	 	For any Hermitian norm $H_V$ on a finitely dimensional complex vector space $V$, and any $k \in \nat^*$, we have 
		\begin{equation}\label{eq_p_expl_calc2}
			{\rm{Sym}}^k H_V
			=
			{\rm{Hilb}}^{\mathbb{P}(V^*)}_k \big( FS^{\mathbb{P}}(H_V) \big)
			\cdot
			\sqrt{\frac{(k + \dim V - 1)!}{k!}},
		\end{equation}
		where we implicitly used the the canonical isomorphism $H^0(\mathbb{P}(V^*), \mathscr{O}(k)) \simeq {\rm{Sym}}^k V$, and by $FS^{\mathbb{P}}(H_V)$ we mean the Fubini-Study metric on $\mathscr{O}(1)$ induced by $H_V$.
	\end{lem}
	\begin{proof}
		Let $e_1, \ldots, e_n$ be an orthonormal basis of $(V, H_V)$.
		For any $k \in \nat$ and a multiindex $\alpha \in \nat^n$, $|\alpha| = k$, we have
		$
			\big\| e^{\alpha} \big\|^2_{{\rm{Sym}}^k H_V}
			=
			\frac{\alpha!}{k!}
		$.
		It is, hence, sufficient to establish that 
		\begin{equation}\label{eq_verif_l2}
			\big\| e^{\alpha} \big\|^2_{{\rm{Hilb}}^{\mathbb{P}(V^*)}_k(FS^{\mathbb{P}}(H_V))}
			=
			\frac{\alpha!}{(k + n - 1)!}.
		\end{equation}
		By pulling back the integral from the definition of the $L^2$-norm on $\mathbb{P}(V^*)$ to the unit sphere $S^{2 n - 1}(V^*) \subset V^*$ under the natural projection map $p : S^{2 n - 1}(V^*) \to \mathbb{P}(V^*)$, cf. \cite[p.6]{DivKobLub}, the verification of (\ref{eq_verif_l2}) boils down to the verification of the identity
		\begin{equation}
			\int_{S^{2 n - 1}} |z|^{2\alpha} d \sigma(z) 
			=
			\frac{\alpha! (n - 1)!}{(k + n - 1)!},
		\end{equation}
		where $d \sigma$ is the standard volume form on the standard unit sphere $S^{2 n - 1} \subset \comp^n$, normalized so that the total volume equals one. The last calculation is standard, cf. \cite[Lemma 2.2]{DivKobLub}.
	\end{proof}
	
	Now, we fix a finitely dimensional vector space $V$ and endow it with a Hermitian norm $H_V$.
	We fix a filtration $\mathcal{F}$ on $V$, and denote by $H_t^{\mathcal{F}}$, $t \in [0, +\infty[$, the ray of Hermitian norms emanating from $H_V$ as in (\ref{eq_bas_st}).
	We denote by $FS^{\mathbb{P}}(H_t^{\mathcal{F}})$, $t \in [0, +\infty[$, the ray of Fubini-Study metrics on the hyperplane line bundle $\mathscr{O}(1)$ of $\mathbb{P}(V^*)$ constructed as in (\ref{sec_fs_bdem}).
	For any $k \in \nat^*$, a filtration $\mathcal{F}$ on $V$ induces a filtration ${\rm{Sym}}^k \mathcal{F}$ on ${\rm{Sym}}^k V$.
	We denote by $H_t^{{\rm{Sym}}^k \mathcal{F}}$, $t \in [0, +\infty[$, the ray of Hermitian norms on ${\rm{Sym}}^k V$ emanating from ${\rm{Sym}}^k H_V$ as in (\ref{eq_bas_st}).
	By Lemmas \ref{lem_p_expl_calc1} and \ref{lem_p_expl_calc}, we deduce that for any $k \in \nat^*$, $t \in [0, +\infty[$, the following identity holds
		\begin{equation}\label{eq_gt_sym_l2}
			H_t^{{\rm{Sym}}^k \mathcal{F}}
			=
			{\rm{Hilb}}^{\mathbb{P}(V^*)}_k \big( FS^{\mathbb{P}}(H_t^{\mathcal{F}}) \big)
			\cdot
			\sqrt{\frac{(k + \dim V)!}{k!}}.
		\end{equation}	
	\par 
	We are now finally ready to prove the main results of this section.
	\begin{proof}[Proof of Theorem \ref{thm_2_step_1}]
		From Theorem \ref{thm_ph_st_regul}, (\ref{eq_gt_sym_l2}) and the fact that the function $\frac{(l + n)!}{l!}$ grows polynomially in $l \in \nat^*$ for fixed $n \in \nat^*$, we see that it is enough to prove that there is $k_0 \in \nat$, such that for any $k \geq k_0$, there are $C > 0$, $l_0 \in \nat$, such that for any $l \geq l_0$, $t \in [0, +\infty[$, under the map (\ref{eq_mult_map}), we have
		\begin{equation}\label{eq_fin_11_00}
			\exp(C(l + t)) \cdot
			{\rm{Hilb}}^{X}_{kl}(FS(H_{t, k}^{\mathcal{F}})^{\frac{1}{k}}, \omega)
			\geq
			\big[ {\rm{Hilb}}^{\mathbb{P}(H^0(X, L^k)^*)}_l(FS^{\mathbb{P}}(H_{t, k}^{\mathcal{F}})) \big].
		\end{equation}
		\par 
		Let us now denote by $\omega_{\mathbb{P}}$ the Fubini-Study Kähler form on $\mathbb{P}(H^0(X, L^k)^*)$ associated with the Hermitian norm $H_{0, k}^{\mathcal{F}} = {\rm{Hilb}}_k(h^L_0)$ on $H^0(X, L^k)$.
		Remark that under the identification $H^0(\mathbb{P}(H^0(X, L^k)^*), \mathscr{O}(l)) = {\rm{Sym}}^l H^0(X, L^k)$, the multiplication map (\ref{eq_mult_map}) corresponds to the restriction map 
		${\rm{Res}}_{l} : 
			H^0(\mathbb{P}(H^0(X, L^k)^*), \mathscr{O}(l))
			\to
			H^0(X, L^{kl})$,
		associated with the Kodaira embedding (\ref{eq_kod}), cf. \cite[(4.62)]{FinSecRing}.
		In other words, the following diagram is commutative 
		\begin{equation}\label{eq_kod_map_comm_d}
		\begin{tikzcd}
			{\rm{Sym}}^l(H^0(X, L^k)) \arrow[swap, rd, "{\rm{Mult}}_{l, k}"] \arrow[r, equal] & H^0(\mathbb{P}(H^0(X, L^k)^*), \mathscr{O}(l)) \arrow[d, "\res_l"]  \\
			&  H^0(X, L^{kl}).
    	\end{tikzcd}
		\end{equation}
		From this observation, Theorem \ref{thm_ot_asymp} in its form (\ref{eq_reform_thm_ot_asymp}), (\ref{eq_ray_norms_decart}) and the very definition of $FS(H_{t, k}^{\mathcal{F}})$ as the pull-back of $FS^{\mathbb{P}}(H_{t, k}^{\mathcal{F}})$ through the Kodaira map, see (\ref{eq_fs_defn}), we conclude that for any $k \in \nat^*$, there are $C > 0$, $l_0 \in \nat$, such that for any $l \geq l_0$, $t \in [0, +\infty[$, we have
		\begin{equation}\label{eq_fin_11_0}
			\exp(C(l + t)) \cdot
			{\rm{Hilb}}^{X}_{kl}(FS(H_{t, k}^{\mathcal{F}})^{\frac{1}{k}}, {\rm{Kod}}_k^* \omega_{\mathbb{P}})
			\geq
			\big[ {\rm{Hilb}}^{\mathbb{P}(H^0(X, L^k)^*)}_l(FS^{\mathbb{P}}(H_{t, k}^{\mathcal{F}}), \omega_{\mathbb{P}}) \big].
		\end{equation}
		\par  
		Since $c_1(\mathscr{O}(1), FS^{\mathbb{P}}(H_{t, k}^{\mathcal{F}}))$ is the Fubini-Study form associated with $H_{t, k}^{\mathcal{F}}$, by (\ref{eq_ray_norms_decart}), there is $C > 0$, such that for any $k \in \nat^*$, $t \in [0, +\infty[$, we have
		\begin{equation}\label{eq_fin_11_2}
			\omega_{\mathbb{P}} 
			\geq
			\exp(- C t k) \cdot c_1(\mathscr{O}(1), FS^{\mathbb{P}}(H_{t, k}^{\mathcal{F}})).
		\end{equation}
		Also, for any fixed Kähler form $\omega$ on $X$ and any $k \in \nat^*$, there is $C > 0$, such that ${\rm{Kod}}_k^* \omega_{\mathbb{P}} \leq C \omega$.
		From this, (\ref{eq_fin_11_0}) and (\ref{eq_fin_11_2}), we deduce (\ref{eq_fin_11_00}).
	\end{proof}
	
	\begin{proof}[Proof of Theorem \ref{thm_2_step_121}]
		Let us consider the product manifold $X \times  X$ and the diagonal submanifold in it, given by $\{(x, x) : x \in X \} =: \Delta \hookrightarrow X \times X$.
		We denote by $L^k \boxtimes L^l$ the line bundle over $X \times  X$, given by $\pi_0^* L^k \otimes \pi_1^* L^l$, where $\pi_0, \pi_1 : X \times X \to X$ are the projections onto the first and second factors respectively.
		The natural identification of $\Delta$ with $X$, Künneth isomorphism and multiplication map (\ref{eq_mult_map2}) can be put into the following commutative diagram 
		\begin{equation}\label{eq_comm_diag}
			\begin{CD}
				H^0(X \times  X, L^k \boxtimes L^l)  @> {\rm{Res}}_{\Delta} >> H^0(\Delta, L^k \boxtimes L^l|_{\Delta}) 
				\\
				@VV {} V @VV {} V
				\\
				H^0(X, L^k) \otimes H^0(X, L^l)  @> {\rm{Mult}}_{l, k} >> H^0(X, L^{k + l}),
			\end{CD}
		\end{equation}
		where ${\rm{Res}}_{\Delta}$ is the restriction morphism to $\Delta \subset X \times X$, defined analogously to (\ref{eq_res_map}).
		Theorem \ref{thm_2_step_121} now follows directly from Theorem \ref{thm_ot_asymp} in its form (\ref{eq_reform_thm_ot_asymp}), applied for $X := X \times X$, $Y := \Delta$, $L_0 := \pi_0^* L$, $L_1 := \pi_1^* L$; $h^{L_0}, h^{L_1} := h^{\mathcal{T}}_t$, Lemma \ref{lem_gr_geod_ray} and (\ref{eq_comm_diag}).
	\end{proof}
	
	\subsection{Resolution of singularities, filtrations and geodesic rays}\label{sect_nat_oper}
	The main goal of this section is to prove Theorem \ref{thm_1_step}.
	For this, we first recall some natural operations on the set of test configurations, which transform any pair of ample test configurations into the one as in Theorem \ref{thm_1_step}, and then we establish that these natural operations do not perturb the validity of Theorem \ref{thm_dist_na}.
	\par 
	We first study how the filtration associated with a test configuration changes under the normalization.
	Let $\mathcal{T} = (\pi: \mathcal{X} \to \comp, \mathcal{L})$ be an arbitrary ample test configuration of a polarized pair $(X, L)$.
	As in Section \ref{sect_filt}, we consider the normalization $\widetilde{\mathcal{T}} = (\widetilde{\pi} : \widetilde{\mathcal{X}} \to \comp, \widetilde{\mathcal{L}})$ of $\mathcal{T}$.
	We denote by $\widetilde{\mathcal{F}}^{\mathcal{T}}$ (resp. $\mathcal{F}^{\mathcal{T}}$) the filtration on $R(X, L)$ associated with $\widetilde{\mathcal{T}}$ (resp. $\mathcal{T}$).
	For $k \in \nat^*$, we denote by $\mathcal{F}^{\mathcal{T}}_k, \widetilde{\mathcal{F}}^{\mathcal{T}}_k$ the filtrations induced on the graded pieces $H^0(X, L^k)$ of $R(X, L)$.
	\begin{thm}\label{thm_compar_normal}
		There is $C > 0$, such that for any $k \in \nat^*$, we have
		\begin{equation}
			d_{+ \infty}(\mathcal{F}^{\mathcal{T}}_k, \widetilde{\mathcal{F}}^{\mathcal{T}}_k) \leq C.
		\end{equation}		 
	\end{thm}
	\begin{rem}
		For $k$ sufficiently divisible, Theorem \ref{thm_compar_normal} was established by Boucksom-Jonsson \cite[Theorem 2.3]{BouckJohn21} using non-Archimedean geometry. 
		Our functional-analytic proof is different. 
	\end{rem}
	\begin{proof}
		Recall first that for any submultiplicative (Archimedean or non-Archimedean) norm $N = \| \cdot \|$ in the sense (\ref{eq_subm_s_ring}) on a ring $A$, we can construct the homogenization (semi)norm $N^{\rm{hom}} = \| \cdot \|^{\rm{hom}}$ on $A$ in the following manner
		\begin{equation}\label{eq_homog}
			\| f \|^{\rm{hom}}
			:=
			\lim_{k \to \infty} \| f^k \|^{\frac{1}{k}}, 
			\qquad
			f \in A.
		\end{equation}
		The above limit exists by submultiplicativity of $N$ and Fekete's lemma. 
		\par 
		\begin{sloppypar}
		Now, taking into account the relation between filtrations on vector spaces and non-Archimedean norms as in (\ref{eq_filtr_norm}), we define the filtration $\mathcal{F}^{\mathcal{T} \rm{hom}}$ on $R(X, L)$ in such a way that 
		$
			\chi_{\mathcal{F}^{\rm{hom}}} = \chi_{\mathcal{F}}^{\rm{hom}}
		$, where $\chi_{\mathcal{F}}$, $\chi_{\mathcal{F}^{\rm{hom}}}$ are the non-Archimedean norms associated with $\mathcal{F}^{\mathcal{T}}$ and $\mathcal{F}^{\mathcal{T} \rm{hom}}$ respectively.
		By submultiplicativity of $\chi_{\mathcal{F}}^{\rm{hom}}$, the filtration $\mathcal{F}^{\mathcal{T} \rm{hom}}$ is submultiplicative.
		\end{sloppypar}
		\par 
		From Boucksom-Jonsson \cite[Lemma A.7]{BouckJohn21}, we have
		$
			\chi_{\mathcal{F}^{\rm{hom}}}
			\leq
			\chi_{\widetilde{\mathcal{F}}}
			\leq
			\chi_{\mathcal{F}}
		$, where $\chi_{\widetilde{\mathcal{F}}}$ is the non-Archimedean norm associated with $\widetilde{\mathcal{F}}^{\mathcal{T}}$. 
		Hence, in order to establish Theorem \ref{thm_compar_normal}, it suffices to prove that there is $C > 0$, such that for any $k \in \nat^*$, we have 
		\begin{equation}\label{eq_homog00}
			d_{+ \infty}(\mathcal{F}^{\mathcal{T}}_k, \mathcal{F}^{\mathcal{T} \rm{hom}}_k) \leq C.
		\end{equation}
		We will establish (\ref{eq_homog00}) using our study of geodesic rays of Hermitian norms.
		\par 
		We fix a positive smooth metric $h^L_0$ on $L$ and denote by $N^{\mathcal{T}}_{t, k}$, $t \in [1, +\infty[$, the ray of norms emanating from ${\rm{Ban}}^{\infty}_k(h^L_0)$ associated with $\mathcal{F}^{\mathcal{T}}_k$ as in (\ref{eq_ray_norm_defn0}).
		We denote by $N^{\mathcal{T}}_t = \sum_{k = 0}^{\infty} N^{\mathcal{T}}_{t, k}$ the associated graded ray of norms on $R(X, L)$.
		As described in the proof of Theorem \ref{thm_dp_ray_norms_herm}, the graded norms $N^{\mathcal{T}}_t$, $t \in [0, +\infty[$, are submultiplicative in the sense (\ref{eq_subm_s_ring}).
		We denote by $N^{\mathcal{T} \rm{hom}}_t = \sum_{k = 0}^{\infty} N^{\mathcal{T} \rm{hom}}_{t, k}$, $t \in [0, +\infty[$, $N^{\mathcal{T} \rm{hom}}_t := \| \cdot \|^{\mathcal{T} \rm{hom}}_t$, the graded ray of norms on $R(X, L)$ associated with $\mathcal{F}^{\mathcal{T} \rm{hom}}$.
		We argue that the following inequalities take place
		\begin{equation}\label{eq_hom_n_ray_com}
			(N^{\mathcal{T}}_t)^{\rm{hom}}
			\leq
			N^{\mathcal{T} \rm{hom}}_t
			\leq
			N^{\mathcal{T}}_t.
		\end{equation}
		The upper bound (\ref{eq_hom_n_ray_com}) follows from the trivial fact that $\chi_{\mathcal{F}^{\rm{hom}}} \leq \chi_{\mathcal{F}}$ and the fact that the construction of geodesic rays from (\ref{eq_ray_norm_defn0}) is monotone in an obvious sense.
		To prove the lower bound of (\ref{eq_hom_n_ray_com}), we take $f \in H^0(X, L^k)$ with a decomposition $f = \sum_{i = 1}^N f_i$, $f_i \in H^0(X, L^k)$.
		By the definition of $\mathcal{F}^{\mathcal{T} \rm{hom}}$, for any $\epsilon > 0$, there is $l \in \nat^*$ such that for any $k \geq l$, $i = 1, \ldots, N$, we have
		\begin{equation}\label{eq_hon_appr_ele}
			\chi_{\mathcal{F}}(f_i^k)^{\frac{1}{k}} 
			\leq 
			\exp(\epsilon)
			\cdot
			\chi_{\mathcal{F}^{\rm{hom}}}(f_i).
		\end{equation}
		\par 
		By the submultiplicativity of $N^{\mathcal{T}}_t$, the norm $\| f \|_t^{\mathcal{T} \rm{hom}}$ of $f$ with respect to $(N^{\mathcal{T}}_t)^{\rm{hom}}$ satisfies
		\begin{equation}\label{eq_fkt_bnd_nm00}
			\| f \|_t^{\mathcal{T} \rm{hom}}
			\leq
			(\| f^k \|_t^{\mathcal{T}})^{\frac{1}{k}},
		\end{equation}
		for any $k \in \nat^*$.
		By the definition of the norm $N^{\mathcal{T}}_t$, we have
		\begin{equation}\label{eq_fkt_bnd_nm}
			\| f^k \|_t^{\mathcal{T}}
			\leq
			\sum_{\alpha_1+ \ldots + \alpha_N = k} \frac{k!}{\alpha_1! \cdots \alpha_N!}
			\cdot
			\big \| f_1^{\alpha_1} \cdots f_N^{\alpha_N}  \big\|
			\cdot
			\chi_{\mathcal{F}} \Big( f_1^{\alpha_1} \cdots f_N^{\alpha_N} \Big)^t.
		\end{equation}
		Let $M > 0$ be the minimal constant such that for any $i = 1, \ldots, N$ and $r = 1, \ldots, l$, we have
		\begin{equation}\label{eq_fkt_bnd_nm2}
			\chi_{\mathcal{F}}(f_i^r)^{\frac{1}{r}} 
			\leq 
			\exp(M)
			\cdot
			\chi_{\mathcal{F}^{\rm{hom}}}(f_i).
		\end{equation}
		Then by (\ref{eq_hon_appr_ele}), (\ref{eq_fkt_bnd_nm}), (\ref{eq_fkt_bnd_nm2}) and submultiplicativity of $\mathcal{F}$ and $N$, we obtain that 
		\begin{equation}\label{eq_fkt_bnd_nm11}
			\| f^k \|_t^{\mathcal{T}}
			\leq
			\exp(M N l)
			\cdot
			\exp(\epsilon k t)
			\cdot
			\Big( 
			\sum_{i = 1}^{N}
			\| f_i \|
			\cdot
			\chi_{\mathcal{F}^{\rm{hom}}}(f_i)^t
			\Big)^k.
		\end{equation}
		From (\ref{eq_fkt_bnd_nm00}) and (\ref{eq_fkt_bnd_nm11}), we have
		\begin{equation}
			\| f \|_t^{\mathcal{T} \rm{hom}}
			\leq
			\exp(M N l / k)
			\cdot
			\exp(\epsilon t)
			\cdot
			\sum_{i = 1}^{N}
			\| f_i \|
			\cdot
			\chi_{\mathcal{F}^{\rm{hom}}}(f_i)^t.
		\end{equation}
		As $\epsilon > 0$ can be made arbitrarily small, and the first factor tends to $0$, as $k \to \infty$, we deduce from the very definition of the norm $N^{\mathcal{T} \rm{hom}}_t$ the lower bound of (\ref{eq_hom_n_ray_com}).
		\par 
		Remark now that for trivial reasons, for any Hermitian metric $h^L$ on $L$, the graded norm  ${\rm{Ban}}^{\infty}(h^{\mathcal{T}})$ is homogeneous, i.e. it satisfies ${\rm{Ban}}^{\infty}(h^L) = ({\rm{Ban}}^{\infty}(h^L))^{\rm{hom}}$.
		From this, (\ref{eq_hom_n_ray_com}) and the second part of Theorem \ref{thm_2_step}, we conclude that there are $C > 0$, $k_0 \in \nat^*$, such that for any $t \in [0, +\infty[$, $k \geq k_0$, we have
		\begin{equation}\label{eq_compar_normal4}
			d_{+ \infty} \big(
			N^{\mathcal{T} \rm{hom}}_{t, k},
			N^{\mathcal{T}}_{t, k}
			\big)
			\leq
			C(t + k).
		\end{equation}
		By dividing both sides of (\ref{eq_compar_normal4}) by $t$ and taking limit $t \to \infty$, from Lemma \ref{lem_two_norms_comp0} and (\ref{eq_d_p_fil_norms_herm}), we deduce (\ref{eq_homog00}), which finishes the proof.
	\end{proof}
	\par 
	Now, let us fix two test configurations $\mathcal{T}_1 = (\pi_1 : \mathcal{X}_1 \to \comp, \mathcal{L}_1)$, $\mathcal{T}_2 = (\pi_2 : \mathcal{X}_2 \to \comp, \mathcal{L}_2)$ of $(X, L)$, such that both $\mathcal{X}_1$ and $\mathcal{X}_2$ are smooth.
	For the standard coordinate $\tau$ on $\comp$, we denote by $X_{1, \tau}$ and $X_{2, \tau}$ the fibers of $\pi_1$ and $\pi_2$ at $\tau$.
	Let $h^{\mathcal{L}}_1$ and $h^{\mathcal{L}}_2$ be some fixed Hermitian metric on $\mathcal{L}_1$ and $\mathcal{L}_2$, and let $\omega_1$, $\omega_2$ be some fixed Kähler forms on $\mathcal{X}_1$, $\mathcal{X}_2$.
	We argue that there are $C > 0$, $N \in \nat$, such that under the natural identification of $(\mathcal{X}_1 \setminus \pi_1^{-1}(0), \mathcal{L}_1|_{\mathcal{X}_1 \setminus \pi_1^{-1}(0)})$ and $(\mathcal{X}_2 \setminus \pi_2^{-1}(0), \mathcal{L}_2|_{\mathcal{X}_2 \setminus \pi_2^{-1}(0)})$ with $(\comp^* \times X, \comp^* \times L)$, for any $0 < |\tau| < 1, x \in X$, we have
	\begin{equation}\label{eq_compar_metr_tc}
	\begin{aligned}
		&
		\exp(-C) \cdot |\tau|^N \cdot \omega_2(\tau, x) \leq \omega_1(\tau, x) \leq \exp(C) \cdot |\tau|^{-N} \cdot \omega_2(\tau, x),
		\\
		&
		\exp(-C) \cdot |\tau|^N \cdot h^{\mathcal{L}}_2|_{X_{2, \tau}}(x) \leq h^{\mathcal{L}}_1|_{X_{1, \tau}}(x) \leq \exp(C) \cdot |\tau|^{-N} \cdot h^{\mathcal{L}}_2|_{X_{2, \tau}}(x).
	\end{aligned}
	\end{equation}
	The proofs of the two statements are identical, so we only concentrate on establishing the latter one. 
	It follows directly from the fact that the test configurations $\mathcal{T}_1$, $\mathcal{T}_2$ can be dominated by a third test configuration, cf. Section \ref{sect_filt}, and the pull-backs of $\mathcal{L}_1$ and $\mathcal{L}_2$ to this third test configuration are isomorphic up to a multiplication by a line bundle associated with a divisor with support in the central fiber.
	\par 
	We will now study the behavior of filtrations and geodesic rays under a resolution of singularities of test configurations.
	Let us fix an arbitrary ample test configuration $\mathcal{T} = (\pi: \mathcal{X} \to \comp, \mathcal{L})$ and an arbitrary $\comp^*$-equivariant resolution of singularities $\mathcal{T}' = (\pi': \mathcal{X}' \to \comp, \mathcal{L}')$ of $\mathcal{T}$.
	Remark  that the line bundle $\mathcal{L}'$ on $\mathcal{X}'$ is no longer ample.
	However, by \cite[Lemma 3]{PhongSturmDirMA}, there are $r_0 \in \nat$ and a line bundle $\mathcal{M}^*$ on $\mathcal{X}'$, given by a $\mathbb{N}$-linear combination of line bundles corresponding to some divisorial irreducible components of the central fiber, such that for any $r \geq r_0$, the line bundle $\mathcal{L}'{}^r \otimes \mathcal{M}$ is ample, where $\mathcal{M}$ is the dual of $\mathcal{M}^*$.
	For $r \geq r_0$, we consider a sequence of ample test configurations $\mathcal{T}(r) = (\pi: \mathcal{X}' \to \comp, \mathcal{L}'{}^r \otimes \mathcal{M})$.
	We regard $\mathcal{T}(r)$ as ample approximations of $\mathcal{T}$.
	\par 
	We denote by  $h^{\mathcal{T}(r)}_t$, $t \in [0, + \infty[$, the geodesic ray of metrics associated with $\mathcal{T}(r)$, emanating from $(h^L_0)^r \otimes h^M$, where $h^L_0$ is some fixed smooth positive metric on $L$ and $h^M$ is some fixed metric on $\mathcal{M}|_{X_1}$. 
	Since the line bundle $\mathcal{M}$ is canonically trivial over $\mathcal{X}' \setminus X'_0$, we may view $h^{\mathcal{T}(r)}_t$ as a ray of metrics on $L^r$.
	We argue that there is $C > 0$, such that for any $r \geq r_0$, there is $C(r) > 0$, such that for any $t \in [0, +\infty[$, the following estimate holds
	\begin{equation}\label{eq_alg_appr_ample}
		\Big|
			d_{+ \infty} \big( h^{\mathcal{T}_1(r)}_t, (h^{\mathcal{T}_1}_t)^r \big)
		\Big|
		\leq
		Ct + C(r).
	\end{equation}
	\par 
	To see this, by Theorem \ref{thm_ph_st_regul}, we know that $h^{\mathcal{T}(r)}_t$ come from a bounded metric on $\mathcal{L}'{}^r \otimes \mathcal{M}$.
	By the same result, the metric $(h^{\mathcal{T}}_t)^r$ comes from a bounded metrics on $\mathcal{L}'{}^r$.
	Since the line bundle $\mathcal{M}$ is associated with a divisor which has support on the central fiber, we deduce (\ref{eq_alg_appr_ample}) similarly to (\ref{eq_compar_metr_tc}) by the relation $t = - \log |\tau|$.
	\par 
	Now, on the level of filtrations, we argue that there is $C > 0$, such that for any $r \geq r_0$, $k \in \nat^*$, the following estimate holds
	\begin{equation}\label{eq_alg_appr_ample21}
		\Big|
			d_{+ \infty} \big( \mathcal{F}^{\mathcal{T}(r)}_k, r \mathcal{F}^{\mathcal{T}}_k \big)
		\Big|
		\leq
		C k,
	\end{equation}
	where by $r \mathcal{F}^{\mathcal{T}}_k$ we mean the filtration on $R(X, L)$ such that its weights correspond to the weights of $\mathcal{F}^{\mathcal{T}}_k$, multiplied by $r$.
	Indeed, to see this, remark that there is $N \in \mathbb{N}$, such that $\mathscr{O}_{\mathcal{X}}(N X'_0) \otimes \mathcal{M}$ corresponds to a divisor, given by a $\mathbb{N}$-linear combination of irreducible components of the central fiber.
	From this, we see directly from the construction of the filtrations, see Section \ref{sect_filt}, that the weights of the filtration $\mathcal{F}^{\mathcal{T}(r)}_k$ are bigger than the weights of the filtration $\mathcal{F}^{\mathcal{T}[r]}_k$ associated with the test configuration $r\mathcal{T}[-N] = (\pi: \mathcal{X}' \to \comp, \mathcal{L}'{}^r \otimes \mathscr{O}_{\mathcal{X}'}(-N X'_0))$.
	Since the weights of the filtration $r \mathcal{F}^{\mathcal{T}}_k$ are bigger than the weights of $\mathcal{F}^{\mathcal{T}(r)}_k$ by a similar reason, we conclude that 
	\begin{equation}\label{eq_alg_appr_ample22}
		d_{+ \infty} \big( \mathcal{F}^{\mathcal{T}(r)}_k, r \mathcal{F}^{\mathcal{T}}_k \big)
		\leq
		d_{+ \infty} \big( \mathcal{F}^{r\mathcal{T}[-N]}_k, r \mathcal{F}^{\mathcal{T}}_k \big).
	\end{equation}
	But directly from the construction of the filtrations, we see that the weights of the filtration $\mathcal{F}^{r\mathcal{T}[-N]}_k$ correspond to the weights of the filtration $\mathcal{F}^{\mathcal{T}}_k$ multiplied by $r$ and translated by $N k$, cf. \cite[Example 1.6]{BouckJohn21} or (\ref{eq_shift_test_1}).
	From this and (\ref{eq_alg_appr_ample22}), we obtain (\ref{eq_alg_appr_ample21}) for $C := N$.
	\par 
	We are finally ready to present the proof of the main result of this section.
	\begin{proof}[Proof of Theorem \ref{thm_1_step}]
		We first argue that Theorem \ref{thm_dist_na} holds for test configurations $\mathcal{T}_1$ and $\mathcal{T}_2$ if and only if it holds for their normalizations $\widetilde{\mathcal{T}}_1$ and $\widetilde{\mathcal{T}}_2$.
		Indeed, since the construction of geodesic rays of metrics is done on desingularizations, the left-hand side of (\ref{eq_dist_na}) doesn't change if one passes to normalizations.
		The right-hand side of (\ref{eq_dist_na}) doesn't change either by Theorem \ref{thm_compar_normal}.
		Without loss of generality, we will, hence, assume from now on that both test configurations $\mathcal{T}_1$ and $\mathcal{T}_2$ are normal.
		\par 
		By semistable reduction theorem, see \cite[\S 2]{ToroidEmb}, for any normal flat projective scheme $\pi: \mathcal{X} \to \comp$, which is a submersion away from the zero fiber, after a finite base-change $\comp \to \comp$, $\tau \mapsto \tau^r$, $r \in \nat^*$, and a blow-up along a sheaf of ideals, which are trivial away from the central fiber, we can obtain a semistable snc model $\pi' : \mathcal{X}' \to \comp$.
		As it was described for example in \cite[Lemma 5]{LiXuSpecial}, by equivariant Hironaka's resolution of singularities theorem, semistable reduction can be performed in equivariant setting, and for any test configuration $\mathcal{T} = (\pi: \mathcal{X} \to \comp, \mathcal{L})$ it yields a test configuration $\mathcal{T}' = (\pi: \mathcal{X}' \to \comp, \mathcal{L}')$, such that $\mathcal{X}' \to \comp$ is a semistable snc model.
		\par 
		Let us prove first that it is sufficient to verify Theorem \ref{thm_dist_na} when the test configurations are given by semistable snc models (not necessarily ample).
		By the above, it is sufficient to verify that both sides of (\ref{eq_dist_na}) change equally under the base-change $\comp \to \comp$, $\tau \mapsto \tau^r$, $r \in \nat$ and blow-ups along sheafs of ideals, which are trivial away from the central fibers.
		\par 
		Clearly, both sides of (\ref{eq_dist_na}) change equally under the base-change $\comp \to \comp$, $\tau \mapsto \tau^r$, $r \in \nat$.
		Indeed, a pull-back of the solution of (\ref{eq_ma_geod_dir}) under this base-change will still be a solution of (\ref{eq_ma_geod_dir}), and due to the relation $t = - \log|\tau|$, the base-change is equivalent to the homothety on the time of the ray $h^{\mathcal{T}}_t \mapsto h^{\mathcal{T}}_{rt}$.
		On another hand, the homothety with the same factor appears for filtrations on the right-hand side of (\ref{eq_dist_na}), i.e. the filtration $\mathcal{F}^{\mathcal{T}}$ changes to $r \mathcal{F}^{\mathcal{T}}$ cf. \cite[Lemma A.2]{BouckJohn21}.
		\par 
		Now, as we explained in Section \ref{sect_filt}, the solutions of (\ref{eq_ma_geod_dir}) are compatible with respect to different resolutions of singularities.
		Hence, the left-hand side of (\ref{eq_dist_na}) doesn't change if we perform a blow-up along a sheaf of ideals, which is trivial away from the central fiber. 
		Zariski's main theorem (which we can apply by the normality of our test configurations) implies that the filtrations associated with our test configurations do not change if we perform a blow-up along a sheaf of ideals, which is trivial away from the central fiber, see \cite[Lemma A.4]{BouckJohn21}.
		Hence the right-hand side of (\ref{eq_dist_na}) doesn't change under this operation either.
		\par 
		Since any two test configurations can be dominated by a third one by performing blow-ups along sheafs of ideals, which are trivial away from the central fibers, cf. Section \ref{sect_filt}, and these blow-ups do not change the validity of Theorem \ref{thm_dist_na}, we can assume that $\mathcal{T}_1 = (\pi: \mathcal{X} \to \comp, \mathcal{L}_1)$ and $\mathcal{T}_2 = (\pi: \mathcal{X} \to \comp, \mathcal{L}_2)$, where $\pi$ is a semistable snc model.
		\par 
		Now, remark that after making blow-ups, our test configurations are no longer ample but only semiample. 
		We denote by $\mathcal{M}_1, \mathcal{M}_2$ the auxiliary line bundles on $\mathcal{X}$ as described before (\ref{eq_alg_appr_ample}), i.e. such that for $r \in \nat$ sufficiently big ($r \geq r_0$), the line bundles $\mathcal{L}_1^r \otimes \mathcal{M}_1$ and $\mathcal{L}_2^r \otimes \mathcal{M}_2$ are ample.
		We consider a sequence of test configurations $\mathcal{T}_1(r) = (\pi: \mathcal{X} \to \comp, \mathcal{L}_1^r \otimes \mathcal{M}_1)$ and $\mathcal{T}_2(r) = (\pi: \mathcal{X} \to \comp, \mathcal{L}_2^r \otimes \mathcal{M}_2)$ associated with $\mathcal{T}_1, \mathcal{T}_2$ as described before (\ref{eq_alg_appr_ample}).
		We will now show that the chordal distances between the geodesic rays of metrics associated with $\mathcal{T}_1(r)$ and $\mathcal{T}_2(r)$ approximate, as $r \to \infty$, the chordal distance between the geodesic rays of metrics associated with $\mathcal{T}_1$ and $\mathcal{T}_2$, and the same holds true for distances between the respective filtrations.
		\par 
		We denote by  $h^{\mathcal{T}_1(r)}_t$ and $h^{\mathcal{T}_2(r)}_t$, $t \in [0, +\infty[$, the geodesic rays of Hermitian metrics on $L$ associated with $\mathcal{T}_1(r)$ and $\mathcal{T}_2(r)$, emanating from $(h^L_0)^r \otimes h^M$ and $(h^L_0)^r \otimes h^{M'}$, where $h^L_0$, $h^M$ and $h^{M'}$ are some fixed metrics on $L$, $\mathcal{M}_1|_{X_1}$ and $\mathcal{M}_2|_{X_1}$ respectively. 
		As in (\ref{eq_alg_appr_ample}), we view $h^{\mathcal{T}_1(r)}_t$ and $h^{\mathcal{T}_2(r)}_t$ as rays of metrics on $L^r$.
		Since from (\ref{eq_finsl_dist_fir}), we have $d_p ( (h^{\mathcal{T}_1}_t)^r, (h^{\mathcal{T}_2}_t)^r ) = r d_p ( h^{\mathcal{T}_1}_t, h^{\mathcal{T}_2}_t )$, by (\ref{eq_alg_appr_ample}), we conclude that 
		\begin{equation}\label{eq_alg_appr_ample1}
			\lim_{r \to \infty} \frac{d_p \big(\{ h^{\mathcal{T}_1(r)}_t \}, \{ h^{\mathcal{T}_2(r)}_t \} \big)}{r}
			=
			d_p \big(\{ h^{\mathcal{T}_1}_t \}, \{ h^{\mathcal{T}_2}_t \} \big).
		\end{equation}
		On the level of filtrations, a similar result holds.
		In fact, by (\ref{eq_alg_appr_ample21}), we have
		\begin{equation}\label{eq_alg_appr_ample10}
			\lim_{r \to \infty} \frac{d_p \big(\mathcal{F}^{\mathcal{T}_1(r)}, \mathcal{F}^{\mathcal{T}_2(r)} \big)}{r}
			=
			d_p \big( \mathcal{F}^{\mathcal{T}_1}, \mathcal{F}^{\mathcal{T}_2} \big).
		\end{equation}
		From (\ref{eq_alg_appr_ample1}) and (\ref{eq_alg_appr_ample10}), we see that if Theorem \ref{thm_dist_na} holds for $\mathcal{T}_1(r)$ and $\mathcal{T}_2(r)$ for any $r \geq r_0$, then it holds for $\mathcal{T}_1$ and $\mathcal{T}_2$.
		This finishes the proof of Theorem \ref{thm_1_step}, as the test configurations $\mathcal{T}_1(r)$ and $\mathcal{T}_2(r)$ are exactly as required in Theorem \ref{thm_1_step}.
	\end{proof}

	\subsection{Distance between $L^2$-norms on families of degenerating manifolds}\label{sect_dist_snc}
	The main goal of this section is to estimate the distance between the $L^2$-norms associated with the geodesic rays of metrics in terms of the distance between the geodesic rays of metrics themselves, i.e. to establish Theorem \ref{thm_3_step}.
	\par 
	\begin{sloppypar}
	Let us first recall the definition of Berezin-Toeplitz operators and the canonical $L^2$-norm.
	Let $(X, L)$ be a polarized projective manifold, $h^L$ be a positive smooth Hermitian metric on $L$.
	We fix a Hermitian vector bundle $(E, h^E)$ on $X$.
	We denote by $K_X = \wedge^n T^* X$ the canonical line bundle over $X$.
	On $H^0(X, L^k \otimes E \otimes K_X)$, we define the natural $L^2$-norm ${\rm{Hilb}}_k^{\rm{can}}(h^L, h^E) = \| \cdot \|_{L^2_{\rm{can}}(h^L, h^E)}$, given for $s \in H^0(X, L^k \otimes E \otimes K_X)$ as follows
	\begin{equation}\label{eq_defn_can_prod}
		\| s \|_{L^2_{\rm{can}}(h^L, h^E)}^2
		=
		\int_X |s(x) \wedge \overline{s}(x)|_{(h^L)^k \otimes h^E}.
	\end{equation}
	Remark that we don't need to fix the volume form on $X$ to define (\ref{eq_defn_can_prod}).
	Recall, however, that any Hermitian metric $h^K$ on $K_X$ defines canonically the positive volume form $dV_{h^K}$ in the following way.
	For any $x \in X$, we put
	\begin{equation}\label{eq_vol_f_defn_sm}
		dV_{h^K}(x) 
		=
		(-\imun)^{n^2 + 2n} dz_1 \wedge \cdots \wedge dz_n \wedge d\overline{z}_1 \wedge \cdots \wedge d\overline{z}_n,
	\end{equation}
	where $| dz_1 \wedge \cdots \wedge dz_n |_{h^K}(x) = 1$.
	In this way, one can verify that ${\rm{Hilb}}_k^{\rm{can}}(h^L, h^E)$ coincides with the classical $L^2$-norm induced by $h^L, h^E, h^K$ and $dV_{h^K}$.
	\par 
	For a given $f \in L^{\infty}(X)$, we define the associated Berezin-Toeplitz operator $T_k(f) \in \enmr{H^0(X, L^k \otimes E \otimes K_X)}$, $k \in \nat$, as follows
	\begin{equation}
		(T_k(f))(g) := B_k (f \cdot g), \quad g \in H^0(X, L^k \otimes E \otimes K_X),
	\end{equation}
	where $B_k$ is the orthogonal (Bergman) projection from $L^2(X, L^k \otimes E \otimes K_X)$ onto $H^0(X, L^k \otimes E \otimes K_X)$ with respect to ${\rm{Hilb}}_k^{\rm{can}}(h^L, h^E)$.
	\end{sloppypar}
	\par 
	It turns out that for two bounded psh metrics $h^L_0$, $h^L_1$ on $L$, the $d_p$-distances, $p \in [1, +\infty[$, between ${\rm{Hilb}}_k^{\rm{can}}(h^L_0, h^E)$ and ${\rm{Hilb}}_k^{\rm{can}}(h^L_1, h^E)$ can be estimated using traces of Berezin-Toeplitz operators.
	More precisely, we denote by $h^L_t$, $t \in [0, 1]$, the distinguished geodesic between $h^L_0$ and $h^L_1$, and let $\dot{h}^L_0$ be its derivative at zero, defined as in (\ref{eq_bnd_darvas_sup}), and viewed as a function on $X$, which is bounded by (\ref{eq_bnd_darvas_sup}).
	\begin{lem}\label{lem_d_p_toepl}
		If the Hermitian line bundle $(E, h^E)$ is positive, and the metrics $h^L_0$ and $h^L_1$ are ordered as $h^L_0 \leq h^L_1$, then for any $k \in \nat$, $p \in [1, +\infty[$, we have
		\begin{equation}
			d_p \Big( {\rm{Hilb}}_k^{\rm{can}}(h^L_0, h^E), {\rm{Hilb}}_k^{\rm{can}}(h^L_1, h^E) \Big)
			\leq
			k \cdot
			\sqrt[p]{\frac{\tr{T_{k}(|\dot{h}^L_0|^p)} }{\dim H^0(X, L^k \otimes E \otimes K_X)}}.
		\end{equation}
	\end{lem}
	\begin{rem}
		A version of this result was established by Darvas-Lu-Rubinstein in \cite[p. 25]{DarvLuRub}, following prior estimates of Berndtsson \cite[(3.2)]{BerndtProb}.
	\end{rem}
	\begin{proof} 
		Quantized maximum principle of Berndtsson \cite[Proposition 3.1]{BerndtProb} states that the distinguished geodesic $H_{k, t}$, $t \in [0, 1]$, between ${\rm{Hilb}}_k^{\rm{can}}(h^L_0, h^E)$, ${\rm{Hilb}}_k^{\rm{can}}(h^L_1, h^E)$ in the space of Hermitian norms on $H^0(X, L^k \otimes E \otimes K_X)$, relates with the $L^2$-norm in the following manner
		\begin{equation}
			H_{k, t}
			\leq
			{\rm{Hilb}}_k^{\rm{can}}(h^L_t, h^E).
		\end{equation}
		The above statement uses crucially the fact that $L^2$-norms are defined on holomorhic sections of $L^k \otimes E$ twisted by $K_X$, and $(E, h^E)$ is positive.
		Remark that Berndtsson established this result under an additional regularity assumption on the path $h^L_t$, $t \in [0, 1]$, and it was later verified by Darvas-Lu-Rubinstein \cite[Proposition 2.12]{DarvLuRub} that this assumption is not necessary.
		Since both paths of norms $H_{k, t}$ and ${\rm{Hilb}}_k^{\rm{can}}(h^L_t, h^E)$, $t \in [0, 1]$, emanate from the same point, we have
		\begin{equation}\label{lem_d_p_toepl1}
			\frac{d}{dt} H_{k, t}|_{t = 0}
			\leq
			\frac{d}{dt} {\rm{Hilb}}_k^{\rm{can}}(h^L_t, h^E)|_{t = 0},
		\end{equation}
		where the partial order $A \leq B$ means that the difference $B - A$ is positive definite with respect to the norm  $H_{k, 0} = {\rm{Hilb}}_k^{\rm{can}}(h^L_0, h^E)$.
		Since the metrics $h^L_0$ and $h^L_1$ are ordered as $h^L_0 \leq h^L_1$, we have
		\begin{equation}\label{lem_d_p_toepl3}
			0 \leq \frac{d}{dt} H_{k, t}|_{t = 0}.
		\end{equation}
		\par 
		A direct calculation shows that 
		\begin{equation}\label{lem_d_p_toepl2}
			\frac{d}{dt} {\rm{Hilb}}_k^{\rm{can}}(h^L_t, h^E)|_{t = 0} = k \cdot T_k( \dot{h}^L_0 ).
		\end{equation}
		Now, from (\ref{eq_dist_transf}), we conclude that
		\begin{equation}\label{lem_d_p_toepl4}
			d_p \Big( {\rm{Hilb}}_k^{\rm{can}}(h^L_0, h^E), {\rm{Hilb}}_k^{\rm{can}}(h^L_1, h^E) \Big)
			=
			\sqrt[p]{\frac{{\rm{Tr}}[|\frac{d}{dt} H_{k, t}|_{t = 0}|^p]}{\dim H^0(X, L^k \otimes E \otimes K_X)}}.
		\end{equation}
		From (\ref{lem_d_p_toepl1}), (\ref{lem_d_p_toepl2}), (\ref{lem_d_p_toepl3}) and (\ref{lem_d_p_toepl4}), we deduce 
		\begin{equation}
			d_p \Big( {\rm{Hilb}}_k^{\rm{can}}(h^L_0, h^E), {\rm{Hilb}}_k^{\rm{can}}(h^L_1, h^E) \Big)
			\leq
			k \cdot
			\sqrt[p]{\frac{{\rm{Tr}}[T_{k}(|\dot{h}^L_0|)^p]}{\dim H^0(X, L^k \otimes E \otimes K_X)}}.
		\end{equation}
		In order to establish Lemma \ref{lem_d_p_toepl}, it is, hence, sufficient to prove that for any $p \in [1, +\infty[$, we have 
		\begin{equation}\label{lem_d_p_toepl5}
			{\rm{Tr}}[T_{k}(|\dot{h}^L_0|)^p] \leq {\rm{Tr}}[T_{k}(|\dot{h}^L_0|^p)].
		\end{equation}
		\par 
		For this, we follow an argument from Darvas-Lu-Rubinstein \cite[p. 25]{DarvLuRub}.
		Let $s_1, \ldots, s_{N_k}$, $N_k := \dim H^0(X, L^k \otimes E \otimes K_X)$ be an orthonormal basis of eigenvectors of $T_{k}(|\dot{h}^L_0|)$ with respect to the Hermitian norm ${\rm{Hilb}}_k^{\rm{can}}(h^L, h^E)$.
		Then we have 
		\begin{equation}
			{\rm{Tr}}[T_{k}(|\dot{h}^L_0|)^p]
			=
			\sum_{i = 1}^{N_k} \Big( \int_X |\dot{h}^L_0|(x) \cdot |s_i(x) \wedge \overline{s_i}(x)|_{(h^L)^k \otimes h^E} \Big)^p.
		\end{equation}
		We now use the fact that $\int_X |s_i(x) \wedge \overline{s_i}(x)|_{(h^L)^k \otimes h^E} = 1$, $i = 1, \ldots, N_k$, the convexity of the function $x \mapsto x^p$ and Jensen's inequality to deduce
		\begin{equation}\label{lem_d_p_toepl6}
			{\rm{Tr}}[T_{k}(|\dot{h}^L_0|)^p]
			\leq
			\sum_{i = 1}^{N_k}  \int_X |\dot{h}^L_0|^p(x) \cdot |s_i(x) \wedge \overline{s_i}(x)|_{(h^L)^k \otimes h^E}.
		\end{equation}
		But (\ref{lem_d_p_toepl6}) corresponds exactly to (\ref{lem_d_p_toepl5}) since the basis $s_1, \ldots, s_{N_k}$ is orthonormal.
	\end{proof}
	\par 
	There are two obstacles for using Lemma \ref{lem_d_p_toepl} in the context of Theorem \ref{thm_3_step}.
	First, to correctly estimate the trace of a Toeplitz operator, we need tight bounds for Bergman kernels.
	This is only possible under sufficient regularity assumptions on metrics on $L$, which we don't have for geodesic rays, as they are only $\mathscr{C}^{1, 1}$ in general.
	Moreover, in Lemma \ref{lem_d_p_toepl}, the canonical line bundle plays a crucial role, but it doesn't appear in the $L^2$-norm from Theorem \ref{thm_3_step}.
	The second obstacle is that the geodesic rays $h^{\mathcal{T}_1}_t, h^{\mathcal{T}_2}_t$, $t \in [0, +\infty[$, associated with two test configurations $\mathcal{T}_1, \mathcal{T}_2$ are not necessarily ordered as it is required for $h^L_0, h^L_1$ from Lemma \ref{lem_d_p_toepl}.
	\par 
	To overcome the first obstacle, we need to replace the $L^2$-norms associated with the geodesic ray of Hermitian metrics on the line bundle by the the $L^2$-norms associated with a ray of smooth Hermitian metrics on the line bundle.
	More precisely, let $\mathcal{T} = (\pi: \mathcal{X} \to \comp, \mathcal{L})$ be a test configuration such that $(\pi, \mathcal{L})$ is an ample semistable snc model.
	We denote by $K_{\mathcal{X} / \comp}$ the relative canonical line bundle of $\pi: \mathcal{X} \to \comp$, defined as $K_{\mathcal{X} / \comp} = K_{\mathcal{X}} \otimes K_{\comp}^*$, where $K_{\mathcal{X}} = \wedge^{n + 1} T^* \mathcal{X}$ and $K_{\comp} = T^* \comp$ are the canonical line bundles of $\mathcal{X}$ and $\comp$.
	Let us fix a smooth Hermitian metric $h^K_{\mathcal{X} / \comp}$ on $K_{\mathcal{X} / \comp}$ and a Kähler form $\omega$ on $X$.
	We denote by $h^{\mathcal{T}}_t$, $t \in [0, +\infty[$, the geodesic ray of metrics on the line bundle, and by $h^{\mathcal{T} {\rm{sm}}}_t$, $t \in [0, +\infty[$, the ray of smooth positive metrics on $L$ associated with some smooth positive metric on $\mathcal{L}$ in the same way as before Theorem \ref{thm_ph_st_regul}.
	\par 
	With this data, there are several ways of putting a norm on $H^0(X, L^k)$. 
	First, we have the norm ${\rm{Hilb}}_k(h^{\mathcal{T}}_t, \omega)$, induced by $h^{\mathcal{T}}_t$ and $\omega$ for any $t \in [0, +\infty[$. 
	Second, due to the canonical isomorphism $K_{\mathcal{X} / \comp}|_{X_t} = K_X$ and the trivial identification $H^0(X, L^k) = H^0(X, L^k \otimes K_X^* \otimes K_X)$, we can endow $H^0(X, L^k)$ with the norm ${\rm{Hilb}}_k^{\rm{can}}(h^{\mathcal{T} {\rm{sm}}}_t, h^{K*}_{\mathcal{X} / \comp}|_{X_t})$.	
	\begin{lem}\label{lem_repl_l2_l2can}
		There is $C > 0$, such that for any $k \in \nat$, $t \in [0, +\infty[$, we have
		\begin{equation}
			d_{+ \infty} \Big( 
			{\rm{Hilb}}_k(h^{\mathcal{T}}_t, \omega),  
			{\rm{Hilb}}_k^{\rm{can}}(h^{\mathcal{T} {\rm{sm}}}_t, h^{K*}_{\mathcal{X} / \comp}|_{X_t})
			\Big)
			\leq
			C(t + k + 1).
		\end{equation}
	\end{lem}
	\begin{proof}
		By Theorem \ref{thm_ph_st_regul}, there is $C > 0$, such that for any $k \in \nat$, $t \in [0, +\infty[$, we have
		\begin{equation}\label{eq_repl_l2_l2can1}
			d_{+ \infty} \Big( 
			{\rm{Hilb}}_k(h^{\mathcal{T}}_t, \omega),  
			{\rm{Hilb}}_k(h^{\mathcal{T} {\rm{sm}}}_t, \omega)
			\Big)
			\leq
			C k.
		\end{equation}
		Remark that $\omega$ can be viewed as a relative Kähler form of the trivial test configuration $\comp \times X \to \comp$.
		By (\ref{eq_compar_metr_tc}) and the observation after (\ref{eq_vol_f_defn_sm}), we conclude that there is $C > 0$, such that for any $k \in \nat$, $t \in [0, +\infty[$, we have
		\begin{equation}\label{eq_repl_l2_l2can2}
			d_{+ \infty} \Big( 
			{\rm{Hilb}}_k(h^{\mathcal{T} {\rm{sm}}}_t, \omega),
			{\rm{Hilb}}_k^{\rm{can}}(h^{\mathcal{T} {\rm{sm}}}_t, h^{K*}_{\mathcal{X} / \comp}|_{X_t})
			\Big)
			\leq
			C t + C.
		\end{equation}		
		A combination of (\ref{eq_repl_l2_l2can1}) and (\ref{eq_repl_l2_l2can2}) yields Lemma \ref{lem_repl_l2_l2can}.
	\end{proof}
	\par 
	Now, we explain how to bound a trace of a Toeplitz operator on a degenerating family of manifolds.
	We prefer to state this result in a greater generality, as it might be of independent interest, and its proof wouldn't be any more complicated.
	We fix an ample semistable snc model $(\pi : \mathcal{X} \to \mathbb{D}, \mathcal{L})$ and endow $\mathcal{L}$ with a smooth positive metric $h^{\mathcal{L}}$.
	Let $\mathcal{E}$ be a holomorphic vector bundle over $\mathcal{X}$ endowed with a Hermitian metric $h^{\mathcal{E}}$.
	For $\tau \in \mathbb{D}$, we denote by $X_{\tau}$ the fiber of $\pi$ at $\tau$. 
	We also use the following notations $L_{\tau} := \mathcal{L}|_{X_{\tau}}$, $h^L_{\tau} := h^{\mathcal{L}}|_{X_{\tau}}$, $E_{\tau} := \mathcal{E}|_{X_{\tau}}$, $h^E_{\tau} := h^{\mathcal{E}}|_{X_{\tau}}$.
	For $k \in \nat$, $0 < |\tau| < 1$, we denote the Berezin-Toeplitz operator (with respect to the norm ${\rm{Hilb}}_k^{\rm{can}}(h^L_{\tau}, h^E_{\tau})$) associated with $f_{\tau} \in L^{\infty}(X_{\tau})$ by $T_{k}(f_\tau) \in \enmr{H^0(X, L_{\tau}^k \otimes E_{\tau} \otimes K_{X_{\tau}})}$.
	The following result, proved in Section \ref{sect_toepl}, is a generalization of the asymptotic results of Boutet de Monvel-Guillemin \cite{BoutGuillSpecToepl} and Ma-Marinescu \cite{MaHol}, \cite{MaMarToepl} to the singular family setting.
	\begin{thm}\label{thm_unif_bound_toepl}
		For any $\epsilon, C > 0$, there is $k_0 \in \nat$, such that for any $f_{\tau} \in L^{\infty}(X_{\tau})$, $0 < | \tau | < \frac{1}{2}$, verifying $|f_{\tau}| < C$, the following holds: for any $0 < | \tau | < \frac{1}{2}$, $k \geq k_0$, we have
		\begin{equation}
			\bigg|
				\frac{\tr{T_k(f_{\tau})} }{\dim H^0(X_{\tau}, L_{\tau}^k \otimes E_{\tau} \otimes K_{X_{\tau}})}
				- 
				\frac{1}{\int_{X_{\tau}} c_1(L_{\tau})^n} \int_{X_{\tau}} f_{\tau}(x) \cdot c_1(L_{\tau}, h^L_{\tau})^{n}
			\bigg|
			\leq 
			\epsilon.
		\end{equation}
	\end{thm}
	\par 
	To overcome the second obstacle for using Lemma \ref{lem_d_p_toepl} in the context of Theorem \ref{thm_3_step}, we use the rooftop envelope, as explained below.
	Recall that for bounded metrics $h^L_0$, $h^L_1$ on $L$, we define the rooftop envelope $P(h^L_0, h^L_1)$, following Ross-Witt Nystr{\"o}m \cite{RossNystAnalTConf}, as $P(h^L_0, h^L_1)(x) := \inf \{ h^L(x) :  h^L_0, h^L_1 \leq h^L  \}$, $x \in X$, where $h^L$ runs over all bounded psh metrics on $L$.
	Clearly, $P(h^L_0, h^L_1)$ is bounded and psh.
	The Pythagorean formula from \cite{DarvasMabCompl} says
	\begin{equation}\label{eq_pyth}
		d_p(h^L_0, h^L_1)^p
		=
		d_p(h^L_0, P(h^L_0, h^L_1))^p
		+
		d_p(h^L_1, P(h^L_0, h^L_1))^p.
	\end{equation}
	\begin{lem}\label{prop_lattice_norms}
		For any two bounded Hermitian metrics $h^L_0$, $h^L_1$ on $L$, and any $p \in [1, +\infty[$, we have
		\begin{multline}
			d_p \Big( {\rm{Hilb}}_k^{\rm{can}}(h^L_0, h^E), {\rm{Hilb}}_k^{\rm{can}}(h^L_1, h^E) \Big)^p 
			\leq 
			d_p \Big( {\rm{Hilb}}_k^{\rm{can}}(h^L_0, h^E), {\rm{Hilb}}_k^{\rm{can}}(P(h^L_0, h^L_1), h^E) \Big)^p
			\\
			+ 
			d_p \Big( {\rm{Hilb}}_k^{\rm{can}}(h^L_1, h^E), {\rm{Hilb}}_k^{\rm{can}}(P(h^L_0, h^L_1), h^E) \Big)^p.
		\end{multline}
	\end{lem}		
	\begin{proof}
		Clearly, it sufficies to establish that for any Hermitian norms $H_0, H_1, H_2$ on a finitely-dimensional vector space $V$, verifying $H_0 \leq H_2$, $H_1 \leq H_2$, we have
		\begin{equation}\label{eq_lattice_norms}
			d_p(H_0, H_1)^p \leq d_p(H_0, H_2)^p + d_p(H_1, H_2)^p.
		\end{equation}
		\par 
		For this, let us fix a basis $e_1, \ldots, e_n$ of $V$, which diagonalises both $H_0$ and $H_1$.
		For any Hermitian norm $H$ on $V$, we denote by $\rho_e(H)$ the Hermitian norm on $V$ given by given by Gram-Schmidt projection, i.e. it is the unique Hermitian norm on $V$, for which the basis $e_1, \ldots, e_n$ is orthogonal and such that for any $i = 1, \ldots, n$, we have
		\begin{equation}
			\| e_i \|_{\rho_e(H)} :=
			\inf_{a_1, \ldots, a_{i - 1} \in \comp} \Big\| e_i + \sum_{j < i} a_j e_j \Big\|_{H}.
		\end{equation} 
		The key property of the Gram-Schmidt projection is that it is 1-Lipschitz, see Boucksom-Eriksson \cite[Lemma 3.4]{BouckErik21}.
		More precisely, for any Hermitian norms $H, H' \in \mathcal{H}_V$, we have
		\begin{equation}\label{eq_gr_sch_decr}
			d_p (\rho_e(H), \rho_e(H')) \leq d_p (H, H').
		\end{equation}
		From (\ref{eq_gr_sch_decr}), we see that to establish (\ref{eq_lattice_norms}), it is enough to prove that 
		\begin{equation}\label{eq_proj_lid}
			d_p(H_0, H_1)^p \leq d_p(H_0, \rho_e(H_2))^p + d_p(H_1, \rho_e(H_2))^p.
		\end{equation}
		As Gram-Schmidt projection obviously preserves order, we have $H_0 \leq \rho_e(H_2)$, $H_1 \leq \rho_e(H_2)$.
		Since all three norms $H_0, H_1, \rho_e(H_2)$ are diagonalized in the same basis, we conclude that 
		\begin{equation}\label{eq_vee_born}
			H_0 \vee H_1 \leq \rho_e(H_2),
		\end{equation}
		where $H_0 \vee H_1 = \| \cdot \|_{\vee}$ is the Hermitian norm, which is diagonalized in basis $e_1, \ldots, e_n$, and which verifies $\| e_i \|_{\vee} := \max \{ \| e_i \|_{H_0}, \| e_i \|_{H_1} \}$.
		Remark, however, that for trivial reasons, we have 
		\begin{equation}\label{eq_vee_sum}
			d_p(H_0, H_1)^p = d_p(H_0, H_0 \vee H_1)^p + d_p(H_1, H_0 \vee H_1)^p.
		\end{equation}
		We deduce (\ref{eq_proj_lid}) by (\ref{eq_vee_born}) and (\ref{eq_vee_sum}).
	\end{proof}
	\begin{rem}
		We warn the reader that despite our use of the notation $H_0 \vee H_1$, the set of Hermitian norms on a given vector space is not a lattice with respect to its natural order!
	\end{rem}
	We are now finally ready to prove the main result of this section.
	\begin{proof}[Proof of Theorem \ref{thm_3_step}]
		First of all, the statement for $p = +\infty$ is trivial, so we only concentrate on the first part of Theorem \ref{thm_3_step}. We fix $p \in [1, +\infty[$.
		By Lemma \ref{lem_repl_l2_l2can}, from which we borrow the notations, and (\ref{eq_tr_weak}), it is sufficient to establish that for any $\epsilon > 0$, there are $k_0 \in \nat$, $C > 0$, such that for any $t \in [0, +\infty[$ and $k \geq k_0$, the following bound holds
		\begin{equation}\label{eq_3_steaap}
			d_p \Big( {\rm{Hilb}}_k^{\rm{can}}(h^{\mathcal{T}_1 {\rm{sm}}}_t, h^{K*}_{\mathcal{X} / \comp}|_{X_t}), {\rm{Hilb}}_k^{\rm{can}}(h^{\mathcal{T}_2 {\rm{sm}}}_t, h^{K*}_{\mathcal{X} / \comp}|_{X_t}) \Big)
			\leq
			k \cdot d_p(h^{\mathcal{T}_1 {\rm{sm}}}_t, h^{\mathcal{T}_2 {\rm{sm}}}_t)
			+
			Ct
			+
			\epsilon k t.
		\end{equation}
		\par
		\begin{sloppypar} 
		For $t \in [0, +\infty[$, we denote $h^L_{3, t} := P(h^{\mathcal{T}_1 {\rm{sm}}}_t, h^{\mathcal{T}_2 {\rm{sm}}}_t)$.
		By definition, we have $h^{\mathcal{T}_i {\rm{sm}}}_t \leq h^L_{3, t}$, $i = 1, 2$.
		Remark that there is $k_0 \in \nat$, such that for any $t \in [0, +\infty[$, the line bundle $L^{k_0} \otimes K_X$, endowed with metrics $(h^{\mathcal{T}_i {\rm{sm}}}_t)^{k_0} \otimes h^{K*}_{\mathcal{X} / \comp}|_{X_t}$, $i = 1, 2$, is positive over $X$.
		This follows from the fact that the metrics $h^{\mathcal{T}_i {\rm{sm}}}_t$ are obtained as restrictions of positive smooth metrics, defined on the total space of the test configuration $\mathcal{X}$.
		We, hence, may apply Lemma \ref{lem_d_p_toepl} for $E := L^{k_0} \otimes K_X$, $h^E := (h^{\mathcal{T}_i {\rm{sm}}}_t)^{k_0} \otimes h^{K*}_{\mathcal{X} / \comp}|_{X_t}$ to conclude that for any $i = 1, 2$, $t \in [0, +\infty[$ and $k \geq k_0$, we have
		\begin{multline}\label{eq_dist_ample_mod1}
			d_p \Big( {\rm{Hilb}}_k^{\rm{can}}(h^{\mathcal{T}_i {\rm{sm}}}_t, h^{K*}_{\mathcal{X} / \comp}|_{X_t}), {\rm{Hilb}}_k^{\rm{can}}(h^L_{3, t}, h^{K*}_{\mathcal{X} / \comp}|_{X_t}) \Big)
			\\
			\leq
			k \cdot
			\sqrt[p]{\frac{\tr{T_k(|\dot{h}^L_{i, t, 0}|^p)} }{\dim H^0(X, L^k)}}
			+
			k_0 \cdot d_{+ \infty}(h^{\mathcal{T}_i {\rm{sm}}}_t, h^L_{3, t}),
		\end{multline}
		where $\dot{h}^L_{i, t, 0}$ is the derivative at $s = 0$ of the distinguished geodesic segment $h^L_{i, t, s}$, $s \in [0, 1]$, between $h^{\mathcal{T}_i {\rm{sm}}}_t$ and $h^L_{3, t}$, defined as in (\ref{eq_geod_as_env}), and the Toeplitz operator $T_k(|\dot{h}^L_{i, t, 0}|^p)$ is taken with respect to the norm ${\rm{Hilb}}_k^{\rm{can}}(h^{\mathcal{T}_i {\rm{sm}}}_t, h^{K*}_{\mathcal{X} / \comp}|_{X_t})$.
		\par 
		From (\ref{eq_compar_metr_tc}), we deduce that there is $C > 0$, such that for any $i = 1, 2$, $t \in [0, +\infty[$, we have
		$
			h^{\mathcal{T}_i {\rm{sm}}}_t \leq h^L_{3, t} \leq \exp(Ct) \cdot h^{\mathcal{T}_i {\rm{sm}}}_t
		$.
		From (\ref{eq_bnd_darvas_sup}), we deduce
		$
			|\dot{h}^L_{i, t, 0}| \leq C t
		$.
		From Theorem \ref{thm_unif_bound_toepl}, applied for $f_{\tau}(x) := \frac{1}{t} |\dot{h}^L_{i, t, 0}|(x)$, where $t = - \log |\tau|$, (\ref{eq_compar_metr_tc}) and (\ref{eq_dist_ample_mod1}), we conclude that for any $\epsilon > 0$, there are $k_0 \in \nat$, $C > 0$, such that for any $t \in [0, +\infty[$ and $k \geq k_0$, the following bound holds
		\begin{multline}\label{eq_dist_ample_mod3}
			d_p \Big( {\rm{Hilb}}_k^{\rm{can}}(h^{\mathcal{T}_i {\rm{sm}}}_t, h^{K*}_{\mathcal{X} / \comp}|_{X_t}), {\rm{Hilb}}_k^{\rm{can}}(h^L_{3, t}, h^{K*}_{\mathcal{X} / \comp}|_{X_t}) \Big)
			\\
			\leq
			k \cdot
			\sqrt[p]{
			\frac{1}{\int c_1(L)^n}
			\int_{X}
			|\dot{h}^L_{i, t, 0}|^p c_1(L, h^{\mathcal{T}_i {\rm{sm}}}_t)^n
			}
			+
			Ct
			+
			\epsilon  k  t.
		\end{multline}
		\end{sloppypar}
		By taking a sum of $p$-powers of (\ref{eq_dist_ample_mod3}) for $i = 1, 2$, from Lemma \ref{prop_lattice_norms}, (\ref{eq_d_p_berndss}) and (\ref{eq_pyth}), we deduce (\ref{eq_3_steaap}).
	\end{proof}

	\subsection{Toeplitz operators on families of degenerating manifolds}\label{sect_toepl}
	On a fixed projective manifold polarized with a positive line bundle, the asymptotic study of Toeplitz operators is fully understrood from the works of Boutet de Monvel-Guillemin \cite{BoutGuillSpecToepl} and Ma-Marinescu \cite{MaHol}, \cite{MaMarToepl}.
	The main goal of this section is to generalize a portion of this theory for degenerating families of manifolds, i.e. to establish Theorem \ref{thm_unif_bound_toepl}.
	This will be done by establishing the uniform version of the diagonal Bergman kernel expansion.
	\par 
	Recall that a proper holomorphic map $\pi : \mathcal{X} \to \mathbb{D}$ is a semistable snc model if and only if $\mathcal{X}$ is smooth and for every $x_0 \in X_0$, there are local holomorphic coordinates $(z_0, \ldots, z_n)$ centered at $x_0 \in \mathcal{X}$, such that for some $l = 0, 1, \ldots, n$, the map $\pi$ writes as 
	\begin{equation}\label{eq_adapted}
		\pi(z_0, \ldots, z_n)
		=
		z_0 \cdots z_l,
	\end{equation}
	see \cite[p. 99]{ToroidEmb}.
	We later call such coordinates the \textit{adapted coordinates} at $x_0$.
	Clearly, $l = 0$ if and only if $x_0$ is a regular point of $X_0$.
	We denote by $\Sigma \subset \mathcal{X}$ the subset of singular points of $X_0$.
	By (\ref{eq_adapted}), $\pi : \mathcal{X} \to \mathbb{D}$ is a submersion away from $\Sigma$.
	\par 
	The following simple lemma will allow us to localize the calculations from Theorem \ref{thm_unif_bound_toepl} away from the singular points of $X_0$.
	We conserve the notations from Theorem \ref{thm_unif_bound_toepl} for the rest of this section.
	\begin{lem}\label{lem_toepl_1}
		For any $\epsilon > 0$, there is a neighborhood $U_{\epsilon} \subset \mathcal{X}$ of $\Sigma$, such that for any $0 < |\tau| < \frac{1}{2}$, the following estimate holds
		\begin{equation}
			\Big|
				\int_{X_{\tau} \cap U_{\epsilon}} c_1(L_{\tau}, h^L_{\tau})^{n}
			\Big|
			\leq 
			\epsilon.
		\end{equation}
	\end{lem}
	\begin{proof}
		By considering adapted coordinates around a point $x_0 \in \Sigma$, we see that it is enough to prove that for a given $\epsilon > 0$, one can choose $\delta > 0$ small enough, such that for any $0 < |\tau| < \frac{1}{2}$, $k = 0, \ldots, n$, $l = 1, \ldots, n$, we have
		\begin{equation}\label{lem_toepl_122}
			\Big|
			\int_{\substack{ |z_0|, |z_1|, \ldots, |z_n| \leq \delta \\ z_0 \cdots z_l = \tau }} dz_0 \wedge d\overline{z}_0 \wedge \cdots \wedge d\hat{z}_k \wedge d\hat{\overline{z}}_k \wedge \cdots \wedge dz_n \wedge d\overline{z}_n
			\Big|
			\leq 
			\epsilon,
		\end{equation}
		where by “hats" we denoted the missing coordinates. 
		Clearly, only the case $k \leq l$ is interesting, as otherwise the integral is zero.
		We can assume for brevity that $k = 0$ and $l = n$.
		After a polar change of coordinates, the bound (\ref{lem_toepl_122}) then reduces to the one of the form
		\begin{equation}
			\Big|
			\int_{\substack{ |r_1|, \ldots, |r_n| \leq \delta \\ r_1 \cdots r_n \geq \tau / \delta }} r_1 \cdots r_n dr_1 \wedge \cdots \wedge dr_n
			\Big|
			\leq 
			\epsilon,
		\end{equation}
		which holds trivially for $\delta > 0$ small enough.
	\end{proof}
	\par 
	\begin{sloppypar}
	Let us now denote by $B_{\tau, k}$ the orthogonal (Bergman) projection from the space of $L^2$-sections $L^2(X_{\tau}, L_{\tau}^k \otimes E_{\tau} \otimes K_{X_{\tau}})$ to $H^0(X_{\tau}, L_{\tau}^k \otimes E_{\tau} \otimes K_{X_{\tau}})$ with respect to the $L^2$-scalar product ${\rm{Hilb}}_k^{{\rm{can}}}(h^L_{\tau}, h^E_{\tau})$.
	For $x, y \in X_{\tau}$, we denote by $B_{\tau, k}(x, y) \in L_{\tau, x}^k \otimes (L_{\tau, y}^*)^k \otimes K_{X_{\tau}, x} \otimes K_{X_{\tau}, y}^* \otimes E_{\tau, x} \otimes E_{\tau, y}^*$ the Bergman kernel of $B_{\tau, k}$, defined so that for any $s \in L^2(X_{\tau}, L_{\tau}^k \otimes E_{\tau} \otimes K_{X_{\tau}})$, we have $B_{\tau, k} s (x) = \int \langle B_{\tau, k}(x, y), s(y) \rangle$.
	For $x \in X_{\tau}$, we view the diagonal Bergman kernel $B_{\tau, k}(x, x)$ as an element from ${\enmr{\mathcal{E}}} \otimes \wedge^{2 n} T\mathcal{X}/\mathbb{D}$, where $T\mathcal{X}/\mathbb{D} := T\mathcal{X} / \pi^*(T\mathbb{D})$.
	\end{sloppypar}
	\begin{thm}\label{lem_toepl_2}
		There are $a_i \in \ccal^{\infty}(\mathcal{X} \setminus \Sigma, {\enmr{\mathcal{E}}} \otimes \wedge^{2 n} T\mathcal{X}/\mathbb{D})$, $i \in \nat$, such that for any $j \in \nat$, $x \in \mathcal{X} \setminus X_0$, $\tau = \pi(x)$, there are $C > 0$, $k_0 \in \nat$, such that for any $k \geq k_0$, we have
		\begin{equation}\label{eq_toepl_2}
			\Big|
				\frac{1}{k^n} B_{\tau, k}(x, x) - \sum_{i = 0}^{j} \frac{a_i(x)}{k^i}
			\Big|
			\leq
			\frac{C}{k^{j + 1}}.
		\end{equation}
		Moreover, $a_0(x) = {\rm{Id}}_{E_{\tau}} \cdot c_1(L_{\tau}, h^L_{\tau})^n$ for any $x \in \mathcal{X} \setminus \Sigma$, and $C > 0$, $k_0 \in \nat$, can be chosen uniformly for $x \in \mathcal{X} \setminus X_0$ varying over $K \setminus X_0$ for compact subsets $K$ of $\mathcal{X} \setminus \Sigma$.
	\end{thm}
	\begin{rem}\label{rem_toepl_2}
		a) In a recent preprint \cite[Theorem 1.3]{WangZhouBerg}, Wang-Zhou considered a similar degenerating context and established another estimate for Bergman kernels.
		Our methods are different.
		\par 
		b)
		The existence of the asymptotic expansion and the fact that the constants from the bound (\ref{eq_toepl_2}) can be chosen uniformly for $x$ varying over compact subsets of $\mathcal{X} \setminus X_0$ from directly from the proof of Dai-Liu-Ma \cite{DaiLiuMa} of the Bergman kernel expansion.
	\end{rem}
	Before proving Theorem \ref{lem_toepl_2}, let us explain how it entails Theorem \ref{thm_unif_bound_toepl}.
	\begin{proof}[Proof of Theorem \ref{thm_unif_bound_toepl}]
		First, we have the following trivial relation between the trace of Berezin-Toeplitz operator and the Bergman kernel:
		\begin{equation}\label{eq_thm_toepl1}
			\tr{T_{\tau, k}(f_{\tau})}
			=
			\int_{X_{\tau}} f_{\tau}(x) \cdot {\rm{Tr}}^{\mathcal{E}_x}[B_{\tau, k}(x, x)],
		\end{equation}
		where ${\rm{Tr}}^{\mathcal{E}_x}[B_{\tau, k}(x, x)]$ is the trace of the endomorphism of $\mathcal{E}_x$ part of $B_{\tau, k}(x, x)$.
		For trivial reasons
		\begin{equation}
			\int_{X_{\tau}} {\rm{Tr}}^{\mathcal{E}_x}[B_{\tau, k}(x, x)]
			=
			\tr{B_{\tau, k}}
			=
			\dim H^0(X_{\tau}, L_{\tau}^k \otimes E_{\tau} \otimes K_{X_{\tau}}).
		\end{equation}
		From this, Theorem \ref{lem_toepl_2}, Lemma \ref{lem_toepl_1} and asymptotic Riemann-Roch theorem, stating 
		\begin{equation}\label{eq_thm_toepl2}
			\dim H^0(X_{\tau}, L_{\tau}^k \otimes E_{\tau} \otimes K_{X_{\tau}}) \sim \rk{E_{\tau}} \cdot k^n \cdot \int_{X_{\tau}} c_1(L_{\tau})^n,
		\end{equation}
		we conclude that for any $\epsilon > 0$, there is a neighborhood $U_{\epsilon} \subset X$ of $\Sigma$ and $k_0 \in \nat$, such that for any $0 < |\tau| < \frac{1}{2}$, $k \geq k_0$, we have 
		\begin{equation}
			\Big|
				\int_{X_{\tau} \cap U_{\epsilon}} {\rm{Tr}}^{\mathcal{E}_x}[B_{\tau, k}(x, x)]
			\Big|
			\leq 
			\epsilon \cdot k^n.
		\end{equation}
		From this, the uniform boundness of $f_{\tau}$ and Lemma \ref{lem_toepl_1}, we conclude that for any $\epsilon > 0$, there is a neighborhood $U_{\epsilon} \subset X$ of $\Sigma$ and $k_0 \in \nat$, such that for any $0 < |\tau| < \frac{1}{2}$, $k \geq k_0$, we have 
		\begin{equation}\label{eq_thm_toepl3}
		\begin{aligned}
			&
			\Big|
				\int_{X_{\tau} \cap U_{\epsilon}}  f_{\tau}(x) \cdot  {\rm{Tr}}^{\mathcal{E}_x}[B_{\tau, k}(x, x)]
			\Big|
			\leq 
			\epsilon \cdot k^n,
			\\
			&
			\Big|
				\int_{X_{\tau} \cap U_{\epsilon}} f_{\tau}(x) \cdot c_1(L_{\tau}, h^{L}_{\tau})^{n}
			\Big|
			\leq 
			\epsilon.
		\end{aligned}
		\end{equation}
		We conclude by (\ref{eq_thm_toepl1}), (\ref{eq_thm_toepl2}), (\ref{eq_thm_toepl3}) and Theorem \ref{lem_toepl_2}.
	\end{proof}
	\par 
	To establish Theorem \ref{lem_toepl_2}, we rely on the localization property of Bergman kernels.
	The spectral gap of the Kodaira Laplacian is central for this.
	We recall it in greater generality of non-compact manifolds -- in a form which we shall use it in the proof of Theorem \ref{lem_toepl_2}.
	\par 
	We fix a complex manifold $X$ with a holomorphic line bundle $L$ over it, endowed with a smooth positive Hermitian metric $h^L$.
	We assume that the manifold $X$ is complete when endowed with the metric associated with the Kähler form $c_1(L, h^L)$.
	Let $(E, h^E)$ be a Hermitian vector bundle over $X$, such that for some $C > 0$, the curvature $R^E$ of $(E, h^E)$, satisfies 
	\begin{equation}\label{eq_re_bound_cconsrt}
		\imun R^E \geq - C \cdot {\rm{Id}}_{E} \cdot c_1(L, h^L).
	\end{equation}
	We denote by $\square_k$ the Kodaira Laplacian of $X$ acting on the space of smooth sections of $L^k \otimes E \otimes K_X$.
	In other words, $\square_k := \dbar^* \circ \dbar$, where $\dbar^*$ is the formal adjoint of $\dbar : \mathscr{C}^{\infty}(X, L^k \otimes E \otimes K_X) \to \mathscr{C}^{\infty}(X, T^{(0, 1)*}X \otimes L^k \otimes E \otimes K_X)$ with respect to the $L^2$-norm induced by $h^L$, $h^E$ and the Kähler form $c_1(L, h^L)$. 
	Since $X$ is complete when endowed with a Kähler form $c_1(L, h^L)$, the operator $\square_k$ is essentially self-adjoint, cf. \cite[Lemma D.1.1]{MaHol}.
	In particular, its spectrum is well-defined.
	\begin{lem}\label{lem_spec_gap}
		There are $\mu > 0$, $k_0 \in \nat^*$, which depend only on the constant $C$ from (\ref{eq_re_bound_cconsrt}), such that for any $k \geq k_0$, the Laplacian has the following spectral gap property 
		\begin{equation}
			{\rm{Spec}}(\square_k) \subset \{0\} \cup [\mu k, +\infty[.
		\end{equation}
	\end{lem}
	\begin{rem}
		The proof is a direct modification of \cite[Theorem 1.5.5]{MaHol}, which itself goes back to Bismut-Vasserot \cite{BVas}.
		The only new element is that $\mu$ depends only on $C$ -- a statement which is undoubtedly well-known to experts in the field.
		For this, it is crucial to consider the Laplacian acting on holomorphic sections of $L^k \otimes E$ twisted by $K_X$.
	\end{rem}
	\begin{proof}
	\begin{sloppypar}
		We define first the Laplacian $\square^1_k$ on the $(0, 1)$-forms with values in $L^k \otimes E \otimes K_X$ through the usual formula $\square^1_k := \dbar^* \circ \dbar + \dbar \circ \dbar^*$, where the adjoint is taken with respect to the $L^2$-norm induced by $h^L$, $h^E$ and the Kähler form $c_1(L, h^L)$. 
		Nakano's inequality, cf. \cite[Theorem 1.4.14 and Corollary 1.4.17]{MaHol}, states that for any $s \in \mathscr{C}^{\infty}(X, T^{(0, 1)*}X \otimes L^k \otimes E \otimes K_X)$ of compact support, we have
		\begin{equation}\label{eq_nakano}
			\langle \square^1_k s, s \rangle_{L^2} \geq 
			\frac{2}{3} \cdot 
			\sum_{j, k = 1}^{n} 
			\Big\langle (k R^L + R^E )(w_j, \overline{w}_k) \overline{w}_k \wedge \iota_{\overline{w}_j} s, s \Big\rangle_{L^2},
		\end{equation}
		where $w_i$, $i = 1, \ldots, n$, is the orthonormal basis of $T^{1, 0} X$, endowed with the Kähler form $c_1(L, h^L)$, $\iota_{\overline{w}_j}$ is the interior product with respect to $\overline{w}_j$, and $\langle \cdot, \cdot \rangle_{L^2}$ is the $L^2$-scalar product on $\mathscr{C}^{\infty}(X, T^{(0, 1)*}X \otimes L^k \otimes E \otimes K_X)$. 
		In particular, due to (\ref{eq_re_bound_cconsrt}), there is $k_0 \in \nat$, depending only on $C > 0$ from (\ref{eq_re_bound_cconsrt}), such that $\imun(k_0  R^L + R^E) \geq 0$.
		From this and (\ref{eq_nakano}), we deduce that for any $s \in \mathscr{C}^{\infty}(X, T^{(0, 1)*}X \otimes L^k \otimes E \otimes K_X)$, we have $
			\langle \square^1_k s, s \rangle_{L^2} \geq 
			\frac{2}{3} (k - k_0) \langle s, s \rangle_{L^2}
		$.
		Hence
		\begin{equation}\label{eq_spec_gap_1forms}
			{\rm{Spec}}(\square^1_k) \subset \Big[\frac{2}{3} (k - k_0), +\infty \Big[.
		\end{equation}
	\end{sloppypar}
		\par 
		Remark, that the operator $\dbar$ intervenes with Laplacians, hence, it intervenes with their spectral projections. In particular, for any eigenvector $s \in \mathscr{C}^{\infty}(X, L^k \otimes E \otimes K_X)$ of $\square_k$, the section $\dbar s$ is an eigenvector of $\square^1_k$ of the same eigenvalue.
		By this and (\ref{eq_spec_gap_1forms}), we deduce Lemma \ref{lem_spec_gap}.
	\end{proof}
	\begin{proof}[Proof of Theorem \ref{lem_toepl_2}]
		Clearly, it suffices to show that the estimate (\ref{eq_toepl_2}) is uniform in a neighborhood of any $x_0 \in X_0 \setminus \Sigma$.
		The main idea of our proof is to relate (through the localization property) the Bergman kernel on a degenerating family of manifolds with the Bergman kernel on a non-degenerating family, for which we know that the asymptotic expansion of Bergman kernels depends smoothly on the parameters of the family, see Remark \ref{rem_toepl_2}.
		\par 
		We fix adapted coordinates $h = (z_0, \ldots, z_n)$, defined in a neighborhood $U$ around $x_0$, and normalize them so that $h$ maps onto $\mathbb{D} \times \mathbb{D}_n$ in $\comp^{n + 1}$.
		We pick a local holomorphic trivialization $\sigma$ of $L$ over $U$.
		Then, under this trivialization, the Hermitian line bundle $(\mathcal{L}, h^{\mathcal{L}})|_U$ is identified with $(\comp \times \mathbb{D} \times \mathbb{D}_n, e^{- \phi})$, where $\phi$ is a plurisubharmonic smooth function over $\mathbb{D} \times \mathbb{D}_n$.
		Similarly, $(\mathcal{E}, h^{\mathcal{E}})|_U$ is identified with a holomorphically trivial vector bundle $\comp^r \times \mathbb{D} \times \mathbb{D}_n$, endowed with metric tensor $(h_{i, j})$, $i, j = 1, \ldots, r$.
		\par 
		Let us now construct a smooth function $\psi$ over $\mathbb{D} \times \comp^n$, plurisubharmonic over the fibers of $\mathbb{D} \times \comp^n \to \mathbb{D}$, such that over $\mathbb{D} \times \mathbb{D}_n(\frac{1}{4})$, $\psi$ coincides with $\phi$, and away from $\mathbb{D} \times \mathbb{D}_n$ -- with $A (|z_1|^2 + \cdots + |z_n|^2) + C$ for some constants $A > 0$, $C \in \real$, where $(z_1, \ldots, z_n)$ are the standard coordinates in $\comp^n$.
		For this, consider the function $\psi := \max_{\epsilon}(\phi, A (|z_1|^2 + \cdots + |z_n|^2- \frac{1}{3}))$, where $\max_{\epsilon}$ is the regularized maximum function, defined as in \cite[Lemma I.5.18]{DemCompl}, for $\epsilon > 0$ small enough.
		Then, for $A$ big enough, we see that $A (|z_1|^2 + \cdots + |z_n|^2 - \frac{1}{3}) < \phi$ over $\mathbb{D} \times \mathbb{D}_n(\frac{1}{2})$, and $A (|z_1|^2 + \cdots + |z_n|^2 - \frac{1}{3}) > \phi$ away from $\mathbb{D} \times \mathbb{D}_n(\frac{1}{\sqrt{2}})$.
		By the properties of the regularized maximum function from \cite[Lemma I.5.18]{DemCompl}, we see that $\psi$ is well-defined over $\mathbb{D} \times \comp^n$, it is smooth, psh over the fibers of the projection $\mathbb{D} \times \comp^n \to \mathbb{D}$ and it interpolates between $\phi$ and $A (|z_1|^2 + \cdots + |z_n|^2) - \frac{A}{3}$.
		\par 
		We also consider an arbitrary metric tensor $(h_{i, j}^0)_{i, j = 1}^{r}$ over $\mathbb{D} \times \comp^n$, such that $h^0_{i, j}$, $i, j = 1, \ldots, r$, coincide with $h_{i, j}$ over $\mathbb{D} \times \mathbb{D}_n(\frac{1}{2})$, and away from $\mathbb{D} \times \mathbb{D}_n$, they are given by $h^0_{i, j} = \delta_{i j}$, where $\delta_{i j}$ is the Kronecker delta function, giving $1$ if $i = j$ and $0$ otherwise.
		\par 
		We consider now the projection $\pi' : \mathcal{X}' \to \mathbb{D}$, $\mathcal{X}' := \mathbb{D} \times \comp^n$, onto the first coordinate and a holomorphically trivial line bundle $\mathcal{L}'$ (resp. vector bundle $\mathcal{E}'$) over $\mathcal{X}'$, endowed with a non-trivial metric $h^{\mathcal{L}'} = e^{- \psi}$ (resp. with a metric tensor $h^{\mathcal{E}'} = (h^0_{i, j})$, $i, j = 1, \ldots, r$).
		As it was done previously for the family $\pi : \mathcal{X} \to \mathbb{D}$, we denote by $X'_{\tau}$ the fibers of $\pi'$, by $L'_{\tau}$, $E'_{\tau}$,  $|\tau| < 1$, the restriction of $\mathcal{L}'$, $\mathcal{E}'$ to the fibers, etc.
		We define $B'_{\tau, k}(x, x) \in \enmr{\mathcal{E}'_x} \otimes \wedge^{2n} T^* X'_{\tau}$, $x \in X'_{\tau}$, analogously to $B_{\tau, k}(x, x)$, $x \in X_{\tau}$.
		\par 
		Remark that the family $\pi' : \mathcal{X}' \to \mathbb{D}$ is a submersion (in particular, the fibers are smooth) and the Hermitian line bundles $(\mathcal{L}', h^{\mathcal{L}'})$, $(\mathcal{E}', h^{\mathcal{E}'})$ are smooth.
		Also there is $C > 0$, such that
		\begin{equation}\label{eq_curv_unif_bnd}
			\imun R^E_{\tau} \geq - C {\rm{Id}}_{E_{\tau}} \cdot c_1(L_{\tau}, h^L_{\tau}), 
			\qquad
			\imun R^{E'}_{\tau} \geq - C {\rm{Id}}_{E'_{\tau}} \cdot c_1(L'_{\tau}, h^{L'}_{\tau}),
		\end{equation}	
		for any $0 < |\tau| < \frac{1}{2}$.
		The latter assertion holds true due to the fact that the metrics on $E_{\tau}, E'_{\tau}$ (resp. $L_{\tau}, L'_{\tau}$) are constructed by the restriction of the metric on the global manifold $\mathcal{X}$, and the metrics on the line bundles $\mathcal{L}, \mathcal{L}'$ are fiberwise positive.
		Due to (\ref{eq_curv_unif_bnd}), we can apply \cite[Theorem 6.1.1]{MaHol}, which says that there are smooth functions $a'_i \in \ccal^{\infty}(\mathcal{X}', \enmr{\mathcal{E}'} \otimes \wedge^{2n} T^* \mathcal{X}'/\mathbb{D})$, $i \in \nat$, such that for any $j \in \nat$, $x \in \mathcal{X}'$, there are $C > 0$, $k_0 \in \nat$, such that for any $k \geq k_0$, we have
		\begin{equation}\label{eq_toepl_prim}
			\Big|
				\frac{1}{k^n} B'_{\tau, k}(x, x) - \sum_{i = 0}^{j} \frac{a'_i(x)}{k^i}
			\Big|
			\leq
			\frac{C}{k^{j + 1}}.
		\end{equation}
		Moreover, $a_0(x) = {\rm{Id}}_{E_{\tau}} \cdot c_1(L_{\tau}, h^L_{\tau})^n$, for any $x \in \mathcal{X}'$.
		Remark that \cite[Theorem 6.1.1]{MaHol} is stated for a single manifold and not in the family setting, as we apply it here.
		Its proof, however, adapts well to the setting when the metric tensors and complex structures depend smoothly on the parameter and the bound (\ref{eq_curv_unif_bnd}) holds uniformly, see \cite[Theorem 4.1.1]{MaHol}. Hence, the constants $C > 0$, $k_0 \in \nat$ can be chosen uniformly for $x$ varying over compact subsets of $\mathcal{X}'$
		\par 
		By Lemma \ref{lem_spec_gap} and (\ref{eq_curv_unif_bnd}), there are $k_0 \in \nat$, $\mu > 0$, such that for any $k \geq k_0$, $0 < |\tau| < \frac{1}{2}$, we have
		\begin{equation}\label{eq_spec_gap_fammm}
			{\rm{Spec}}(\square_{k, \tau}) \subset \{0\} \cup [\mu k, +\infty[, 
			\qquad
			{\rm{Spec}}(\square'_{k, \tau}) \subset \{0\} \cup [\mu k, +\infty[.
		\end{equation}
		\par 
		We argue that from (\ref{eq_spec_gap_fammm}), Bergman kernels have localization property in the sense of \cite[Proposition 4.1.6]{MaHol}.
		By this we mean that for any $l \in \nat$, there are $C > 0$,  $k_0 \in \nat$, such that for any $0 < |\tau| < \frac{1}{2}$, $k \geq k_0$, $x \in U \cap X_{\tau}$, verifying $h(x) = (\tau, x') \in \mathbb{D} \times \mathbb{D}_n(\frac{1}{4})$, we have
		\begin{equation}\label{eq_toepl_prim2}
			\Big|
				B_{\tau, k}(x, x)
				-
				B'_{\tau, k}(x', x')
			\Big|
			\leq
			\frac{C}{k^l},
		\end{equation}
		where the absolute value is taken with respect to the metric induced by $h^E_{\tau}$ and $c_1(L_{\tau}, h^L_{\tau})$ (which are identified with $h^{E'}_{\tau}$ and $c_1(L'_{\tau}, h^{L'}_{\tau})$ by means of $h$).
		Once this is done, from (\ref{eq_toepl_prim}) and (\ref{eq_toepl_prim2}), the proof of Theorem \ref{lem_toepl_2} would follow.
		\par 
		To establish (\ref{eq_toepl_prim2}), we follow the argument from \cite[Proposition 4.1.6]{MaHol}.
		Let $\epsilon > 0$ be such that for any $0 < |\tau| < \frac{1}{2}$, the geodesic ball (with respect to the metric associated with the Kähler form $c_1(\mathcal{L}, h^{\mathcal{L}})$) around $h^{-1}(\tau, \mathbb{D}_n(\frac{1}{2}))$ lies in the domain of definition of $h$.
		Let $\rho : \real \to [0, 1]$, be a symmetric smooth bump function such that $\rho(v) = 1$ for $|v| \leq \frac{\epsilon}{4}$ and $\rho(v) = 0$ for $|v| \geq \frac{\epsilon}{2}$.
		We define the function $F: \real \to \real$ as
		\begin{equation}
			F(a) = \Big(\int_{- \infty}^{+ \infty}  \rho(v) dv \Big)^{-1} \cdot \int_{- \infty}^{+ \infty} e^{\imun v a} \rho(v) dv.
		\end{equation}
		For $k \in \nat^*$, we set $\phi_k(a) := F(a) \cdot 1_{[\sqrt{\mu k}, + \infty[}$, where $1_{[\sqrt{\mu k}, + \infty[}$ is the indicator function.
		From (\ref{eq_spec_gap_fammm}), the following identities holds
		\begin{equation}\label{eq_phik_decomp}
			\phi_k(\sqrt{\square_{\tau, k}}) + B_{\tau, k}
			=
			F(\sqrt{\square_{\tau, k}}),
			\qquad
			\phi_k(\sqrt{\square'_{\tau, k}}) + B'_{\tau, k}
			=
			F(\sqrt{\square'_{\tau, k}}),
		\end{equation}
		\par By smoothness of $\rho$, for any $m \in \nat$, there is $C > 0$, such that
		$
			\sup_{a \in \real} |a|^m F(a) \leq C
		$.
		Hence, for any $l, m \in \nat$, there are $C > 0$,  $k_0 \in \nat$, such that for any $k \geq k_0$, we have
		$
			\sup_{a \in \real} |a|^m |\phi_k(a)| \leq C / k^l
		$.
		In particular, for any $l, m, r \in \nat$, there are $C > 0$,  $k_0 \in \nat$, such that for any $k \geq k_0$, $0 < |\tau| < \frac{1}{2}$, $s \in \mathscr{C}^{\infty}(X_{\tau}, L^k_{\tau} \otimes E_{\tau} \otimes K_{X_{\tau}})$, $s' \in \mathscr{C}^{\infty}(X'_{\tau}, (L'_{\tau})^k \otimes E'_{\tau} \otimes K_{X'_{\tau}})$, we have
		\begin{equation}\label{eq_local_spec_pt}
		\begin{aligned}
			&\Big\| 
				\square_{\tau, k}^m \phi_k(\sqrt{\square_{\tau, k}}) \square_{\tau, k}^r s
			\Big\|_{L^2}
			\leq
			\frac{C}{k^l} \| s \|_{L^2},
			\\
			&\Big\| 
				(\square'_{\tau, k})^m \phi_k(\sqrt{\square'_{\tau, k}}) (\square'_{\tau, k})^r s'
			\Big\|_{L^2}
			\leq
			\frac{C}{k^l} \| s' \|_{L^2}.
		\end{aligned}
		\end{equation}
		\par 
		Since our families are smooth in the neighborhood of $x_0$ and $x'_0 = h(x_0)$, the uniform version of Sobolev and elliptic estimates in the form \cite[Lemma 1.6.2]{MaHol} hold.
		From this and (\ref{eq_local_spec_pt}), by repeating the argument of the proof of \cite[Proposition 4.1.5 after (4.1.7)]{MaHol}, we see that for any $l \in \nat$, there are $C > 0$,  $k_0 \in \nat$, such that for any $k \geq k_0$, $0 < |\tau| < \frac{1}{2}$, the Schwartz kernels $\phi_k(\sqrt{\square_{\tau, k}})(x, x)$ of $\phi_k(\sqrt{\square_{\tau, k}})$ and $\phi_k(\sqrt{\square'_{\tau, k}})(x', x')$ of $\phi_k(\sqrt{\square'_{\tau, k}})$ for any $x \in U$, $x' \in h(U)$, verify
		\begin{equation}\label{eq_local_spec_pt2}
			\Big| 
			\phi_k(\sqrt{\square_{\tau, k}})(x, x)
			\Big|
			\leq
			\frac{C}{k^l},
			\qquad
			\Big| 
			\phi_k(\sqrt{\square'_{\tau, k}})(x', x')
			\Big|
			\leq
			\frac{C}{k^l}.
		\end{equation}
		\par 
		Now, by the finite propagation speed of solutions of hyperbolic equations, cf. \cite[Theorem D.2.1 and proof of Proposition 4.1.5]{MaHol}, we see that the Schwartz kernel $F(\sqrt{\square_{\tau, k}})(x, x)$ (resp. $F(\sqrt{\square'_{\tau, k}})(x', x')$) of $F(\sqrt{\square_{\tau, k}})$ (resp. $F(\sqrt{\square'_{\tau, k}})$) depends only on the geometry of $(\pi: \mathcal{X} \to \mathbb{D}, \mathcal{L}, h^{\mathcal{L}}, \mathcal{E}, h^{\mathcal{E}})$ (resp. $(\pi': \mathcal{X}' \to \mathbb{D}, \mathcal{L}', h^{\mathcal{L}'}, \mathcal{E}', h^{\mathcal{E}'})$) in $\epsilon$-neighborhood of $x$ (resp. $x'$).
		But since $h$ is a local isometry between $(\pi: \mathcal{X} \to \mathbb{D}, \mathcal{L}, h^{\mathcal{L}}, \mathcal{E}, h^{\mathcal{E}})$ and $(\pi': \mathcal{X}' \to \mathbb{D}, \mathcal{L}', h^{\mathcal{L}'}, \mathcal{E}', h^{\mathcal{E}'})$, by our choice of $\epsilon > 0$, for any $x \in U \setminus X_0$, $h(x) = (\tau, x') \in \mathbb{D} \times \mathbb{D}_n(\frac{1}{4})$, we have
		\begin{equation}\label{eq_local_pt1}
			F(\sqrt{\square_{\tau, k}})(x, x)
			=
			F(\sqrt{\square'_{\tau, k}})(x', x').
		\end{equation}
		From (\ref{eq_phik_decomp}), (\ref{eq_local_spec_pt2}) and (\ref{eq_local_pt1}), we deduce (\ref{eq_toepl_prim2}).
	\end{proof}

\bibliography{bibliography}

		\bibliographystyle{abbrv}

\Addresses

\end{document}